\documentclass[reqno]{amsart}
\usepackage[parfill]{parskip} 
\usepackage{graphicx}
\usepackage{mathabx}
\usepackage{amssymb}
\usepackage{epstopdf}
\usepackage{stmaryrd}
\usepackage{pgf,tikz}
\usetikzlibrary{positioning, arrows.meta}
\usepackage{esint}
\usepackage{mathrsfs} 
\usepackage[english]{babel}
\usepackage{amssymb,
enumerate}
\usepackage{color}
\usepackage[latin1]{inputenc}
\usepackage[T1]{fontenc}
\usepackage{amsfonts}
\usepackage{amssymb}
\usepackage{amsmath}
\usepackage{amsthm}
\usepackage{mathtools}
\usepackage{graphicx,xcolor}
\usepackage{afterpage}
\usepackage[colorlinks=true, linkcolor=red, citecolor=cyan, urlcolor=green]{hyperref} 
\usepackage{bbm}
\usepackage{latexsym}
\usepackage[foot]{amsaddr}

\newtheorem{Rem}[]{Remark}

\def\be{\begin{equation}}
\def\ee{\end{equation}}
\def\bea{\begin{eqnarray}}
\def\eea{\end{eqnarray}}
\def\beal{\begin{align}}
\def\eeal{\end{align}}

\def\nn{\nonumber}

\newcommand{\var}{\mathrm{Var}}
\newcommand{\OOO}[1]{O \left(#1\right)}

\newcommand{\op}{\textup{op}}

\newcommand{\dd}{\mathrm{d}}

\newcommand{\EE}{{\mathbb{E}}}
\newcommand{\PP}{{\mathbb{P}}}
\renewcommand{\P}{{\mathbb{P}}}
\newcommand{\R}{{\mathbb{R}}}

\newcommand{\NN}{{\mathbb{N}}}

\def\l{\lambda}

\def\nn{\nonumber}

\newcommand{\Prob}{{\mathrm{prob}}}

\newcommand{\ia}{_{i \rightarrow a}}
\newcommand{\ib}{_{i \rightarrow b}}
\newcommand{\ja}{_{j \rightarrow a}}
\newcommand{\jb}{_{j \rightarrow b}}
\newcommand{\ka}{_{k \rightarrow a}}
\newcommand{\ai}{_{a \rightarrow i}}

\newcommand{\bi}{_{b \rightarrow i}}
\newcommand{\bj}{_{b \rightarrow j}}

\newcommand{\ti}{^{(t)}}
\newcommand{\zzero}{^{(0)}}
\newcommand{\oone}{^{(1)}}
\newcommand{\ttwo}{^{(2)}}

\newcommand{\tti}{^{(t+1)}}
\newcommand{\Bal}{\mathcal{B}_a^\lambda}
\newcommand{\Jil}{\mathcal{J}_i^\lambda}

\newcommand{\err}{\mathtt{e}}
\newcommand{\ERR}{\mathtt{E}}

\def\R{\mathbb{R}}
\def\s{\sigma}
\def\a{\alpha}
\def\e{\varepsilon}
\def\d{\delta}
\def\g{\gamma}

\def\b{\beta}
\def\D{\Delta}
\def\r{\rho}
\def\t{\tau}

\def\sign{\operatorname{sign}}

\def\N{\mathbb{N}}

\newcommand{\edg}{\mathsf{EDG}}

\newcommand{\meanv}[1]{\left\langle#1\right\rangle}

\newcommand{\parav}[1]{\left(#1\right)}
\newcommand{\squarev}[1]{\left[#1\right]}

\newcommand{\absv}[1]{\left|#1\right|}

\newcommand{\nocontentsline}[3]{}
\newcommand{\tocless}[2]{\bgroup\let\addcontentsline=\nocontentsline#1{#2}\egroup}

\DeclareMathSymbol{\leqslant}{\mathalpha}{AMSa}{"36} 
\DeclareMathSymbol{\geqslant}{\mathalpha}{AMSa}{"3E}
\DeclareMathSymbol{\eset}{\mathalpha}{AMSb}{"3F}     
\renewcommand{\leq}{\;\leqslant\;}                  
\renewcommand{\geq}{\;\geqslant\;}

\newcommand{\ind}[1]{1_{\{#1\}}}

\newcommand{\E}[1]{E\left[#1\right]}

\def\a{\alpha}
\def\e{\varepsilon}
\def\d{\delta}
\def\g{\gamma}

\def\b{\beta}
\def\D{\Delta}
\def\l{\lambda}
\def\r{\rho}
\def\s{\sigma}
\def\t{\tau}

\def\R{\mathbb{R}}
\def\C{\mathbb{C}}
\def\E{E}

\theoremstyle{plain}

\newtheorem{remark}{Remark}[section]

\newtheorem{theorem}{Theorem}[section]

\newtheorem{lemma}[theorem]{Lemma}

\newtheorem{proposition}[theorem]{Proposition}

\newtheorem{corollary}[theorem]{Corollary}

\numberwithin{equation}{section}

\title[From belief propagation to the AMP equations]
{Derivation of the AMP equations from belief propagation for the $\ell_2$ minimisation problem}

\author{Giuseppe Genovese$^{1,2}$}
\address{$^1$Mathematisches Institut, Albert-Ludwigs Universit\"at Freiburg, Ernst-Zermelo-Strasse 1, 
D-79104 Freiburg, Germany.}
\address{$^2$Department of Mathematics, University of British Columbia, 1984 Mathematics Road V6T 1Z2 Vancouver BC, Canada.} 
\author{Arianna Piana$^3$}
\address{$^3$Department of Mathematics, Weizmann Institute of Science, Rehovot 76100, Israel.}
\date{\today}                     

\begin{document}
\maketitle

\begin{abstract}
We consider the $\ell_p$-minimisation, which consists of finding the vector $x\in\R^N$ which minimises $\|x\|_p$ subject to the linear constraint $y=Ax$, where $y\in\R^m$ is given and $A$ is a $m\times N$ random matrix with i.i.d. sub-Gaussian centred entries ($m<N$). This can be viewed as the zero temperature version of a statistical mechanics problem, in which one introduces a suitable Gibbs measure on $\R^N$. To such a Gibbs measure there are associated belief propagation equations. 

We prove in the easiest case $p=2$ that the means of the distributions obtained by the belief propagation iteration satisfy asymptotically the approximate message passing equations. 
\vspace{0.5cm}
\end{abstract}

\section{Introduction}

The aim of this paper is providing a rigorous derivation of the approximate message passing (AMP) equations starting from the belief propagation (BP) equations. BP consists of a set of recursive equations for the marginal distributions of a target probability distribution \cite{perla}. A brief description of the method in our case of interest is given below. The reader can look for instance at \cite{mezard} for a comprehensive discussion. 
We will take the BP equations as a starting point of our considerations. The characterisation of the convergence of the BP algorithm is a prominent open problem \cite{coja}, \cite{k} that will not concern us here. 

The AMP equations are recursive equations for the means of a target distribution. In the framework of statistical physics, they are known as TAP equations, after the seminal paper of Thouless, Anderson and Palmer \cite{tap}. The subject found renewed interest in the mathematical community mainly after Bolthausen devised a way to prove convergence of the iteration in the whole high temperature region \cite{bolt}. Bolthausen's method later was widely generalised by Montanari and coauthors, see for instance \cite{bayati2},
\cite{tantoarkg4}, \cite{tantoarkg5}, \cite{tantoarkg6}, \cite{far}, \cite{celentano}.

From the point of view of statistical physics, a fundamental problem related to the AMP equations is their derivation from the model of interest. More precisely, the problem is to show that the solution of the AMP equations describes the actual means under the Gibbs measure under investigation. In the context of the Sherrington-Kirkpatrick model (and its $p$-spin versions), the question has been studied in \cite{talabook}, \cite{chat}, \cite{yau}, \cite{chen2}, \cite{chen}. These results are mainly limited to the so-called high temperature phase (however, for a low temperature result, see \cite{tuca}). 

A different path was taken in the contributions \cite{DMM1}, \cite{DMM2}, \cite{bayati1} and \cite{mez}, where the AMP equations were not derived directly from the model distribution, but rather from the BP iteration, regardless of its convergence. This was done for model distributions related to compressed sensing (see also the discussion below) and for the Hopfield model, but such derivations are not mathematically rigorous. The present paper fits this line of research, providing, to the best of our knowledge, the first fully rigorous derivation of the AMP equations from the BP iteration. We note however that one crucial step of our derivation, namely the Gaussian approximation of the densities in the BP iterations, has been studied in the previous works \cite[Sect. 4]{sly}, \cite[Sect. 7]{mossel} and \cite[Appendix D]{celentano}). The common denominator of these articles seems to be the underlying tree structure of the models under consideration. This is a central difference with the present paper, as we take a purely analytical viewpoint and we do not investigate the structure of the factor graph associated to the BP algorithm (more details below). 

The model we are able to handle is in any case very simple, as we focus on the quadratic minimisation. Here many of the technical difficulties are already present, but the quadratic (i.e. Gaussian) structure of the problem helps greatly their resolution (more details are given in the following paragraphs). In this sense, our contribution is purely methodological and must be regarded as a first step towards the understanding of the link between the BP and AMP iteration. 

\subsection{Set up}
In this paper we will work with a random matrix $A\in \R^{m\times N}$ with $m/N\simeq \d\in(0,1)$, whose entries are i.i.d. sub-Gaussian r.vs with $E[A_{11}]=0$, $E[A^2_{11}]=\frac1m$. We assume two properties on the law of $A$: firstly, that the one-dimensional law of each entry $A_{ai}$ has no atom at $0$, so $\PP(A_{ai}=0)=0$; secondly, that the distribution of $A$ to satisfy the following hypercontractive inequality: for all multilinear forms of order $\ell$ in the entries of $A$ $\Psi_\ell$, there is a constant $C_\ell$ such that
\be\label{eq:hiper-intro}
\|\Psi_\ell\|_{L^p}\leq C_\ell p^{\frac{\ell}{2}}\|\Psi_\ell\|_{L^2}\,. 
\ee

Given an outcome vector $y\in \R^m$, a classical problem in statistics amounts to recover the vector $x\in\R^N$ such that $y=Ax$ under some non-linear constraint. We will be interested in the solutions of an $\ell_p$-minimisation 
\be\label{eq:basispursuit}
\mbox{minimise $\|x\|_p$ subject to $y=Ax$}\,
\ee
where $\|x\|_p\coloneqq \parav{\sum_{i=1}^N|x_i|^p}^{\frac1p}$. 
This problem is particularly interesting for the theory of compressive sensing, where the case $p=1$ is the most relevant \cite{DMM1}, \cite{DMM3}. It is well-known indeed that minimisation of any $\ell_p$ norm with $p>1$ under the same linear constraint will not give sparse solutions, while any $p\in[0,1)$ would do, but the recovery is NP-hard \cite{CS}. On the other hand, the simplest case is $p=2$:\be\label{eq:basispursuitp=2}
\mbox{minimise $\|x\|_2$ subject to $y=Ax$}\,.
\ee
The solution of \eqref{eq:basispursuitp=2} is explicitly given by $x^\star={A}^T(A{A}^T)^{-1}y$.

Let us introduce BP for the problem of our interest. 
We define the probability measure
\be\label{eq:mu}
\mu^{(p)}_{N,\b}(dx)=Z^{-1}_{N,\b}e^{-\b \|x\|_p^p}1_{\{x\,:\, y=Ax\}}\,.
\ee
Here $\b>0$ is a parameter and $Z_{N,\b}$ a normalisation factor. One expects that as $\b\to\infty$ the measure $\mu_{N,\b}$ gets closer and closer to the uniform measure on the solutions of \eqref{eq:basispursuit}. Instead of dealing with the measure $\mu^{(p)}_{N,\b}$, we look at the BP iteration, defined as follows.
For each $t\in\N$ and $i\in[N]$, $a\in[m]$, we defined the continuous and integrable functions $\nu^{(t)}\ia ,\hat \nu^{(t)}\ai $ according to the belief propagation equations
\begin{align}\label{eq:BP}
\hat \nu\ti\bi (x_i)&\simeq \int_{y-A^{[i]}x=\{A_{ai}x_i\}_{a\in[m]}}dx^{[i]}\prod_{j\neq i}\nu\ti\ja (x_j)\,,\\
\nu^{(t+1)}\ia (x_i)&\simeq \pi_{\b,p}(x_i)\prod_{b\neq a} \hat \nu\ti\bi (x_i)\,.\nn
\end{align}
Here $dx^{[i]}:=\prod_{j\neq i}dx_j$, $A^{[i]}\in\R^{m\times {N-1}}$ denotes the matrix obtained by $A$ erasing the $i$-th column and $x^{[i]}\in \R^{N-1}$ the vector obtained by $x\in\R^N$ erasing the $i$-th component; the symbol $\simeq$ here indicates identity modulo a normalisation constant.
In addition, we defined the $\ell_p$ prior by
\be
\pi_{\b,p}(s)\coloneqq Z'(\b,p)e^{-\b|s|^p}\,,\qquad \b>0,\quad s>0\,
\ee
where $Z'(\b,p)$ is a normalisation constant. For simplicity we choose the initial condition $\nu^{(0)}\ia (x_i)$ to be a centred distribution with variance independent on $(i,a)$. 

\begin{figure}[!h]
    \centering
    \begin{tikzpicture}
    [
    neuron/.style={circle, draw=black,  fill=black ,thick, minimum size=3mm, prefix after command= {\pgfextra{\tikzset{every label/.style={black}}}}},
    lambda/.style={rectangle, draw=black, fill=white, thick, minimum size=3mm, prefix after command= {\pgfextra{\tikzset{every label/.style={black}}}}}
    ]
        \node [lambda, label={$1$}](n1) [] {};
        \node [lambda](n2) [right of= n1, xshift=0.5cm] {};
        \node [lambda, label={$a$}](n3) [right of= n2, xshift=0.5cm] {};
        \node [lambda](n4) [right of= n3, xshift=0.5cm] {};
        \node [lambda, label={$m$}](n5) [right of= n4, xshift=0.5cm] {};
        \node [neuron, label=below:{$1$}](l1) [below of= n1, xshift=-1cm, yshift=-1.5cm] {};
        \node [neuron](l2) [right of= l1, xshift=0.5cm] {};
        \node [neuron, label=below:{$i$}](l3) [right of= l2, xshift=0.5cm] {};
        \node [neuron](l4) [right of= l3, xshift=0.5cm] {};
        \node [neuron](l5) [right of= l4, xshift=0.5cm] {};
        \node [neuron, label=below:{$N$}](l6) [right of= l5, xshift=0.5cm] {};
        \draw[-] (n1) -- (l1); \draw[-] (n1) -- (l2); \draw[-] (n1) -- (l3); \draw[-] (n1) -- (l4); \draw[-] (n1) -- (l5); \draw[-] (n1) -- (l6);
        \draw[-] (n2) -- (l1); \draw[-] (n2) -- (l2); \draw[-] (n2) -- (l3); \draw[-] (n2) -- (l4); \draw[-] (n2) -- (l5); \draw[-] (n2) -- (l6);
        \draw[-] (n3) -- (l1); \draw[-] (n3) -- (l2); \draw[-] (n3) -- (l3); \draw[-] (n3) -- (l4); \draw[-] (n3) -- (l6);
        \draw[-] (n4) -- (l1); \draw[-] (n4) -- (l2); \draw[-] (n4) -- (l3); \draw[-] (n4) -- (l4); \draw[-] (n4) -- (l6);
        \draw[-] (n5) -- (l1); \draw[-] (n5) -- (l2); \draw[-] (n5) -- (l3); \draw[-] (n5) -- (l4); \draw[-] (n5) -- (l6);
    \end{tikzpicture}
    \caption{Complete bipartite factor graph for the BP iteration. The squares are the factor nodes, the dots are variable nodes.}
    \label{fig:factor_graph_CS}
\end{figure}
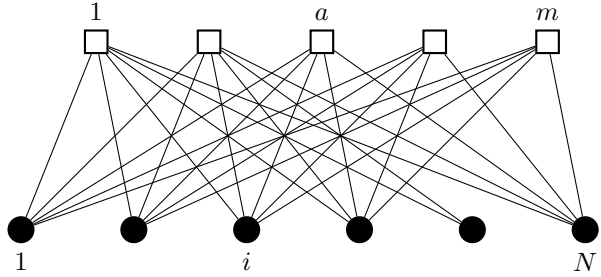

The usual interpretation of the BP equations assumes the presence of an underlying factor graph, that is a complete bipartite graph with edges linking only each factor node, labeled by $a\in[m]$, with all the variable nodes, labeled by $i\in[N]$. The marginals $\nu^{(t)}\ia$ are beliefs that propagate from the variable to the factor nodes and $\hat\nu^{(t)}\ai$ are beliefs that propagate from the factor to the variable nodes. If the factor graph is a tree it is not hard to show that the BP equations reconstruct the measure $\mu^{(p)}_{N,\b}$ exactly \cite{mezard}. For our proposes, the distributions $\hat\nu^{(t)}\ai$ can be thought just as ancillary objects in order to compute the $\nu^{(t)}\ia$. 

We take the following route. We shall derive a set of recursive equations for the means and the variances of the marginals $\nu^{(t)}\ia$ that are exact in the limit $N,m\to\infty$ (recall that $m\simeq\d N$). We will show that the variances are independent of the labels $(i,a)$, while the means are asymptotically independent of the factor labels. This asymptotic recurrence for the means are the AMP equations. Finally, we verify that the solution of the AMP equations for $\b$ large enough coincides with the true solution $x^\star$ of the problem \eqref{eq:basispursuitp=2}.


\subsection{Main Result}

We introduce mean and variance of the BP iterates
\bea
x\ia\ti\coloneqq \int s\nu\ti\ia(s)ds\,,&\qquad&v\ia\ti\coloneqq \int s^2\nu\ti\ia(s)ds-\parav{\int s\nu\ti\ia(s)ds}^2\,,\label{eq:xandv}\\
\hat x\ai\ti\coloneqq \int s\hat\nu\ti\ai(s)ds\,,&\qquad&
\hat v\ai\ti\coloneqq \int s^2 \hat \nu\ti\ai(s)-\parav{\int s\hat\nu\ti\ai(s)ds}^2\label{eq:xandv-HAT}\,.
\eea
We will also need 
\be\label{eq:rhovarsigma}
\r\ti\ja \coloneqq \int s^3\nu\ti_{ia}(s)ds\,.
\ee
It is not hard to show that (see the subsequent Lemma \ref{lemma:facilemaleki})
\bea
\hat{x}\ai \ti&=&x\ia \ti+\frac{y_a-\sum_{j=1}^NA_{aj}x\ti\ja }{A_{ai}}\label{eq:xx-hat}\\
\hat{v}\ai \ti& =& \frac{\sum_{k \neq i} A_{ak}^2 v_{\ka}\ti}{A_{ai}^2}\,\label{eq:v_hat}
\eea
(note that here dividing by the elements of $A$ is armless, due to our assumptions of the distribution of the entries). 
Much harder is to relate $\hat{x}\ai \ti, \hat{v}\ai \ti$ to $ x\ia ^{(t+1)}, v\ia ^{(t+1)}$, which will take most of our efforts. 
Our main result is the following. 

\begin{theorem}\label{TH:main}
Let $m<N\in\N$ and $\d:=m/N$. Let $A\in \R^{m\times N}$ be a random matrix with i.i.d. sub-Gaussian symmetric entries with $\mathbb E[A_{12}^2]=1/m$, $\mathbb P(A_{12}=0)=0$ and satisfying the condition (\ref{eq:hiper-intro}). Let also $y\in\R^m$ with $\max_{b\in[m]}|y_b|\leq1$. 

Consider the BP equations \eqref{eq:BP} with $p=2$ and initial conditions $\nu^{(0)}_{i\to a}$, that are bounded and continuous densities, with finite fourth moment and $x\ia\zzero=0$, $v\ia\zzero=v^{(0)}$. Set
\be\label{eq:defTH-bar-v^t}
v^{(t)}:=\frac{v\zzero(1-\d)}{\d^t(1-\d)+2\b v\zzero(1-\d^t)}\,. 
\ee
Recall \eqref{eq:xandv}. Then, for all $(i,a)\in[N]\times[m]$ we have
\be\label{eq:THvarianza}
v\ia^{(t)} \;=\; v^{(t)} \;+\; O_\P \parav{\frac1{\sqrt N}}\,.
\ee
Moreover, let $\Delta\ti \coloneqq \delta/(2\beta v\ti +\delta)$ and
$\Gamma_\lambda\ti =\prod_{\tau=1}^{\lambda}\Delta^{(t-\tau)}$.
It holds that
\be\label{eq:xAMP}
x\ia\tti=X_i\tti+O_\mathbb{P}(t\Delta\ti\Gamma_1\ti N^{-1/2})\qquad \forall a\in[m]\,,
\ee
where the sequence $\{X_i\ti\}_{t\in\N}$ satisfies the AMP equations
\be\label{eq:AMP-cora}
X_i\tti= X_i\ti
+\Delta\ti  \left(\sum_{b=1}^m y_bA_{bi}-\sum_{b=1}^m\sum_{j=1}^N A_{bi}A_{bj} X_j\ti \right)+O_\PP \Big(\Delta\ti  \Gamma\ti _{1} N^{-1/2}\Big)\,.
\ee
\end{theorem}
Here $O_\P(N^{-\frac12})$ denotes a quantity that is smaller than $N^{-\frac12}$ with probability larger than $1-e^{-N^\e}$ for some $\e>0$. 
\begin{proof}
The fact that the variance remains constant along the iteration, that is (\ref{eq:THvarianza}), is proved in Proposition \ref{prop:shadow}. Moreover, the AMP equations (\ref{eq:xAMP}), (\ref{eq:AMP-cora}) are obtained combining Proposition \ref{cor:finitebeta-component} with Proposition \ref{prop:shadow}.
\end{proof}

As a corollary, we establish that the solution of the AMP equations \eqref{eq:AMP-cora} in the limit $N\to\infty$ and $\b$ large converges as $t\to\infty$ to the solution $x^\star$ of the $\ell_2$ minimisation problem that can be achieved by classical methods. This can be thought as a sanity check of the entire method, as we do not know a priori that the BP equations reconstruct the true distribution $\mu^{(2)}_{N,\b}$ of (\ref{eq:mu}). 

\begin{corollary}\label{cor:main}
Same assumptions of Theorem \ref{TH:main}. Let $x^\star={A}^T(A{A}^T)^{-1}y$. 
If $\b\geq (1-\d)/2v\zzero$, we have
\be\label{eq:cormain}
\lim_{t\to\infty} \lim_{N\to\infty}\ x\ia \ti   =  x_i^\star
\qquad\text{for every }i\in[N], \,\,a\in[m].
\ee
\end{corollary}


\subsection{Oveview of the proof}

Next, we give a high level exposition of the main ideas of our proof. 
Looking at the BP equations (\ref{eq:BP}) one realises easily that the $\hat \nu\ti\ai$ are the law of the weighted sums (with the elements of the matrix $A$) of independent random variables distributed according to the $\nu\ia^{(t)}$. Therefore, as remarked in \cite{DMM2}, it can be computed by arguments based on the central limit theorem (CLT). In practice, we have that
\be\label{eq:heu-hat}
\hat \nu\ti\ai(s)\sim e^{-\frac{(s-\hat x\ai\ti)^2}{2\hat v\ai\ti}}\,.
\ee
The second BP equation takes the products of these approximated Gaussian densities and it multiplies the result by the $\ell_p$-prior. The product of Gaussian densities is a Gaussian density (apart from unimportant global factors), whose mean and the variance $\mu\ia\tti$ and $\s\ia\tti$ however can be large, namely of order $\sqrt N$. Ignoring for the moment this crucial issue, we continue and write then
\be\label{eq:heu-nu}
\nu\tti\ia(s)\sim e^{-\b|s|^p+\frac{(s-\mu\ia\tti)^2}{2 \s\ia\tti}}\,.
\ee
In principle this distribution is quite nice and allows us to perform also some explicit computation. Therefore we could keep going, alternating the CLT to compute \eqref{eq:heu-hat} and the product formula for \eqref{eq:heu-hat}. This is indeed the strategy suggested in \cite{DMM2}.

Such a route is however not so easy to implement, as several sources of difficulties arise. First of all, the errors coming from the CLT approximations need to be carefully estimated. We remark that the usual order of magnitude $N^{-\frac12}$ of the reminder coming from \eqref{eq:heu-hat} is not sufficient to our ends, as in the next BP step we have to take the $m$-fold product of these error terms, which can be large (as $m$ is proportional to $N$). To overcome this issue, we employ an Edgeworth expansion (see \cite{bobkov} for recent developments) to estimate the density. If on one hand the reminder of the Edgeworth expansion remains small after taking the product, on the other hand we need to deal with the extra terms of the expansion. Estimating them is computationally expensive, but feasible for the first few iterates (we do it only for $t=1$). However, it becomes immediately too complicated as $t$ grows, due to the complex dependencies on the elements of the matrix $A$. To fix the ideas, assume that the initial distributions $\nu\zzero\ia(s)$ are centred at some given values $x\zzero\ja$ (we will set $x\zzero\ja=0$ in the rest of the paper). For $t=0$ then, the Hoeffding inequality immediately gives
\be\label{eq:mepiacesse}
\sum_{j\neq i}A_{aj}x\ti\ja=O_\P(1)\,
\ee
(which in turn implies $\hat x\zzero\ai=O_\P(\sqrt N)$). To extend \eqref{eq:mepiacesse} to all $t\geq1$ one should track down the exact dependency of $x\ti\ja$ on the elements of $A$, as a direct inspection gives that the independence of $x\ti\ja$ on $A_{aj}$ is lost for $t\geq 3$. This looks prohibitive and we could find an alternative strategy only for $p=2$. 

The choice $p=2$ gives many advantages. One of them is certainly the possibility to analyse the special case of Gaussian initial distributions. In this case we do not need the CLT at any iteration step: the distributions remain always Gaussian and we need only to control the means and variances. In particular, we verify directly that with this initialisation the messages $x\ti\ja$ are $O_\P(1)$ for every $t\in\N$ (for which we need the hypercontractivity assumption on the distribution of $A$). We stress that this is a peculiarity of the $p=2$ case, as otherwise already the distribution $\nu\oone\ia(s)$ is of the form (\ref{eq:heu-nu}), thus not a Gaussian, and then a CLT argument is needed to compute $\hat \nu\oone\ai$ via (\ref{eq:heu-hat}). This is not incidental, but functional for the rest of the proof. Indeed in the general setting the sequence of means and variances can be derived perturbatively from the one obtained with Gaussian initialisation. Here the perturbative regime is given only by the CLT: for any initialisation the first iteration produces densities close to be Gaussian, so that for all successive iterations the BP updates always remain in the Gaussian basin of attraction. 
In practice, this means that we use a different version of the local CLT for all time steps \(t \ge 1\), exploiting the fact that the iteration of CLT steps makes the random variables in the play closer and closer to be Gaussian distributed.
For any other $p\neq2$, the BP equations oscillate between the Gaussian and the densities given by (\ref{eq:heu-hat}), hence the strategy described above does not apply.

\subsection{Outline of the paper} The paper is organised as follows. The sections \ref{sect:exactGaussian}-\ref{sect:exactGaussian-bis} deal with the special case of Gaussian initialisation. In Section \ref{sect:exactGaussian} we do some preliminary computations and derive the message passing equations from the BP iteration (i.e. the equations for the means, still dependent on the factor nodes); in Section \ref{sect:proofProp} we show that the messages $x\ti\ja$ are $O_\P(1)$ for every $t\in\N$ by obtaining an explicit combinatorial formula for them and using the Wick Theorem to estimate it; in Section \ref{sect:exactGaussian-bis} we conclude, deriving the AMP equations. The sections \ref{sect:GaussianI}-\ref{sect:shadow} deal with the case of generic initialisations. In Section \ref{sect:GaussianI} and \ref{sect:successive} we prove that the expectation values w.r.t. $\nu\ti\ia$ are close to Gaussian expectation values. In Section \ref{sect:GaussianI} we handle the most challenging first step of the BP iteration, in which one passes from the arbitrary initial distributions to the nearly-Gaussian distributions $\nu\oone\ia$ after the CLT. The results presented in Section \ref{sect:GaussianI} hold for all $p$; in Section \ref{sect:successive} we use a different argument for all $t\geq1$ that exploits the improved proximity of the output of the CLT to a Gaussian distribution at every step; Section \ref{sect:shadow} is the core of this work, where all pieces come together. Here we prove that the sequence of means and variances of the BP iterates with arbitrary initialisation are close to those with Gaussian initialisation. We do it by introducing a third auxiliary sequence that shadows the first two. Finally in Section \ref{sect:AnaAMP} we analyse the AMP equations and prove Corollary \ref{cor:main}. This part is pretty standard, but we include it for sake of completeness. We attach four appendices, containing important auxiliary statements and proofs. In Appendix \ref{App:equiv} we prove that the $L^2$ norm of certain tensors, formed by product of matrix elements of $A$, which are of interest in Section \ref{sect:proofProp}, are bounded by the case in which $A$ has standard Gaussian entries (it is here that the hypercontractivity assumption on the law of the entries of $A$ is used); In Appendix \ref{app:LLT} we present the versions of the local CLT used in Section \ref{sect:GaussianI} and Section \ref{sect:successive}. Appendix \ref{app:tail} contains few tail estimates, useful when we need to bound sums of powers of the entries of the matrix $A$; finally in Appendix \ref{App:prooflemma} we prove a technical lemma stated in Section \ref{sect:GaussianI}. 

\subsection{Notations}

All probabilities $\PP(\cdot)$ and expectations $\EE[\cdot]$ are taken with respect to the randomness of the design matrix $A$ (and any other explicitly declared sources). Everywhere $C,c$ will denote absolute constants possibly varying from line to line.  

We set
\be\label{eq:def-Gauss}
\phi_{\mu,\s}(x)\coloneqq \frac{e^{-\frac{(x-\mu)^2}{2\s}}}{\sqrt{2\pi\s}}\,.
\ee

\smallskip
\noindent\emph{Asymptotic notation for random variables.}
We say that a centered  random variable $Z$ is $O_\PP(N^\alpha)$ if for all $\varepsilon>0$, $
\PP\bigl(|Z|>N^{\alpha+\varepsilon}\bigr)\;\lesssim\; \exp(-N^{2\varepsilon}),$
with $\var(Z)$ bounded uniformly in $N$.  
Unless noted otherwise, all these bounds are uniformly over the indices.  
Finally, we write $Z_N \approx Z_N'$ to indicate that $Z_N-Z_N' = O_\PP(N^{-1/2})$, uniformly over all relevant indices and for any fixed time horizon. Equivalently, there exist constants $c,C>0$ such that $
\PP\bigl(|Z_N-Z_N'|> C N^{-1/2}\bigr)\;\le\; C e^{-cN}.$

Note that if $X=O_\PP(N^{a})$ and $Y=O_\PP(N^{b})$ then $XY=O_\PP(N^{a+b})$. Indeed this follows from
$$
\P(XY\geq t)\leq \P(X\geq t/s, Y\leq s)+\P(X\geq t/Y, Y\geq s)\leq \P(X\geq t/s)+\P(Y\geq s)
$$
valid for any $t,s>0$. We choose $s=N^{b+\e}$ and $t=N^{a+b+2\e}$ and we are done. 

{\em Concentration for row/column norms.}
We will repeatedly use the following consequence of the Bernstein inequality:
\begin{equation}\label{eq:concentration}
\sum_{b\neq a}A_{bi}^2 \approx 1,
\qquad
\sum_{j\neq i}A_{aj}^2 \approx \frac{1}{\delta},
\end{equation}
uniformly over  $(a,i) $, with probability tending to one as  $N\to\infty $. These approximations will be used throughout the iterative computations.

\smallskip
\noindent\emph{Deterministic comparators.}
For nonnegative quantities $f,g$, $f\lesssim g$ means $f\le Cg$ for a numerical constant $C$ independent of $N$ (and of $\beta$, unless indicated); $f\asymp g$ abbreviates $f\lesssim g$ and $g\lesssim f$. Subscripts (e.g.\ $\lesssim_{\delta,t}$) record allowable parameter dependence.

\smallskip
\noindent\emph{Norms.}
For a real random variable $X$, $\|X\|_{L^p}=(\EE|X|^p)^{1/p}$. For $z\in\R^N$, $\|z\|_p=(\sum_i|z_i|^p)^{1/p}$ and $\|z\|_\infty=\max_i|z_i|$. $\|M\|_{\op}$ denotes the operator norm of a matrix $M$.

\subsection{Acknowledgments.} The authors thank David Belius and Ofer Zeitouni for their helpful comments and Andrea Montanari for pointing out reference \cite{celentano}. They are especially grateful to Erwin Bolthausen, who was involved in an initial stage of the project and gave crucial inputs to this work. G.G. was partially supported by the NSERC Discovery Grant RGPIN-2025-04930. The paper was partially written while G.G. was in residence at the Simons Laufer Mathematical Sciences Institute in Berkeley, California, supported by the NSF Grant No. DMS-1928930.
A.P. is supported by the Israeli Council for Higher Education (CHE) via the Weizmann Data Science Research Center, and by a research grant from the Estate of Harry Schutzman. 

\section{Exact computations with a Gaussian initial distribution I: MP equations}\label{sect:exactGaussian}

In the present and in the next two sections we assume that the initialisation $\nu^{(0)}\ia (x_i)$ of the BP equations 
\begin{align}
\hat \nu\ti\bi (x_i)&\approx \int_{y-A^{[i]}x=\{A_{ai}x_i\}_{a\in[m]}}      \dd x^{[i]}_1\ldots \dd x^{[i]}_N \prod_{j\neq i}\nu\ti\ja (x_j)\,,\quad t\geq0\label{eq:BP-sect2}\\
\nu^{(t+1)}\ia (x_i)&\approx \pi_{q,\b}(x_i)\prod_{b\neq a} \hat \nu\ti\bi (x_i)\,.\label{eq:BPII-sect2}
\end{align}
is given by a centred Gaussian distribution uniformly over the variable and factor nodes, namely
$\nu\ia^{(0)}=\phi_{0, v^{(0)}}$ for some $v_0>0$. The subsequent main result in this section establishes the validity of the message passing equations (\ref{eq:MP-gauss}) in this easier case.

\begin{proposition}\label{prop:mean} 
Let $v_0>0$. Assume that for all $i\in[N], a\in[m]$ $\nu\ia^{(0)}=\phi_{0, v^{(0)}}$.
Then, for all $t\in\NN$, the variances are given by some values independent on the labels $v\ia\ti=v\ti $ and the means obey the message passing (MP) equations
\begin{equation}\label{eq:MP-gauss} 
    x\ia\tti = \Delta\ti\sum_{b\neq a}y_bA_{bi} -\Delta\ti\sum_{b\neq a}\sum_{j\neq i}A_{bi}A_{bj} x\jb\ti. 
\end{equation} 
Moreover,
\begin{equation}\label{eq:iterationX}
    x\ia \ti = \sum_{\lambda=1}^t (-1)^{\lambda+1}  \Gamma_{\lambda}\ti  \sum_{\Bal}\sum_{\Jil} y_{b_\lambda} A_{b_1j_1}\prod_{\iota=1}^{\lambda-1}A_{b_\iota j_{\iota+1}}A_{b_{\iota+1}j_{\iota+1}}\,,
\end{equation}
where
\begin{equation}\label{eq:defGamma}
    \Gamma\ti_\lambda\coloneqq \prod_{\tau=1}^\lambda \Delta ^{(t-\tau)}, \quad \text{ with }\quad \Delta\ti \coloneqq \frac{\delta}{2\beta v\ti+\delta}\,,\quad t\geq0\,.
\end{equation}
and we define the set of indices
    \begin{align}
    \Jil&=\{j_1,j_2, \dots, j_{\lambda} \in [N]: j_1=i \text{ and }j_\alpha \neq j_{\alpha+1} \text{ for } \alpha\in [\lambda-1]\}\,,\label{eq:defJi}\\
    \Bal&=\{b_1, b_2, \dots, b_{\lambda} \in [m]: b_1 \neq a \text{ and } b_\alpha\neq b_{\alpha+1} \text{ for } \alpha \in [\lambda-1]\}\,.\label{eq:defBa}
\end{align}
\end{proposition}

The first step of our analysis is the following observation from \cite{DMM2}. We report the proof for completeness. 

\begin{Lem}\label{lemma:facilemaleki}
    If $\nu\ia\ti$ has mean $x\ia\ti$ and variance $v\ia\ti$, then $\hat{\nu}\ai\ti$ has mean
    \begin{equation}\label{eq:x_ai}
        \hat{x}\ai\ti = \frac{y_a}{A_{ai}}-\frac{1}{A_{ai}}\sum_{j\neq i}A_{aj}x\ja\ti 
    \end{equation}
    and variance 
    \begin{equation}\label{eq:v_ai}
    \hat{v}\ai\ti=\frac{1}{A_{ai}^2}\sum_{j\neq i}A_{aj}^2 v\ja\ti\,.
    \end{equation}
\end{Lem}
\begin{proof}
    Define the random variables  $\xi_{i \rightarrow a}\ti \sim \nu_{i \rightarrow a}\ti $ and $\hat{\xi}_{a \rightarrow i}\ti \sim \hat{\nu}_{a \rightarrow i}\ti.$ Then 
    \begin{equation} \label{eq:xi_hat}
    \frac{1}{A_{ai}}\left(y_a-\sum_{j \neq i}A_{aj} \xi_{j \rightarrow a}\ti\right) \overset{(d)}{=} \hat{\xi}_{a \rightarrow i}\ti.
\end{equation}
In fact, for any Borel set $S$ we have 
\begin{align*}
    \hat{\nu}_{a \rightarrow i}\ti(S) = \Prob\left(y_a-\sum_{j \neq i}A_{aj}x_j \in A_{ai}S \right) = \Prob\left( \frac{1}{A_{ai}} \left(y_a-\sum_{j \neq i}A_{aj}x_j\right) \in S \right).
\end{align*}
The claim follows by taking the expectation and the variance of \eqref{eq:xi_hat}.
\end{proof}
We now show the way Gaussian densities propagate thought the BP iteration. 
\begin{Lem}\label{lemma:gaussianfacile}
The following hold:
\begin{itemize}
\item[i)] If $\nu\ia^{(0)}$ is a centred Gaussian density, then for all $t$, $\nu\ia\ti$ and $\hat{\nu}\ai\ti$ are Gaussian densities.
\item [ii)] If $\hat{\nu}\ai\ti$ is a Gaussian density with mean $\hat{x}\ai\ti$ and variance $\hat{v}\ai\ti$ then $\nu\ia\tti$ is a Gaussian density with mean 
    \begin{equation}\label{eq:x_ia}
        x\ia\tti = \frac{\sum\limits_{b\neq a}\frac{\hat{x}\bi\ti}{\hat{v}\bi\ti}}{2\beta+\sum\limits_{b\neq a} \frac{1}{\hat{v}\bi\ti}} \,,
    \end{equation}
    and variance
    \begin{equation}\label{eq:v_ia}
        v\ia\tti = \frac{1}{2\beta + \sum\limits_{b\neq a}\frac{1}{\hat{v}\ti\bi}} \,.
    \end{equation}
    \end{itemize}
\end{Lem}
\begin{proof}
 $i)$   If for some $t\geq1$ $\nu\ia\ti$ is a Gaussian density, then the random variables $\xi_{i \rightarrow a}\ti \sim \nu_{i \rightarrow a}\ti $ defined in the proof of the previous result are centred Gaussian, therefore also  the variables $\hat{\xi}_{a \rightarrow i}\ti$ given by (\ref{eq:xi_hat}) are Gaussian. Thus $\hat{\nu}\ai\ti$ is a Gaussian density with mean $\frac{y_a}{A_{ai}}$ and variance $\frac{v\ia^{(0)}}{A_{ai}^2}\sum_{j\neq i}A_{aj}^2$. Next, by the second BP equation (\ref{eq:BPII-sect2}), $\nu\ia^{(t+1)}$ is the product of Gaussian densities, therefore will be Gaussian as well. 
 
Since $\nu\ia^{(0)}$ is a (centred) Gaussian density, we can start the above reasoning from $t=0$ and iterate.
 
$ii)$ By $i)$ and \eqref{eq:BP-sect2}
 \begin{equation*}
\nu\ia\tti(x_i) \cong  \exp\left(-\beta x_i^2\right)\prod_{b\neq a}
\exp\left(-\frac{(x_i-\hat{x}\bi\ti)^2}{2\hat{v}\bi\ti}\right).
 \end{equation*}
Collecting the quadratic and linear terms in $x_i$,
\begin{align*}
\nu\ia\tti(x_i)
&\cong \exp \left(-\frac12\Bigl(2\beta+\sum_{b\neq a}\frac{1}{\hat{v}\bi\ti}\Bigr)x_i^2
+ x_i\sum_{b\neq a}\frac{\hat{x}\bi\ti}{\hat{v}\bi\ti}\right),
\end{align*}
which is a one-dimensional Gaussian density with variance
 $\left(2\beta+\sum_{b\neq a}\frac{1}{\hat{v}\bi\ti}\right)^{-1} $ and mean 
 $$
 \frac{\sum\limits_{b\neq a}\frac{\hat{x}\bi\ti}{\hat{v}\bi\ti}}{2\beta+\sum\limits_{b\neq a} \frac{1}{\hat{v}\bi\ti}}\,.
 $$
\end{proof}

\begin{Lem}\label{lem:variance}
Let $v_0>0$. Assume that for all $i\in[N], a\in[m]$ $\nu\ia^{(0)}=\phi_{0, v^{(0)}}$. Then for all $t\in \NN$ 
    \begin{equation}\label{eq:unifVar-edge}
        v\ia\ti=v\ti \approx \frac{v^{(0)}(1-\delta)}{\delta^t(1-\delta)+2\beta v^{(0)}(1-\delta^t)}.
    \end{equation}
\end{Lem}
\begin{proof}
    Combining equations \eqref{eq:v_ia} and \eqref{eq:v_ai} we get the iteration for the variance
\begin{equation}\label{eq:iteration_v}
    v\ia\tti=\frac{1}{2\beta+\sum\limits_{b\neq a}\frac{A_{bi}^2}{\sum\limits_{j\neq i}A_{bj}^2v\jb\ti}}.
\end{equation}
Under the assumption $v^{(0)}\ia=v^{(0)}$, we see from the above equation $v\ia\ti=v\ti$ for all $t>0$ and we have
\begin{equation*}
    \sum_{b\neq a}\frac{A_{bi}^2}{\sum\limits_{j\neq i}A_{bj}^2 v\jb\ti} = \frac{1}{v\ti}\sum_{b\neq a}\frac{A_{bi}^2}{\sum\limits_{j\neq i}A_{bj}^2} \approx \frac{\delta}{v\ti}.
\end{equation*}
The fact that the numerator and the denominator at the r.h.s. of the identity above are independent justifies the last step. 
Thus, we get 
\begin{equation}\label{eq:iteration_v1}
    v\tti \approx \frac{v\ti}{2\beta v\ti+\delta}.
\end{equation}
We now prove by induction that 
\begin{equation}\label{eq:IH}
    v\ti \approx \frac{v^{(0)}}{\delta^t+2\beta v^{(0)}\sum\limits_{\tau=0}^{t-1}\delta^\tau}\,,
\end{equation}
and the claim follows by summing the geometric series: $\sum_{\tau=0}^{t-1}\delta^\tau=\frac{1-\delta^t}{1-\delta}$ (recall $\delta\in(0,1)$).

For $t=1$, $v\oone \approx \frac{v^{(0)}}{2\beta v^{(0)}+\delta}$ follows immediately from \eqref{eq:iteration_v1}. Let us suppose \eqref{eq:IH} holds for $t>1$. We have
\begin{align*}
    v\tti & \approx \frac{v\ti}{2\beta v\ti+\delta} \\
    & \approx \frac{v^{(0)}}{\delta^t+2\beta v^{(0)}\sum\limits_{\tau=0}^{t-1}\delta^\tau} \frac{1}{2\beta \frac{v^{(0)}}{\delta^t+2\beta v^{(0)}\sum\limits_{\tau=0}^{t-1}\delta^\tau} +\delta }\\
    &= \frac{v^{(0)}}{\delta^{t+1} +2\beta v^{(0)} \left(1+\sum\limits_{\tau=0}^{t-1}\delta^{\tau+1} \right)}\\
    &= \frac{v^{(0)}}{\delta^{t+1} +2\beta v^{(0)} \sum\limits_{\tau=0}^{t}\delta^{\tau}}\,.
\end{align*}
\end{proof}
\begin{corollary}\label{cor:vhat}
    Same assumptions of Lemma \ref{lem:variance}. Then, for all $t\in\NN$, it holds 
    \begin{equation}\label{eq:vHat}
        \hat{v}\ai\ti\approx \frac{1}{\delta A_{ai}^2}v\ti.
    \end{equation}
\end{corollary}
\begin{proof}
    By \eqref{eq:v_ai} and Lemma \ref{lem:variance} we have
    \begin{equation*}
        \hat{v}\ai\ti = \frac{1}{A_{ai}^2}\sum_{j\neq i}A_{aj}^2 v\ti \approx \frac{1}{\delta A_{ai}^2}v\ti.
    \end{equation*}
\end{proof}
\begin{Rem}
 Combining equations \eqref{eq:x_ia} and \eqref{eq:x_ai} we have the iteration for the mean
\begin{equation}\label{eq:itermean}
    x\ia\tti = \frac{\sum\limits_{b\neq a}A_{bi}\frac{y_b-\sum\limits_{j\neq i}A_{bj}x\jb\ti}{\sum\limits_{j\neq i}A_{bj}^2v\jb\ti}}{2\beta+\sum\limits_{b\neq a}\frac{A_{bi}^2}{\sum\limits_{j\neq i}A_{bj}^2v\jb\ti}}\,.
\end{equation}
\end{Rem}

We are now ready to derive the message passing equations.

\begin{proof}[Proof of \eqref{eq:MP-gauss}]
Starting from \eqref{eq:itermean} and using Corollary \ref{cor:vhat}, Lemma \ref{lem:variance}, $\sum_{b\neq a}A^2_{bi}\approx 1$ and $\sum_{j\neq i}A^2_{bj}\approx \d^{-1}$ we get \eqref{eq:MP-gauss} by simple algebra, recalling the explicit form of $\D\ti$ given in \eqref{eq:defGamma}. 
\end{proof}

The remaining part of Proposition \ref{prop:mean} is proven by the following lemma. 

\begin{Lem}\label{lem:mean_gaussian_chaos}
Fix $v\ia\zzero=v\zzero$. Define the index sets
\begin{align}
\Jil
&=\bigl\{(j_1,\dots,j_\lambda)\in [N]^\lambda:\ j_1=i\text{ and }j_\alpha\ne j_{\alpha+1}\ \forall \alpha\in[\lambda-1]\bigr\},\label{eq:defJi}\\
\Bal
&=\bigl\{(b_1,\dots,b_\lambda)\in [m]^\lambda:\ b_1\ne a\text{ and }b_\alpha\ne b_{\alpha+1}\ \forall \alpha\in[\lambda-1]\bigr\}.\label{eq:defBa}
\end{align}
Then, for all  $t\in\NN $,
\begin{equation}\label{eq:iterationX}
x\ia \ti
=\sum_{\lambda=1}^t (-1)^{\lambda+1} \Gamma_{\lambda}\ti
\sum_{\Bal}\sum_{\Jil}
y_{b_\lambda} A_{b_1j_1}\prod_{\iota=1}^{\lambda-1}A_{b_\iota j_{\iota+1}}A_{b_{\iota+1}j_{\iota+1}},
\end{equation}
where
\begin{equation}\label{eq:defGamma}
\Gamma\ti_\lambda=\prod_{\tau=1}^\lambda \Delta^{(t-\tau)},
\qquad
\Delta^{(\tau)}=\frac{\delta}{ 2\beta v^{(\tau)}+\delta }.
\end{equation}
\end{Lem}

\begin{proof} We argue by induction on  $t $.
 
\emph{Base case  $t=1 $.} With  $v^{(0)}\ia=v^{(0)} $, equations \eqref{eq:x_ia}, \eqref{eq:unifVar-edge} and Corollary~\ref{cor:vhat} give 
 \begin{equation*} x\ia\oone = v\oone\sum_{b\neq a}\frac{\hat x\bi^{(0)}}{\hat v\bi^{(0)}} \approx \frac{v\oone}{v^{(0)}} \delta \sum_{b\neq a}A_{bi}^2\hat x\bi^{(0)}.  \end{equation*} 
Using \eqref{eq:x_ai} at  $t=0 $,  $\hat x\bi^{(0)}=y_b/A_{bi} $, hence  \begin{equation*} x\ia\oone \approx \frac{v\oone}{v^{(0)}} \delta \sum_{b\neq a}y_bA_{bi} =\Delta\oone\sum_{b\neq a}y_bA_{bi},  \end{equation*} which matches \eqref{eq:iterationX} for  $t=1 $ (with  $\Gamma\oone_1=\Delta\oone $). \\
\emph{Induction step  $t\to t+1 $.} Assume \eqref{eq:iterationX} holds at time  $t $. Combining \eqref{eq:x_ia}, \eqref{eq:iteration_v1} and Corollary~\ref{cor:vhat}, 
\begin{align} 
    x\ia\tti &= v\tti\sum_{b\neq a}\frac{\hat x\bi\ti}{\hat v\bi\ti} \approx \frac{v\tti}{v\ti} \delta\sum_{b\neq a}A_{bi}^2\hat x\bi\ti \approx \Delta\ti\sum_{b\neq a}A_{bi}^2\hat x\bi\ti. \label{eq:decomposizione} 
\end{align} 
From \eqref{eq:x_ai}, 
\begin{equation}\label{eq:split} 
    x\ia\tti = \Delta\ti\sum_{b\neq a}y_bA_{bi} -\Delta\ti\sum_{b\neq a}\sum_{j\neq i}A_{bi}A_{bj} x\jb\ti. 
\end{equation} 
Insert the induction hypothesis for  $x\jb\ti $ into the second term: 
\begin{align} 
     \sum_{b\neq a}\sum_{j\neq i}A_{bi}A_{bj} x\jb\ti &=  \sum_{\lambda=1}^t(-1)^{\lambda+1}\Gamma\ti_\lambda \sum_{b\neq a}\sum_{j\neq i}A_{bi}A_{bj} \sum_{\mathcal{B}_b^\lambda}\sum_{\mathcal{J}_j^\lambda} y_{b_\lambda} A_{b_1j_1} \prod_{\iota=1}^{\lambda-1} A_{b_\iota j_{\iota+1}}A_{b_{\iota+1}j_{\iota+1}}. \label{eq:step} 
\end{align} 
Introduce the concatenated index sets 
\begin{align*}
    \mathcal{J}_{i,j}^{\lambda+1} &=\{(j_0,\dots,j_\lambda): j_0=i,\ (j_1,\dots,j_\lambda)\in\mathcal{J}_j^\lambda\},\\ 
    \mathcal{B}_{a,b}^{\lambda+1} &=\{(b_0,\dots,b_\lambda): b_0\ne a,\ (b_1,\dots,b_\lambda)\in\mathcal{B}_b^\lambda\}. 
\end{align*}
Then \eqref{eq:step} rewrites as 
 \begin{equation*} 
\sum_{\lambda=1}^t(-1)^{\lambda+1}\Gamma\ti_\lambda \sum_{\mathcal{B}_{a,b}^{\lambda+1}}\sum_{\mathcal{J}_{i,j}^{\lambda+1}} y_{b_{\lambda+1}} A_{b_0j_0}\prod_{\iota=0}^{\lambda-1}A_{b_\iota j_{\iota+1}}A_{b_{\iota+1}j_{\iota+1}}. 
 \end{equation*} 
Relabeling indices shows that this equals 
 \begin{equation*} 
\sum_{\lambda=2}^{t+1}(-1)^{\lambda}\Gamma\ti_{\lambda-1} \sum_{\mathcal{B}_{a}^{\lambda}}\sum_{\mathcal{J}_{i}^{\lambda}} y_{b_{\lambda}} A_{b_1j_1}\prod_{\iota=1}^{\lambda-1}A_{b_\iota j_{\iota+1}}A_{b_{\iota+1}j_{\iota+1}}. 
 \end{equation*} 
Plugging into \eqref{eq:split} yields 
 \begin{equation*} 
x\ia\tti = \Delta\ti\sum_{\lambda=1}^{t+1}(-1)^{\lambda+1}\Gamma\ti_{\lambda-1} \sum_{\mathcal{B}_{a}^{\lambda}}\sum_{\mathcal{J}_{i}^{\lambda}} y_{b_{\lambda}} A_{b_1j_1}\prod_{\iota=1}^{\lambda-1}A_{b_\iota j_{\iota+1}}A_{b_{\iota+1}j_{\iota+1}}, 
 \end{equation*} 
and since  $\Gamma\ti_{\lambda-1}\Delta\ti=\Gamma\tti_\lambda $, we obtain \eqref{eq:iterationX} at time  $t+1 $. 
\end{proof}


\section{Tail control of the messages}\label{sect:proofProp}

In this section our goal is to show that the messages  $x\ia\ti $ are $O_\P(1)$. Moreover we will establish a suitable decomposition that will be useful in the next section. Our main result follows.

\begin{Prop}\label{prop:AMP_dec}
Let $v_0>0$. Assume that for all $i\in[N], a\in[m]$ $\nu\ia^{(0)}=\phi_{0, v^{(0)}}$. For every $t\in\NN$ and edge $(i,a)$,
\begin{equation}\label{eq:scomposizione}
x\ia\ti
=\sum_{\lambda=1}^t (-1)^{\lambda+1} \Gamma_\lambda\ti 
\bigl( x_i^{[\lambda]} + z\ai ^{[\lambda]} \bigr),
\end{equation}
where for all $\lambda\in[t]$
\begin{equation*}
\max_{i\in [N]}|x_i^{[\lambda]}|=O_\PP(1)
\quad\text{and}\quad
\max_{\substack{i\in [N]\\ a\in[m]}}|z\ai ^{[\lambda]}|=O_\PP(N^{-1/2})\,.
\end{equation*}
\end{Prop}
\begin{proof}
Let us consider the case $t=1$ first. From \eqref{eq:x_ia}-\eqref{eq:v_ia} and Corollary~\ref{cor:vhat} we have
\be\label{eq:casotugu1}
x\ia\oone =\Delta\oone \sum_{b\neq a}y_bA_{bi}
= \Delta\oone  \left(\sum_{b=1}^m y_bA_{bi}-y_aA_{ai}\right).
\ee
Comparing (\ref{eq:scomposizione}) and (\ref{eq:casotugu1}) and noting that $\sum_b y_b A_{bi}=O_\mathbb{P}(1)$  and $y_aA_{ai}=O_\mathbb{P}(N^{-1/2})$ gives the assertion in the case $t=1$. 

Now we pass to consider $t\geq2$. By Lemma~\ref{lem:mean_gaussian_chaos} we have
\begin{equation*} 
x\ia\ti
=\sum_{\lambda=1}^t (-1)^{\lambda+1}\Gamma_\lambda\ti
\sum_{\mathcal{B}_a^\lambda}\sum_{\mathcal{J}_i^\lambda}
y_{b_\lambda} A_{b_1j_1}\prod_{\iota=1}^{\lambda-1}A_{b_\iota j_{\iota+1}}A_{b_{\iota+1}j_{\iota+1}}.
\end{equation*} 
Fix $\lambda$ and abbreviate
\be\label{def:f}
f(b_1,\dots,b_\lambda)
=\sum_{\mathcal{J}_i^\lambda}
y_{b_\lambda} A_{b_1j_1}\prod_{\iota=1}^{\lambda-1}A_{b_\iota j_{\iota+1}}A_{b_{\iota+1}j_{\iota+1}}\,,
\ee
so that
\be\label{eq:XintemrsofF}
x\ia\ti
=\sum_{\lambda=1}^t (-1)^{\lambda+1}\Gamma_\lambda\ti
\sum_{\mathcal{B}_a^\lambda} f(b_1,\dots,b_\lambda)\,.
\ee
 Set $b_0=a$ and $B^\lambda =\{ \{ b_1, \dots, b_\lambda\} \in [m]^\lambda\}.$ We can rewrite the set $\Bal$ as 
    \begin{equation*}
        \Bal=B^\lambda \setminus \bigcup \limits_{\zeta=1}^\lambda  \bigcup_{\substack{B \subseteq B^\lambda \\ \lvert B\lvert=\zeta}} \left\{b_\alpha \in B: b_\alpha=b_{\alpha-1} \right\}.
    \end{equation*}
splitting the inclusion-exclusion terms according to whether $a$ is
present in the set of enforced equalities yields the decomposition 
\bea
x_i^{[\lambda]}
& \coloneqq&  \sum_{B^\lambda} f(b_1,\dots,b_\lambda)
 +  \sum_{\zeta=1}^\lambda (-1)^{\zeta}
\sum_{\substack{B\subseteq B^\lambda\setminus\{a\}\\ |B|=\zeta}}
\sum_{B^\lambda\setminus B}
f(b_1,\dots,b_\lambda)\prod_{b_\alpha\in B}\mathbbm{1}_{\{b_\alpha=b_{\alpha-1}\}},\label{eq:defxquadrat}\\
z\ai ^{[\lambda]}
& \coloneqq  &\sum_{\zeta=1}^\lambda (-1)^{\zeta}
\sum_{\substack{B\subseteq B^\lambda\\ |B|=\zeta,\ a\in B}}
\sum_{B^\lambda\setminus B}
f(b_1,\dots,b_\lambda)\prod_{b_\alpha\in B}\mathbbm{1}_{\{b_\alpha=b_{\alpha-1}\}}.\label{eq:defzquadrat}
\eea

The proof is then completed by the subsequent Proposition \ref{prop:tails-xy}. 
\end{proof}

\begin{proposition}\label{prop:tails-xy}
Fix $t\in\NN$. There exist positive constants $K_1=K_1 \bigl(t,\delta,\|y\|_2,\Gamma\ti_1\bigr)$ and
$K_2=K_2 \bigl(t,\delta,\|y\|_2,\Gamma\ti_1\bigr)$, independent of $(i,a)$ and $N$, such that for all $\alpha\ge0$,
 \be\label{eq:tails-x}
\mathbb{P} \left(\bigl|x\ia\ti\bigr|\ge \alpha\right)
 \le  K_1 \exp \Bigl(- K_2 \alpha^{1/t}\Bigr).
 \ee
 and
 \be\label{eq:tails-z}
\mathbb{P} \left(\bigl|z\ia\ti\bigr|\ge \alpha\right)
 \le  K'_1 \exp \Bigl(- K'_2 (\sqrt N\alpha)^{1/t}\Bigr).
 \ee
 \end{proposition}

\begin{proof}
We prove (\ref{eq:tails-x}). 
By Lemma~\ref{lem:Lp_sharp} below, there exists a constant
$D_t=D_t\bigl(t,\delta,\|y\|_2,\Gamma\ti_1\bigr) > 0$ such that for all $p\ge 2$
 \begin{equation*}
\|x\ia\ti\|_{L^p}
 \le  D_t p^t,
\qquad
D_t = C \Gamma\ti_{1} 
\frac{\lVert y\rVert_2}{\sqrt N} 
\frac{1}{\delta^{ t/2}} 
t (4t)^{ t+1}.
 \end{equation*}
By Markov's inequality,
 \begin{equation*}
\mathbb{P} \left(\left\lvert x\ia\ti\right\lvert \ge \alpha\right)
 \le \left(\frac{D_t p^t}{\alpha}\right)^{ p}\,.
 \end{equation*}
Choose
 $
p  =  \max \left\{2,\ \left\lfloor \bigl(\alpha/(eD_t)\bigr)^{1/t}\right\rfloor\right\}.
 $
If $\alpha \ge eD_t 2^{ t}$, then $D_t p^t/\alpha\le e^{-1}$ and thus
 \begin{equation*}
\mathbb{P} \left(\left\lvert x\ia\ti\right\lvert \ge \alpha\right)
 \le  e^{-p}
 \le  \exp \left( - K_2 \alpha^{1/t}\right),
 \end{equation*}
for some $K_2=K_2\bigl(t,\delta,\|y\|_2,\Gamma\ti_1\bigr)>0$. If instead $\alpha < eD_t 2^{ t}$, the bound is trivial ($\le 1$) and can be absorbed into a front constant $K_1=K_1\bigl(t,\delta,\|y\|_2,\Gamma\ti_1\bigr)>0$.
This yields the stated inequality, uniformly over $(i,a)$.

The proof of (\ref{eq:tails-z}) is identical, using Lemma \ref{lem:boundary-Lp-corrected} instead of Lemma \ref{lem:Lp_sharp}. 
\end{proof}

Next we establish uniform $L^p $ bounds for  $x\ia\ti $ and $z\ia\ti $ (see Lemma~\ref{lem:Lp_sharp} and Lemma~\ref{lem:boundary-Lp-corrected}).
The strategy is to control the $L^p $ norm of each of the terms appearing in the expansions (\ref{eq:defxquadrat}) and (\ref{eq:defzquadrat}). Each singular term writes as a finite sum of multilinear polynomials in the entries of the matrix $A$.
Two standard tools will be used:

\ 

\begin{itemize}
\item {\em Minkowski's inequality (triangle inequality in $L^p$).}
\item \emph{ Hypercontractivity} (for Gaussian r.vs this was proved by Nelson, see Theorem~5.10 in~\cite{janson}).
If $\Psi_\ell$ is an $\ell$-linear form (or chaos of order $\ell$), then for $p\ge2$,
 \be\label{eq:nelson}
\|\Psi_\ell\|_{L^p} \le C_\ell\|\Psi_\ell\|_{L^2} p^{\ell/2}.
 \ee
\end{itemize}

Moreover we will need to estimate the $L^2$ norms on the r.h.s. by the ones w.r.t. the standard Gaussian measure. 
This is done by a direct application of Proposition \ref{prop:equiv} in Appendix \ref{App:equiv}. 

Finally, to estimate the $L^2$ norm of the Gaussian chaoses we will use the following version of the Wick theorem (see for instance \cite[Theorem 1.28]{janson} and surrounding discussion). Let $I$ be a index-set. A \emph{pairing} is a partition of $I$ into pairs $\{i,j\}$. We denote by $\mathcal{P}_I$ the set of all possible pairings of $I$.

\begin{theorem}\label{thm:Wick4}
    Let $l\in\N$ and $(g_i)_{i\in[4l]}$ be Gaussian centred random variables. Let $(m_i, n_i, p_i, q_i)_{i\in[l]}$ be a family of indices and denote $\mathcal{A}_l=(\{m_i,n_i, p_i,q_i\})_{i\in[l]}$. Then it holds
    \begin{equation}\label{eq:Wick4}
        \EE \left[ \prod_{i=1}^l g_{m_i}g_{n_i}g_{p_i}g_{q_i}\right]=\sum_{P \in \mathcal{P}_{\mathcal{A}_l}}\prod_{\{\alpha,\beta \}\in P}\EE[g_\alpha g_\beta]\,,
    \end{equation}
    where we denote by $\mathcal{P}_{\mathcal{A}_l}$ the set of all pairings of $\mathcal{A}_l$.
    In particular, if $\var(g_i)=1$ for all $i$, we have 
    \begin{equation*}
        \EE \left[ \prod_{i=1}^l g_{m_i}g_{n_i}g_{p_i}g_{q_i}\right]=\sum_{P \in \mathcal{P}_{\mathcal{A}_l}}\prod_{\{\alpha,\beta \}\in P}\delta_{\alpha\beta}\,.
    \end{equation*}
\end{theorem}

\begin{Lem}\label{lem:Lp_sharp}
Same assumptions of Proposition \ref{prop:AMP_dec}. For every $t\in\NN$ and $p\geq2$,
\begin{equation}\label{eq:Lp_sharp_bound}
\bigl\lVert x\ia\ti\bigr\rVert_{L^p}
 \le 
C \Gamma\ti_{1} 
\frac{\lVert y\rVert_2}{\sqrt N} 
\frac{p^{ t}}{\delta^{ t/2}} 
t (4t)^{ t+1},
\end{equation}
where $C>0$ is an absolute constant and $\Gamma\ti_{1}$ is defined in \eqref{eq:defGamma}.
\end{Lem}

\begin{proof} 
By Lemma~\ref{lem:mean_gaussian_chaos}, 
 \begin{equation*} 
x\ia\ti =\sum_{\lambda=1}^{t}(-1)^{\lambda+1}\Gamma\ti_{\lambda} \Psi\ia^{(2\lambda-1)},  \end{equation*} with  \begin{equation*} \Psi\ia^{(2\lambda-1)} \coloneqq \sum_{\mathcal{B}_a^\lambda}\sum_{\mathcal{J}_i^\lambda} y_{b_\lambda} A_{b_1 j_1}\prod_{\iota=1}^{\lambda-1}A_{b_\iota j_{\iota+1}}A_{b_{\iota+1}j_{\iota+1}}, 
 \end{equation*} 
with index sets $\mathcal{B}_a^\lambda$, $\mathcal{J}_i^\lambda$ as in \eqref{eq:defBa}-\eqref{eq:defJi}. 
Minkowski's and Nelson's inequality yield \begin{equation}\label{eq:basicNelson} 
    \bigl\lVert x\ia\ti\bigr\rVert_{L^p} \le \sum_{\lambda=1}^{t}\Gamma\ti_{\lambda} \bigl\lVert \Psi\ia^{(2\lambda-1)}\bigr\rVert_{L^p} \le \sum_{\lambda=1}^{t}\Gamma\ti_{\lambda} (p-1)^{\frac{2\lambda-1}{2}} \bigl\lVert \Psi\ia^{(2\lambda-1)}\bigr\rVert_{L^2}. 
\end{equation} Hence it suffices to bound $\lVert \Psi\ia^{(2\lambda-1)}\rVert_{L^2}$. In virtue of Proposition \ref{prop:equiv} of Appendix \ref{App:equiv} we can reduce to do the estimate assuming the entries of the matrix $A$ to be independent standard Gaussian r.vs. \smallskip\\
\emph{Wick expansion with explicit index constraints.} Write $A_{bj}=\bar A_{bj}/\sqrt m$ with i.i.d.\ standard Gaussians $\bar A_{bj}$. Let $\widetilde{\mathcal{B}}_a^\lambda$, $\widetilde{\mathcal{J}}_i^\lambda$ be independent copies of $\mathcal{B}_a^\lambda$, $\mathcal{J}_i^\lambda$, and write 
 \begin{equation*}
\|\Psi\ia^{(2\lambda-1)}\|_{L^2}^2 =\EE\Big[\Psi\ia^{(2\lambda-1)} \widetilde{\Psi}\ia^{(2\lambda-1)}\Big], 
 \end{equation*} 
where $\widetilde{\Psi}$ is an independent copy of $\Psi$. Expanding, 
\begin{align} 
    \|\Psi\ia^{(2\lambda-1)}\|_{L^2}^2 &=\sum_{\mathcal{B}_a^\lambda}\sum_{\widetilde{\mathcal{B}}_a^\lambda} \sum_{\mathcal{J}_i^\lambda}\sum_{\widetilde{\mathcal{J}}_i^\lambda} y_{b_\lambda}y_{\widetilde b_\lambda}  \EE \left[ A_{b_1 j_1}A_{\widetilde b_1\widetilde j_1} \prod_{\iota=1}^{\lambda-1} A_{b_\iota j_{\iota+1}}A_{b_{\iota+1} j_{\iota+1}} A_{\widetilde b_\iota \widetilde j_{\iota+1}}A_{\widetilde b_{\iota+1} \widetilde j_{\iota+1}} \right]\nonumber\\ 
    &=m^{-(2\lambda-1)}    \sum_{\mathcal{B},\widetilde{\mathcal{B}},\mathcal{J},\widetilde{\mathcal{J}}} y_{b_\lambda}y_{\widetilde b_\lambda}  \EE \left[ \bar A_{b_1 j_1}\bar A_{\widetilde b_1\widetilde j_1} \prod_{\iota=1}^{\lambda-1} \bar A_{b_\iota j_{\iota+1}}\bar A_{b_{\iota+1} j_{\iota+1}} \bar A_{\widetilde b_\iota \widetilde j_{\iota+1}} \bar A_{\widetilde b_{\iota+1} \widetilde j_{\iota+1}} \right].\label{eq:L2square_pre} 
\end{align} 
We introduce the node list $\mathcal{A}_\lambda=(A_\ell)_{\ell=0}^{\lambda-1}$ with 
\begin{align*}
    A_0&=\{n_0,q_0\}=\{(b_1,j_1),(\widetilde b_1,\widetilde j_1)\}, \\
    A_\iota&=\{m_\iota,n_\iota,p_\iota,q_\iota\} =\{(b_\iota,j_{\iota+1}),(b_{\iota+1},j_{\iota+1}), (\widetilde b_\iota,\widetilde j_{\iota+1}),(\widetilde b_{\iota+1},\widetilde j_{\iota+1})\}. 
\end{align*} 
Let $\Lambda=4\lambda-2$ be the total number of Gaussian factors. By Theorem \ref{thm:Wick4} the expectation in \eqref{eq:L2square_pre} is a sum over pairings $P\in\mathcal{P}_{\mathcal{A}_\lambda}$, each contributing a product of Kronecker deltas enforcing equality of the paired indices. Thus 
\begin{equation}\label{eq:L2square_pairings} 
\|\Psi\ia^{(2\lambda-1)}\|_{L^2}^2 =m^{-(2\lambda-1)}  \sum_{P\in\mathcal{P}_{\mathcal{A}_\lambda}} \sum_{\mathcal{B},\widetilde{\mathcal{B}},\mathcal{J},\widetilde{\mathcal{J}}} y_{b_\lambda}y_{\widetilde b_\lambda}  \prod_{\{\alpha,\beta\}\in P} \delta_{\alpha\beta}. 
\end{equation} 
\emph{Index constraints.} The admissible index sets enforce the following conditions  
\begin{equation*} 
j_\alpha\neq j_{\alpha+1}\quad(\alpha\in[\lambda-1]), \qquad b_1\neq a,  b_\alpha\neq b_{\alpha+1}\quad(\alpha\in[\lambda-1]).  
\end{equation*} 
Consequently, the following forbidden identifications must not be imposed by $P$, otherwise the corresponding term vanishes: 
\begin{align} 
    m_\iota &\neq n_\iota, & p_\iota &\neq q_\iota, \label{eq:V1}\\
    n_\iota &\neq m_{\iota+1}, & q_\iota &\neq p_{\iota+1}, \label{eq:V2}\\ 
    m_\iota &\neq m_{\iota+1}, & p_\iota &\neq p_{\iota+1}, \label{eq:V3}\\ 
    n_\iota &\neq n_{\iota+1}, & q_\iota &\neq q_{\iota+1}, \label{eq:V4}\\ 
    m_\iota &\neq n_{\iota+1}, & p_\iota &\neq q_{\iota+1}, \label{eq:V5}\\ 
    n_\iota &\neq m_{\iota+2}, & q_\iota &\neq p_{\iota+2}. \label{eq:V6} 
\end{align} 
These constraints 
will eliminate many pairings. \smallskip\\ 
\emph{The identity pairing.} Consider the identity pairing $\mathsf{I}_\Lambda$ that pairs across the two copies at matching positions: 
 \begin{equation*} 
n_0\leftrightarrow q_0,\qquad m_\iota\leftrightarrow p_\iota,\quad n_\iota\leftrightarrow q_\iota\qquad(\iota=1,\dots,\lambda-1). 
 \end{equation*} 
This pairing obeys \eqref{eq:V1}-\eqref{eq:V6} and enforces  $ \mathcal{J}_i^\lambda=\widetilde{\mathcal{J}}_i^\lambda  $,  $ \mathcal{B}_a^\lambda=\widetilde{\mathcal{B}}_a^\lambda  $, $b_\lambda=\widetilde b_\lambda$, whence 
 \begin{equation*} \sum_{\mathcal{B},\widetilde{\mathcal{B}},\mathcal{J},\widetilde{\mathcal{J}}} y_{b_\lambda}y_{\widetilde b_\lambda} \prod_{\{\alpha,\beta\}\in \mathsf{I}_\Lambda}\delta_{\alpha\beta} = \sum_{\mathcal{B}_a^\lambda}\sum_{\mathcal{J}_i^\lambda} y_{b_\lambda}^2 =(\#\mathcal{B}_a^\lambda)(\#\mathcal{J}_i^\lambda) \|y\|_2^2. 
 \end{equation*} 
Counting the choices gives $(\#\mathcal{B}_a^\lambda)\asymp m^{\lambda-1}$ and $(\#\mathcal{J}_i^\lambda)\asymp N^{\lambda-1}$. Therefore the identity contribution equals 
\begin{equation}\label{eq:identity_value} 
    m^{-(2\lambda-1)}\cdot m^{\lambda-1}N^{\lambda-1}\cdot \|y\|_2^2 =\frac{1}{m^{\lambda}N^{1-\lambda}}\|y\|_2^2 =\frac{1}{\delta^\lambda}\cdot\frac{\|y\|_2^2}{N},
\end{equation} 
since $m=\delta N$. \smallskip\\ 
\emph{Perturbations of the identity pairing.} Any $P\in\mathcal{P}_{\mathcal{A}_\lambda}$ can be written as a perturbation of $\mathsf{I}_\Lambda$: select an even number $k$ of perturbed nodes $S\subset\{n_0,m_1,n_1,p_1,q_1,\dots,m_{\lambda-1},n_{\lambda-1},p_{\lambda-1},q_{\lambda-1}\}$, and keep the identity pairing on the complement, while on $S$ arrange any pairing $\mathsf{Q}_k$ that still obeys the constraints \eqref{eq:V1}-\eqref{eq:V6}. We denote this by  $ P=\mathsf{I}_{\Lambda-k}\cup \mathsf{Q}_k  $. 
Let 
 \begin{equation*} V(k)\coloneqq \sum_{\mathcal{B},\widetilde{\mathcal{B}},\mathcal{J},\widetilde{\mathcal{J}}} y_{b_\lambda}y_{\widetilde b_\lambda}  \prod_{\{\alpha,\beta\}\in \mathsf{I}_{\Lambda-k}\cup \mathsf{Q}_k}\delta_{\alpha\beta}. 
 \end{equation*} 
\emph{Claim:} For every even $k\in\{2,\dots,\Lambda\}$, either $V(k)=0$ (if \eqref{eq:V1}-\eqref{eq:V6} are violated) or 
\begin{equation}\label{eq:Vk_leq_V0} 
    V(k)\le V(0)=\sum_{\mathcal{B}_a^\lambda,\mathcal{J}_i^\lambda} y_{b_\lambda}^2\leq \|y\|_2^2. 
\end{equation} 
\emph{Proof of Claim.} Each cross-copy identification imposed by $\mathsf{Q}_k$ glues indices across the two copies in addition to the identity constraints, thus reduces the number of free summation indices. The worst case (largest value) arises when no additional equalities are imposed beyond $\mathsf{I}_\Lambda$, namely $k=0$. \smallskip \\
\emph{Counting admissible perturbations.} It remains to bound the number of admissible $\mathsf{Q}_k$ for a fixed (even) $k$. Easy counting gives
 \begin{equation*}
\#\{\text{admissible }\mathsf{Q}_k\}  \le  \prod_{\substack{\kappa=3\\ \kappa\ \text{odd}}}^{k+1}(\Lambda-\kappa)  \le  (4\lambda)^{ k/2+1}, 
 \end{equation*} 
where we used $\Lambda=4\lambda-2$. 
Summing over even $k$ gives 
 \be\label{eq:pertcount} 
\sum_{\substack{k=0\\ k\text{ even}}}^{\Lambda}\#\{\mathsf{Q}_k\}  \le  1+\sum_{\ell=1}^{\Lambda/2}(4\lambda)^{ \ell}  \le \sum_{\ell=0}^{2\lambda+1}(4\lambda)^{ \ell}. 
 \ee
\smallskip 
\emph{Collecting the $L^2$ bound.} By \eqref{eq:L2square_pairings}, \eqref{eq:Vk_leq_V0} and \eqref{eq:pertcount}, 
 \begin{equation*} 
\|\Psi\ia^{(2\lambda-1)}\|_{L^2}^2  \le  m^{-(2\lambda-1)} V(0)  \sum_{\substack{k=0\\ k\text{ even}}}^{\Lambda}\#\{\mathsf{Q}_k\}  \le  \frac{1}{\delta^\lambda}\frac{\|y\|_2^2}{N} \sum_{\ell=0}^{2\lambda+1}(4\lambda)^{ \ell}. 
 \end{equation*} 
Taking square roots, 
\begin{equation}\label{eq:L2Psi_final} 
    \bigl\lVert \Psi\ia^{(2\lambda-1)}\bigr\rVert_{L^2}  \le  \frac{\|y\|_2}{\sqrt N} \frac{1}{\delta^{\lambda/2}} (4\lambda)^{ \lambda+1}. 
\end{equation} \smallskip 
\emph{From $L^2$ to $L^p$.} 
Insert \eqref{eq:L2Psi_final} into \eqref{eq:basicNelson} and bound $p^{(2\lambda-1)/2}\le p^\lambda$, $(4\lambda)^{\lambda+1}\le (4t)^{t+1}$, $\Gamma\ti_\lambda\le \Gamma\ti_1$, and $\sum_{\lambda=1}^t \delta^{-\lambda/2}\le t \delta^{-t/2}$. This yields \eqref{eq:Lp_sharp_bound}. \end{proof}

\begin{Lem}\label{lem:boundary-Lp-corrected}
Fix $\lambda\in\mathbb{N}$. Let $(a,i)\in[m]\times[N]$ and (recall \eqref{eq:defzquadrat})
\begin{equation*}
z\ai^{[\lambda]}
\;=\;
\sum_{\zeta=1}^{\lambda}(-1)^{\zeta}\Phi\ai^{(\zeta)}
,
\end{equation*}
where
\be
\Phi\ai^{(\zeta)}:=\sum_{\substack{B\subseteq B^\lambda\\ |B|=\zeta,\ a\in B}}
\sum_{B^\lambda\setminus B}
f(b_1,\dots,b_\lambda)\,
\prod_{b_\alpha\in B}\mathbbm{1}_{\{b_\alpha=b_{\alpha-1}\}}
\ee
and
\be
f(b_1,\dots,b_\lambda)
\;=\;
\sum_{\mathcal{J}_i^\lambda}
y_{b_\lambda}\, A_{b_1 j_1}\,
\prod_{\iota=1}^{\lambda-1} A_{b_\iota j_{\iota+1}}A_{b_{\iota+1}j_{\iota+1}}.
\ee
Here $b_0=a$, $B_\l:=\{\{m_1,\ldots, m_\l\}\in[m]^\l\}$ and the index set
$\mathcal{J}_i^\lambda$ is defined in \eqref{eq:defJi}. Then, for every $p\ge2$, there exists a constant $C>0$ (independent of $N,m,a,i$ and of $\beta$) such that
\begin{equation*}
\bigl\|z\ai^{[\lambda]}\bigr\|_{L^p}
\;\le\;
\,\frac{C}{\sqrt N}\,
\frac{p^{\lambda}}{\delta^{\,(\lambda-1)/2}}\,
(4\lambda)^{\lambda+1},
\qquad\text{uniformly in }(a,i)\in[m]\times[N].
\end{equation*}
\end{Lem}

\begin{proof}
The proof is similar to the one of the previous lemma and we will only sketch it, underlining the main points. 
Every summand in $z\ai^{[\lambda]}$ is a homogeneous polynomial of degree $2\lambda-1$ in the entries of $A$, with the crucial feature that at least one adjacency equality involves the left endpoint $b_0=a$; consequently each summand contains an explicit factor $A_{ai}$.\\

Again by Proposition \ref{prop:equiv}, we may w.l.o.g. assume that $A$ has i.i.d. standard Gaussian entries. By Minkowski's and Nelson's inequality, for $p\ge2$,
\begin{equation*}
\bigl\|z\ai^{[\lambda]}\bigr\|_{L^p}
\;\le\;
\sum_\zeta \|\Phi\ai^{(\zeta)}\|_{L^p}
\;\le\;
\sum_\zeta (p-1)^{(2\lambda-1)/2}\,\|\Phi\ai^{(\zeta)}\|_{L^2}
\;\lesssim\;
p^{\lambda}\,\sum_\zeta \|\Phi\ai^{(\zeta)}\|_{L^2}.
\end{equation*}

Next we expand $\|\Phi\ai^{(\zeta)}\|_{L^2}^2$ as in Lemma~\ref{lem:Lp_sharp}. 
The constraints of \eqref{eq:defJi}-\eqref{eq:defBa} eliminate forbidden pairings; the identity pairing is admissible and dominant. 
Relative to the bulk estimate \eqref{eq:L2Psi_final}, the boundary summand differs in three ways:
\begin{enumerate}
\item one row index is pinned at $a$, so a factor $\sqrt m$ from summing over row indices is lost,
\item one column index is pinned at $i$, so a factor $\sqrt N$ from summing over column indices is lost,
\item there is an explicit Gaussian factor $A_{ai}$, contributing $\|A_{ai}\|_{L^2}=m^{-1/2}$.
\end{enumerate}
Together, these effects reduce the size of the leading term by a factor of order $1/(\sqrt{m}\sqrt{N})$.
Hence
\begin{equation*}
\|\Phi\ai^{(\zeta)}\|_{L^2}
\;\le\;
C'\,\frac{\|y\|_2}{N}\,
\frac{1}{\delta^{(\lambda-1)/2}}\,
(4\lambda)^{\lambda+1},
\end{equation*}
uniformly over the indices and over $\zeta$.
All the other pairings, obtained as perturbations of the identity,  are treated as in Lemma~\ref{lem:Lp_sharp}. 
Combining that with the previous display 
yields the stated $L^p$ bound.

Finally, using that $\max_{a\in[m]}|y_a|\leq 1$, we have $\|y\|_2 \leq \sqrt{\delta N}$. Substituting gives
\[
\frac{\|y\|_2}{N} \delta^{-(\lambda-1)/2}\leq  \frac{1}{\sqrt N}\,\delta^{(2-\lambda)/2}\,.
\]
This proves that 
$z\ai^{[\lambda]}=O_\PP(N^{-1/2})$ uniformly in $(a,i)$.
\end{proof}

By similar techniques, we establish also the next lemma, that will be used in the next section.

\begin{Lem}\label{lem:double_sum_Y}
Define the quantity
\[
Z_i \;\coloneqq\; \sum_{b=1}^m \sum_{j=1}^N A_{b i} A_{b j} Y\bj\ti,
\]
where the messages \(Y\bj\ti\) are defined by 
\[
Y_{a \to i}\ti(\beta) \coloneqq  \sum_{\lambda=1}^t (-1)^{\lambda+1} \Gamma_\lambda\ti \, z_{a \to i}^{[\lambda]},
\]
Then
\[
\max_{i\in[N]}|Z_i| = O_\PP \big( N^{-1/2} \big).
\]
\end{Lem}

\begin{proof}
Write
\[
Z_i = \sum_{b=1}^m \sum_{j=1}^N A_{b i} A_{b j} Y\bj\ti = \sum_{\lambda=1}^t (-1)^{\lambda+1} \Gamma_\lambda\ti \sum_{b=1}^m \sum_{j=1}^N A_{b i} A_{b j} z\bj^{[\lambda]}.
\]
It suffices to bound each summand
\[
Z_i^{(\lambda)} \coloneqq  \sum_{b=1}^m \sum_{j=1}^N A_{b i} A_{b j} z\bj^{[\lambda]}.
\]
To bound \(\|Z_i^{(\lambda)}\|_{L^2}\), we expand
\[
\mathbb{E} \big[(Z_i^{(\lambda)})^2 \big]
= \sum_{\substack{b,b'=1}}^m \sum_{\substack{j,j'=1}}^N \mathbb{E} \Big[A_{b i} A_{b j} z\bj^{[\lambda]} A_{b' i} A_{b' j'} z_{b' \to j'}^{[\lambda]} \Big].
\]
Using again Theorem \ref{thm:Wick4}, this expectation is given by a sum over all pairings of Gaussian variables, which include the explicit factors \(A_{b i}\), \(A_{b j}\), \(A_{b' i}\), \(A_{b' j'}\) as well as those appearing inside \(z\bj^{[\lambda]}, z_{b' \to j'}^{[\lambda]}\).
The presence of the explicit factors fixes both a row index (\(b\)) and two column indices (\(i,j\)). Correspondingly, the number of free summation indices is reduced by two compared to the bulk message expansions of Lemma~\ref{lem:Lp_sharp}.
From the combinatorial analysis of Lemma~\ref{lem:boundary-Lp-corrected}, we deduce that
\[
\|Z_i^{(\lambda)}\|_{L^2} \le C \frac{\|y\|_2}{N^{3/2}} \cdot \text{Poly}(t,\delta)\leq \frac{C}{N} \text{Poly}(t,\delta),
\]
where $\text{Poly}(t,\delta)$ is a polynomial function of \(t\) and \(\delta\).
Hence, $\|Z_i^{(\lambda)}\|_{L^2} = O\left( N^{-1} \right)$.
By Nelson's hypercontractivity and Markov's inequality, this implies $Z_i^{(\lambda)} = O_\PP \left( N^{-1} \right).$
Summing over finitely many \(\lambda =1, \ldots, t\) preserves this bound.
\end{proof}


\section{Exact computations with a Gaussian initial distribution II: AMP equations}\label{sect:exactGaussian-bis}

We are now ready to derive the approximate message passing equations using the decomposition of the messages of Proposition \ref{prop:AMP_dec}. This completes the proof of Theorem \ref{TH:main} in the case of a Gaussian initial condition.
\begin{Prop} \label{cor:finitebeta-component}
Assume that $\nu\ia^{(0)}=\phi_{0,v^{(0)}}$ for all $i\in[N], a\in[m]$.
Let $\Delta\ti \coloneqq \delta/(2\beta v\ti +\delta)$, where $v\ti $ is the (edge-independent) variance from Lemma~\ref{lem:variance}, and recall
$\Gamma_\lambda\ti =\prod_{\tau=1}^{\lambda}\Delta^{(t-\tau)}$ from~\eqref{eq:defGamma}.
It holds that
\be
x\ia\tti=X_i\ti+O_\mathbb{P}(t\Delta\ti\Gamma_1\ti N^{-1/2})\qquad \forall a\in[m]\,,
\ee
where the sequence $\{X_i\ti\}_{t\in\N}$ satisfies the AMP equations
\be\label{eq:AMP-cor}
X_i\tti= X_i\ti (\beta)
+\Delta\ti  \left(\sum_{b=1}^m y_bA_{bi}-\sum_{b=1}^m\sum_{j=1}^N A_{bi}A_{bj} X_j\ti (\beta)\right)+O_\PP \Big(\Delta\ti  \Gamma\ti _{1} N^{-1/2}\Big)\,.
\ee
\end{Prop}

\begin{proof}
    We start by decomposing the message passing equations \eqref{eq:MP-gauss} as 
    \begin{align}
        x\ia\tti &= \Delta \ti \sum_{b=1}^m y_b A_{bi} - \Delta\ti y_a A_{ai} \label{eq:parte1} \\
        & -\Delta\ti \sum_{b=1}^m \sum_{j=1}^N A_{bi} A_{bj}x\jb\ti + \Delta\ti A_{ai} \sum_{j=1}^N A_{aj} x\ja\ti \label{eq:parte2} \\
        &+ \Delta\ti \sum_{b=1}^m A_{bi}^2 x\ib\ti - \Delta\ti A_{ai}^2 x\ia\ti \label{eq:parte3}.
    \end{align}
We recognise that some of the terms above do not depend on the index $a$. Comparing with the decomposition of Proposition \ref{prop:AMP_dec}, we set
\be
X_i\ti (\beta)\coloneqq \sum_{\lambda=1}^t(-1)^{\lambda+1}\Gamma_\lambda\ti  x_i^{[\lambda]},
\qquad
Y\ai \ti (\beta)\coloneqq \sum_{\lambda=1}^t(-1)^{\lambda+1}\Gamma_\lambda\ti  z\ai ^{[\lambda]}\,,
\ee
so that
\be\label{X+Y}
x\ia\ti =X_i\ti (\beta)+Y\ai \ti (\beta)\,
\ee
and (recall that $\Gamma_\lambda\ti\leq\Gamma_1\ti$ for all $t$ and for all $\lambda\geq 1$)
\be\label{eq:Ypiccolo}
\max_{i\in [N]} \max_{a\in [m]}|Y\ai \ti (\beta)|=O_\PP(t\Gamma_1\ti N^{-1/2})\,. 
\ee
We now replace $x\ia\ti$ on the l.h.s. of \eqref{eq:parte1} with the r.h.s. of \eqref{X+Y} and similarly for $x\jb\ti$ (suitably changing the index $a$ into $b$ and $i$ into $j$) in (\ref{eq:parte2}) and (\ref{eq:parte3}). Next we separate all the terms depending on the index $a$ and all those where this index does not appear. 

The $a$-independent
terms give \eqref{eq:AMP-cor} with the explicit remainder
\begin{equation}\label{eq:R-explicit}
R_i\ti (N,\beta)
=\Delta\ti  \left(
\sum_{b=1}^m A_{bi}^2 Y\ib\ti (\beta)-\sum_{b=1}^m\sum_{j=1}^N A_{bi}A_{bj} Y_{j\to b}\ti (\beta)\right)\,. 
\end{equation}
By \eqref{eq:Ypiccolo} and $\sum_{b}A^2_{bi}\approx 1$ it follows that
\be
\sum_{b=1}^m A_{bi}^2 Y\ib\ti (\beta)=O_\PP\big(\Gamma_1\ti N^{-1/2}\big)\,.
\ee
Moreover, by Lemma \ref{lem:double_sum_Y} 
\begin{equation*}
    \sum_{b=1}^m\sum_{j=1}^N A_{bi}A_{bj} Y_{j\to b}\ti (\beta) = O_\PP(\Gamma_1\ti N^{-1/2}),
\end{equation*}
and this recovers \eqref{eq:AMP-cor}.
\end{proof}

\begin{Rem}
We stress that the only $\beta$-dependence on the r.h.s. of \eqref{eq:AMP-cor} is given by the factors
$\Delta\ti $ and $\Gamma_1\ti =\Delta^{(t-1)}$. For the behaviour as $\beta\to\infty$ (taken after $N\to\infty$), see Section \ref{sect:AnaAMP}. 
\end{Rem}


\section{Gaussian approximations for $t=1$}\label{sect:GaussianI}

In this section we will work with a general $\ell_q$ prior of the form
\be
\pi_{\b,q}(s)\coloneqq Z'(\b,q)e^{-\b|s|^q}\,,\qquad \b>0,\quad s>0\,
\ee
where $Z'(\b,q)$ is a normalisation constant.

We set
\be\label{eq:musigmat=1}
\mu\ia\oone \coloneqq \frac{\sum_{b\neq a} \frac{\hat{x}\bi\zzero}{\hat{v}\bi\zzero}}{\sum_{b \neq a} \frac{1}{\hat{v}\bi\zzero}}, \qquad 
\sigma\ia\oone \coloneqq \frac{1}{\sum_{b \neq a} \frac{1}{\hat{v}\bi\zzero}}.
\ee

Let us define for brevity
\be\label{eq:average}
\mathtt{A}[F]_{q,\beta,i\to a}\coloneqq \frac{\int  \pi_{\b, q}(s)F(s)\phi_{\mu\ia\oone, \s\ia\oone}(s)\dd s}{\int \pi_{\b, q}(s)\phi_{\mu\ia\oone, \sigma\ia\oone}(s) \dd s}
\ee
The main result of this section is the following.
\begin{Prop}\label{prop:Gaussian-approx}
Assume that $\nu^{(0)}\ia $ is bounded and continuous with 
$$\int|s|^4\nu^{(0)}\ia (s)\dd s<\infty\,.$$
Let $q>0$ and consider the BP equations (\ref{eq:BP}) with initial conditions $\{\nu\zzero\ia\}_{i\in[N], a\in[m]}$. 
For all $F$ such that 
\be\label{eq:mathttF}
\int \pi_{\b,q}(s)|F(s)|^2\phi_{\mu,\s}(s)\dd s=C_{q,\b,\mu,\s}<\infty\,,
\ee
we have
\be\label{eq:Gauss-appro}
\left|\int \nu\ia \oone (s)F(s)\dd s-\mathtt{A}[F]_{q,\b,i\to a}\right|
\leq \mathtt{A}[|F|]_{q,\b,i\to a}O_\P\parav{\frac{1}{ N^{\frac12-2\e}}}\,.
\ee
Moreover, $\nu\ia \oone$ depends continuously on the entries of the matrix $A$. 
\end{Prop}

Let us set
\be\label{eq:cum-t-3}
P\ai \coloneqq \sum_{j\neq i}\frac{\dd^3}{\dd r^3}\log\int \pi_{\b,q}(s)\nu^{(0)} \ja (s)e^{r\left(\frac{y_a}{A_{ai}(N-1)}-\frac{A_{aj}}{A_{ai}}s\right)}ds\,,
\ee
and
\be\label{eq:defEDG}
\edg\ai (s)\coloneqq \left(1+\frac{1}{3!}\frac{P\ai }{(\hat v\zzero \ia )^{\frac32}}H_3\left(\frac{s-\hat x\zzero \ia }{\sqrt{\hat v\zzero \ia }}\right)\right)\,,
\ee
where
$H_3(x)=x^3-3x$
is the third Hermite polynomial. 

To prove the main proposition here, the following result will be instrumental. The proof is rather technical and it is given in Appendix \ref{App:prooflemma}.

\begin{Lem}\label{lemma:Pprob}
We have
\be\label{eq:Pprob1}
\left|\sum_{b \neq a}\frac{\hat x\zzero\bi P\bi }{(\hat v\zzero\bi )^{3}}\right|= O_\P\parav{\frac{1}{N}}\,,\qquad\mbox{and}\qquad
\left|\sum_{b \neq a}\frac{(\hat x\zzero\bi )^2P\bi }{(\hat v\zzero\bi )^{3}}\right|\leq O_\P\parav{\frac{1}{\sqrt N}}
\ee
and
\be\label{eq:Pprob2}
\left|\sum_{b \neq a}\frac{P\bi }{(\hat v\zzero\bi )^{2}}\right|= O_\P\parav{\frac{1}{\sqrt N}}\,,\quad \left|\sum_{a\in[m]\setminus\{b\}}\frac{P\ai }{(\hat v\zzero\ia )^{3}}\right|= O_\P\parav{\frac{1}{N^{\frac32}}}\,.
\ee
\end{Lem}

\begin{remark}\label{rmk:K}
A direct consequence of \eqref{eq:Pprob1} is that
\be\label{eq:K}
\absv{\frac{\hat x\zzero\ai P\ai }{(\hat v\zzero\ai )^{3}}}= O_\P\parav{\frac{1}{N}}\,,\qquad \absv{\frac{(\hat x\zzero\ai )^2P\ai }{(\hat v\zzero\ai )^{3}}}= O_\P\parav{\frac{1}{\sqrt N}}\,. 
\ee
\end{remark}

\begin{Lem}\label{lemma:simpel}
There is a constant $c=c(\d, \nu^{(0)})$ such that the following holds: 
\bea
\mathbb{P}\left(\left|\sum_{j\neq i} A_{aj}x^{(0)}\ja  \right|\geq \l\right)&\leq& Ce^{-c\l^2}\,,\quad \l\geq0\label{eq:claimXk}\\
\mathbb{P}\left(\left|\sum_{j\neq i} \parav{A^2_{aj}-\frac{1}{m}}v^{(0)}\ja \right|\geq \l\right)&\leq& Ce^{-c\l^2m}\quad\mbox{$0\leq \l\ll1$}\label{eq:claimP2}\\
\mathbb{P}\left(\left|A^3_{ai}P\ai \right|\geq \l\right)&\leq& \begin{cases}
Ce^{-c\l^{2} m^2}&\quad\mbox{$\l\leq \frac1{m^{\frac34}}$}\\
Ce^{-c\l^{\frac23} m}&\quad\mbox{$\l\geq \frac1{m^{\frac34}}$}
\end{cases}\label{eq:claimP3}
\,.
\eea
\end{Lem}
\begin{proof}
These are tail probabilities of sums of i.i.d. r.vs. Recall that each $A_{aj}$ can be written as $X_{ai}/\sqrt m$ where $X_{ai}$ is a sub-Gaussian r.v. with unitary variance. 
The Hoeffding inequality yields \eqref{eq:claimXk} and the Bernstein inequality yields \eqref{eq:claimP2} (for $\l$ small enough). The inequality (\ref{eq:claimP3}) follows from Lemma \ref{prop:tail}. 
\end{proof}

\begin{remark} 
From (\ref{eq:claimXk}), (\ref{eq:claimP2}) it follows that $\mu\oone \ia$ and $\s\oone \ia$ are $O_\P(1)$.
\end{remark}

\begin{Lem}\label{lemma:multiplicative}
With probability at most $1-e^{-cN}$ the following holds. If $|s|\lesssim \sqrt {N}$ then there is a continuous function $\chi\ai :\mathbb{R}\rightarrow \mathbb{R}$ with $\lvert\chi\ai \lvert \lesssim \frac{1}{N^{\frac{3}{2}}}$ such that
\be\label{eq:multiplicative1}
\hat\nu^{(0)}\ai (s)=\phi_{\hat x^{(0)}\ai , \hat v^{(0)}\ai }\left(s\right)\edg\ai (s)\left(1+\chi\ai (s)\right)\,. 
\ee
The function $\chi\ai$ depends continuously from the entries of the matrix $A$. 
Moreover, $\sup_{s\in\R}\hat\nu^{(0)}\ai (s)\leq 1$.
\end{Lem}

\begin{proof}
In the first part of the proof, we proceed as in \cite[Lemma 5.2.1]{malekiT}.
Let us define for any $a\in[m]$ a sequence of independent r.vs $\{\xi^{(0)}\ia \}_{i\in[N]}$ whose distribution has density $\nu^{(0)}\ia $. Note that the variables $\{\xi^{(0)}\ia \}_{i\in[N]}$ are independent of the realisation of the matrix $A$, a property which is lost for $t\geq1$. 
For $a\in[m]$, we also set
$$
X_{aj}=\frac{y_a}{A_{ai}(N-1)}-\frac{A_{aj}}{A_{ai}}\xi^{(0)}\ja \,,\quad j\in[N]\setminus\{i\}\,.
$$
The distributions of the variables $X_{aj}$ then will have all the regularity and decay of the $\nu^{(0)}\ia $. Moreover it is clear that the distribution of $\sum_{j\neq i}X_{aj}$ is simply $\hat\nu^{(0)}\ai $. Recall that  $\hat\nu^{(0)}\ai $ has mean and variance given by the following expressions
\be
\hat x^{(0)}\ai=\frac{y_a}{A_{ai}}\,,\qquad \hat v^{(0)}\ai=A^{-2}_{ai}\sum_{j\neq i} A^2_{aj}\,.
\ee

These identifications allow us to apply the local central limit theorem in the form of Theorem \ref{lemma:local-limit-teorem} of Appendix \ref{app:LLT} (with $P_1=\hat x^{(0)}\ai $ and $P_2=\hat v^{(0)}\ai $). In particular Corollary \ref{cor;LLT} gives that for any $L>0$ there is a function $\chi\ai$ such that for all $|s-\hat x^{(0)}\ai |\leq L\sqrt {\hat v^{(0)}\ai }$ 
\be\label{eq:stima1}
\hat\nu^{(0)}\ai (s)=\phi_{\hat x^{(0)}\ai , \hat v^{(0)}\ai }\left(s\right)\edg\ai (s)\left(1+\chi\ai (s)\right)\,,
\ee
where 
\be\label{eq:stimachi-intermedio}
\chi\ai (s)\leq C\parav{\frac{N}{(\hat v^{(0)}\ai )^\frac52}+e^{-c \hat v^{(0)}\ai }}\,.
\ee
Moreover, $\chi\ai$ depends continuously on $\hat x^{(0)}\ai, \hat v^{(0)}\ai$ and $P\ai$, which are continuous in the entries of $A$. 

Next, we note that we can choose $L$ so large that there is $L'>0$ for which \eqref{eq:stima1} and \eqref{eq:stimachi-intermedio} hold for all $|s|\leq L'\sqrt {\hat v^{(0)}\ai }$\,. 
Indeed, the Hoeffding inequality (w.r.t. the law of the matrix $A$) gives
\be\label{eq:facille}
\sum_{j\neq i}A_{aj}\xi^{(0)}\ja =O_\P\parav{1}\,. 
\ee
Combining \eqref{eq:facille} and (\ref{eq:xx-hat}) and using that $A_{ai}$ is independent of the $\{A_{aj}\}_{j\neq i}$ we have that $|\hat x\ja ^{(0)}|\leq C\sqrt m$ with probability larger than $1-e^{-cm}$. This entails that with probability at least $1-e^{-cm}$ we can restrict ourself to consider $|s|\leq L'\sqrt{\hat v\ja ^{(0)}}$ where $L'>0$ is a large constant.

Here $\hat v^{(0)}\ai$ is still a random quantity that we can easily localise.
By \eqref{eq:v_ai} we have
\be
\absv{\hat v^{(0)}\ai -\frac{1}{mA^2_{ai}} \sum_{j\neq i}v^{(0)}\ja }=A^{-2}_{ai}\absv{ \sum_{j\neq i}A^2_{aj}v^{(0)}\ja -\frac{1}{m} \sum_{j\neq i}v^{(0)}\ja }
\ee
and
\be
\absv{ \sum_{j\neq i}A^2_{aj}v^{(0)}\ja -\frac{1}{m} \sum_{j\neq i}v^{(0)}\ja }\leq \t A^{2}_{ai}
\ee
with probability at least $1-e^{-c\min(\t \E[A^2_{ai}],\E[A^4_{ai}]\t^2) N}$ by the Bernstein inequality, the independence of $A_{ai}$ and $\hat v^{(0)}\ai $ and the Jensen inequality. 
Thus we set $\t=\e m$ for some $\e>0$ and we get (we abbreviate $V:=m^{-1}\sum_{}$)

\be\label{eq:fine-v}
\absv{\hat v^{(0)}\ai -\frac{V}{A^2_{ai}}}\leq \e m\,,
\ee
with probability larger than $1-e^{-c\e N}$. The bound \eqref{eq:fine-v} entails
\be\label{eq:fine-vv}
\hat v^{(0)}\ai = O_\P(m)\,. 
\ee
Using (\ref{eq:fine-vv}) into (\ref{eq:stimachi-intermedio}) we have that for all $|s|\leq L\sqrt N$
\be\label{eq:stimachi1}
\chi\ai (s)\leq C\parav{\frac{1}{N^\frac32}+e^{-c N}}\,,
\ee
which gives the bound on $\chi\ai$. 

Finally, the global bound $\sup_{s\in\R}\hat\nu^{(0)}\ai (s)\leq 1$ follows again by Theorem \ref{lemma:local-limit-teorem} for $N$ large enough. 
\end{proof}

When taking products, the Edgeworth term is the most delicate to handle. 

\begin{Lem}\label{lemma:EDG}
Let $\e>0, \a\in(0,\frac12)$. 
Then there is a constant $K_N>0$ and a polynomial function $\chi'$ with
\be
\sup_{|s|\leq N^{\a}}|\chi'(s)|=O_\P\parav{\frac{1}{N^{\frac12-\a}}}\,,
\ee
such that
\be
\prod_{b\neq a}\edg\bi (s)
 =K_N\left(1+\chi'\ai(s)\right)\,.
\ee 
Moreover, $\chi'\ai$ depends polynomially from the entries of $A$. 
\end{Lem}
\begin{proof}
We write
\bea
\frac{1}{3!}\frac{P\ai }{(\hat v^{(0)}\ai )^{\frac32}}H_3\left(\frac{s-\hat x^{(0)}\ai }{\sqrt{\hat v^{(0)}\ai }}\right)
&=&\frac{1}{3!}\frac{P\ai }{(\hat v^{(0)}\ai )^{3}}s^3
-\frac{\hat x^{(0)}\ai }{2}\frac{P\ai }{(\hat v^{(0)}\ai )^3}s^2\label{eq:somma2}\\
&+&\left(\frac{(\hat x^{(0)}\ai )^2}{2}\frac{P\ai }{(\hat v^{(0)}\ai )^{3}}-\frac{1}{2}\frac{P\ai }{(\hat v^{(0)}\ai )^{2}}\right)s\label{eq:somma3}\\
&-&\frac{(\hat{x}^{(0)}\ai )^2}{6}\frac{P\ai }{(\hat{v}^{(0)}\ai )^3}-\frac{\hat{x}\ai ^{(0)}}{2}\frac{P\ai }{(\hat{v}\ai ^{(0)})^2}\label{eq:somma4}\,. 
\eea

Set for brevity
$$
S\ai (s)\coloneqq \eqref{eq:somma2}+\eqref{eq:somma3}\,,\qquad K\ai \coloneqq \eqref{eq:somma4}\,.
$$
By Lemma \ref{lemma:Pprob} and triangular inequality we have
\be\label{eq:affatica-1}
\absv{\sum_{b\neq a}S\bi }\lesssim 
\frac{|s|^3}{N^{\frac32-\e}}+
\frac{s^2}{N^{1-\e}}+\frac{|s|}{N^{\frac12-\e}}
\ee
and therefore
\be\label{eq:affatica}
\sup_{|s|\leq N^{\a}}\absv{\sum_{b \neq a}S\bi }\lesssim 
\frac{1}{N^{\frac12-\a-\e}}
\ee
with probability greater than $1-e^{-N^\e}$.

Moreover by Remark \ref{rmk:K} and the triangular inequality, we obtain that on the same event of probability larger than $1-e^{-cN^{\e}}$ it holds that
\be\label{eq:affatica-aridaje}
|K\ai |\leq \frac{C}{N^{\frac{1}2-\e}}\,.
\ee

We work from now on on that event of probability larger than $1-e^{-cN^{\e}}$. 
Next, we note that thanks to \eqref{eq:affatica-aridaje} we can assume that for $N$ large enough $\inf_{b\neq a}|1+K\bi |>\frac12$. We write
\be
\prod_{b\neq a}\left(1+K\bi +S\bi (s)\right)
=\prod_{b\neq a}\left(1+K\bi \right)\prod_{b\neq a}\left(1+\frac{S\bi }{1+K\bi }\right)\nn\,.
\ee
We identify the constant $K_N$ with the first product on the r.h.s. above: 
\be
K_N\coloneqq \prod_{b\neq a}\left(1+K\bi \right)\,.
\ee
Also, we write
\be
\prod_{b\neq a}\left(1 +\frac{S\bi (s)}{1+K\bi }\right)=:1+\chi'(s)\,,
\ee
where
\bea
\sup_{|s|\leq N^{\a}}|\chi'(s)|&\leq& \sup_{|s|\leq N^{\a}}\sum_{k\in[m]}\left(\sum_{b\,:\,|1+K\bi |\geq1/2}\left|\frac{S\bi }{1+K\bi }\right|\right)^k\nn\\
&\leq&\sum_{k\geq1}\left(2\sup_{|s|\leq N^{\a}}\sum_{a\neq b}\left|S\bi \right|\right)^k=O_\P\parav{\frac{1}{N^{\frac12-\a}}}\,,\nn
\eea
by (\ref{eq:affatica}). Finally the fact that $\chi'\ai$ depends polynomially on the entries of $A$ is clear from the construction above.
\end{proof}

\begin{Lem}\label{lemma:prod}
Let $\a\in(0,\frac12)$.
There is a continuous function
$\Tilde{\chi}\ai :\mathbb{R} \rightarrow \mathbb{R}$ such that 
\be
\Tilde{\chi}\ai = O_\P\parav{\frac{\delta}{N^{\frac12-\a}}}\,,
\ee 
a constant $K_{ib}(m,N)>0$ and a function $|\Upsilon_{m,N}|\leq 1$ such that
\be\label{eq:prod}
\prod_{b \neq a}\hat\nu^{(0)}\bi (s)=K_{ia}(m,N)\phi_{\mu\oone _{\ia}, \s\oone _{\ia}}(s)\left(1+\Tilde{\chi}_{\ai}\right)\ind{|s|\leq LN^{\a}}+\Upsilon_{m,N}\ind{|s|\geq N^{\a}}\,.
\ee
Moreover, $\Tilde\chi\ai$ depends continuously from the entries of $A$. 
\end{Lem}
\begin{proof}
By Lemma \ref{lemma:multiplicative},
 we write for $\a\in(0,1/2)$
\be
\hat\nu^{(0)}\ai (s)=\phi_{\hat x^{(0)}\ai , \hat v^{(0)}\ai }\left(s\right)\edg\ai (s)\left(1+\chi\ai (s)\right)1_{\{|s|\leq N^{\a}\}}+\hat\nu^{(0)}\ai (s)1_{\{|s|\geq N^{\a}\}}\,,
\ee
where the functions $\chi\ai (s)$ are bounded by $1/N^{3/2}$ for $|s|\leq N^{\a}$. 
Thus 
\bea
\prod_{b \neq a}\hat\nu^{(0)}\bi (s)&=&\prod_{b \neq a}\phi_{\hat x^{(0)}\bi , \hat v^{(0)}\bi }\left(s\right)\edg\bi (s)\left(1+\chi\bi (s)\right)1_{\{|s|\leq N^{\a}\}}\label{eq:NONUpsilon}\\
&+&\prod_{b\neq a}\hat\nu^{(0)}\bi (s)1_{\{|s|\geq N^{\a}\}}\label{eq:Upsilon}\,.
\eea
In the term (\ref{eq:Upsilon}) we simply bound each term of the product by one, thanks to Lemma \ref{lemma:multiplicative}. We get
\be
|\eqref{eq:Upsilon}|\leq 1_{\{|s|\geq N^{\a}\}}\,. 
\ee
Thus we have to evaluate the product in (\ref{eq:NONUpsilon}). First, bearing in mind \eqref{eq:musigmat=1}, we compute
\bea
&&\prod_{b \neq a}\phi_{\hat x^{(0)}\bi , \hat v^{(0)}\bi }\left(s\right)\nn\\
&=&\prod_{b \neq a}(2\pi \hat v^{(0)}\bi )^{-\frac 12}\exp\left(\frac12\left(\frac{(\mu\oone \ib )^2}{\s\oone \ib}-\sum_{b \neq a}\frac{(\hat x^{(0)}\bi )^2}{\hat v^{(0)}\bi }\right)\right)
\sqrt{2\pi\s\oone \ib }\phi_{\mu\oone \ib , \s\oone \ib }(s)\nn\\
&=:&\widetilde K_{ib}(m,N)\phi_{\mu\oone \ib , \s\oone \ib }(s)\,. 
\eea
Moreover,
\be
\prod_{b \neq a}\left(1+\chi\bi (s)\right) =: 1+{\chi}''\ai(s) 
\ee
where clearly $|{\chi}''\ai|\lesssim\sum_{b\neq a}\chi\bi\lesssim \frac{\delta}{\sqrt N}$.
Finally the product of $\edg\ai $ terms is evaluated by Lemma \ref{lemma:EDG}. 
To conclude, we set
\be
(1+\chi\ai')(1+{\chi}''\ai)=:1+\Tilde{\chi}\ai\,, 
\ee
where $\chi'$ was estimated in Lemma \ref{lemma:EDG}. 
\end{proof}

\begin{proof}[Proof of Proposition \ref{prop:Gaussian-approx}]

Combining Lemma \ref{lemma:prod} and the BP equations \eqref{eq:BP} we have
\be\label{eq:reprt=1}
\int F(s)\nu\ib\oone \dd s=\frac{\mathtt A[F]_{2,\b,i\to b}+\mathtt A[F\Tilde{\chi}\ai]_{2,\b,i \to b}+\D(F)}{1+\mathtt A[\Tilde{\chi}\bi]_{2,\b, i\to b }+\D(1)}\,,
\ee
where 
\be\label{eq:defDelta}
\Delta(F)\coloneqq \frac{\int_{|s|\geq N^{\a}}\pi_{\b, q}(s) F(s)\left(\Upsilon_{m,N}+K_{ib}(m,N)(1+\Tilde{\chi}\bi (s))\phi_{\mu\ti\ib , \s\ti\ib }(s)\right)\dd s}{K_{ib}(m,N)\int  \pi_{\b, q}(s) \phi_{\mu\ti\ib, \sigma\ti\ib}\dd s}\nn\,.
\ee
Since $\Upsilon_{m,N}$ is bounded and $\tilde{\chi}\bi$ is small, we have that for some absolute constants $c,C>0$
\be\label{eq:deltabound}
|\D(F)|\leq C \mathtt A[F1_{\{|s|\geq N^\a\}}]\leq \sqrt{\mathtt A[F^2]} e^{-c\b N^{\a q}}\,.
\ee
Therefore
\bea
\absv{\int F(s)\nu\ib \oone ds-\mathtt A[F]}&\leq& \absv{\mathtt A[F\tilde \chi\bi]-\mathtt A[F]\mathtt A[\tilde \chi\bi]}+|\D(F)|+|\D(1)|\nn\\
&\leq& \frac{C}{N^{\frac12-\a-\e}}\mathtt A[|F|]+Ce^{-c\b N^{\a q}}\,. \nn
\eea
with probability at least $1-e^{-N^{\e}}$. Choosing $\a=\e$ we recover the assertion. 
\end{proof}
\begin{corollary}\label{cor:analiticity}
For the same choice of the element of the matrix $A$ of Proposition \ref{prop:Gaussian-approx},
the function $\l\in\R\to\int e^{i\l s}\nu\ib \oone \dd s$ is analytic and it depends continuously on the entries of the matrix $A$. 
\end{corollary}
\begin{proof}
The continuous dependence on the entries of $A$ follows directly from the analogue property of $\nu\ib \oone$ given by Lemma \ref{lemma:prod}.

Consider $F(s)=e^{cs^2}$ for $c>0$ suitably small. Starting by (\ref{eq:reprt=1}) (with the same definition of $\D$ as \eqref{eq:defDelta}) we have
\bea
\int e^{cs^2}\nu\ib \oone \dd s&=&\frac{\mathtt A[e^{cs^2}]_{2,\b, i\to b }+\mathtt A[e^{cs^2}\tilde\chi]_{2,\b, i\to b }+\D(e^{cs^2})}{1+\mathtt A[\chi]_{2,\b, i\to b }+\D(1)}\,,\nn\\
&\leq&\mathtt A[e^{cs^2}]_{2,\b, i\to b }+\mathtt A[e^{cs^2}\tilde\chi]_{2,\b, i\to b }+\D(e^{cs^2})\nn\\
&\leq&3\mathtt A[e^{cs^2}]_{2,\b, i\to b }
\eea
for $N$ sufficiently large. We used that $\tilde\chi$ is small as $N\to\infty$ and \eqref{eq:deltabound}, since $A[e^{2cs^2}]_{2,\b, i\to b }$ is finite for $c$ small enough. The density $\nu\ib \oone $ must then decay at least exponential fast at infinity, whence the first part of the assertion. 
\end{proof}


\section{Successive Gaussian approximations}\label{sect:successive}

In this section we will work with the $\ell_2$ prior, that is
\be
\pi_{\b,2}(s)\coloneqq Z(\b,2)e^{-\b|s|^2}\,,\qquad \b>0,
\ee
where $Z(\b,2)$ is a normalisation constant. 
In analogy with \eqref{eq:musigmat=1} we put
\be\label{eq:musigmat}
\mu\ti\ib \coloneqq \frac{\sum_{a\neq b}\frac{\hat x^{(t-1)}\ia }{\hat v^{(t-1)}\ia }}{\sum_{a\neq b}\frac{1}{\hat v^{(t-1)}\ia }} \,,\qquad \s\ti\ib \coloneqq  \left(\sum_{a\neq b}\frac{1}{\hat v^{(t-1)}\ia }\right)^{-1}
\ee
and
$$
[F]_{\mu\ia\ti, \s\ia\ti}:=\frac{\int  \pi_{\b, 2}(s)F(s)\phi_{\mu\ia\ti, \s\ia\ti}(s)\dd s}{\int \pi_{\b, 2}(s)\phi_{\mu\ia\ti, \sigma\ia\ti}(s) \dd s}\,. 
$$
The main result of this section follows. 

\begin{proposition}\label{prop:gaussapprox-t}
Let $k\in\N$. Set
\be
\mathcal X\ia\ti:=\int ds s^k\nu\ia \ti(s)-[s^k]_{\mu\ti\ia, \s\ti\ia}\,.
\ee
For all $t\geq2$ it holds that $\mathcal X\ia\ti$ depends continuously on the entries on $A$ and
\be
\mathcal X\ia\ti= O_\P\parav{\frac{k!}{\sqrt N}}\,.
\ee
\end{proposition}

The key step towards the proof of Proposition \ref{prop:gaussapprox-t} is the following. 

\begin{Lem}\label{lemma:approx=Gauss=t}
Let $\l\in\R$ and
\be\label{eq:defupsilont}
\upsilon\ia \ti(\l)\coloneqq \int dse^{i\l s}\nu\ia \ti(s)-[e^{i\l s}]_{\mu\ti_{ia}, \s\ti_{ia}}
\ee
Then with probability $1-e^{-N^\e}$ it holds that 
\be\label{eq:Gauss-Fourier}
\sup_{\l\in\R}\absv{\upsilon\ia \ti(\l)}
\leq \frac{1}{ N^{\frac12-\e(t)}}\,,
\ee
where $\e(t)\coloneqq \e\sum_{k=0}^{t-1}2^{-k}$.
\end{Lem}
\begin{proof}
Let us deal first with the case $t=1$. Choosing $F(s)=e^{i\l s}$ in Proposition \ref{prop:Gaussian-approx} gives (\ref{eq:Gauss-Fourier}). 

Assume now that \eqref{eq:Gauss-Fourier} holds true for some $t\geq 2$. It follows then by Theorem \ref{prop:LLTPII} that
\be\label{eq:phi+chi}
\hat\nu\ti\ai (s)=\phi_{\hat x\ti\ai , \hat v\ti\ai }(s)+ \chi\ti\ai (s)\,,
\ee
where for all $r\in(0,1)$
\be\label{eq:stimachi1-s6}
\sup_{s\in\R}|\chi\ti\ai (s)|\lesssim \frac{1}{N^{\frac12-\e_t+r}}\,.
\ee
For a reason that will appear clear below we choose $r=1-\e2^{-t}$.
Combing the second BP equation \eqref{eq:BP} with \eqref{eq:phi+chi} we get
\be
\nu\ib ^{(t+1)}\simeq \pi_{2,\b}(s)\prod_{a\in[m]\setminus \{b\}}\phi_{\hat x\ti\ai , \hat v\ti\ai }+\pi_{2,\b}(s)\sum_{k=1}^{m-1}\sum_{\substack{I\subset[m]\setminus \{b\}\\|I|=k}}\prod_{a\in I}\chi\ti\ai \prod_{a\notin I}\phi_{\hat x\ti\ai , \hat v\ti\ai }\,. 
\ee

Thus we have
\be\label{eq;by}
\int e^{i\l s}\nu\ib ^{(t+1)}ds=\frac{[e^{i\l s}]_{\mu\ti\ib , \s\ti\ib }+\D_1\ti}{1+\D_2\ti}\,\,,
\ee
where
\bea
\D_1\ti&\coloneqq &\sum_{k=1}^{m-1}\sum_{\substack{I\subset[m]\setminus \{b\}\\|I|=k}}\frac{\int \pi_{\b,2}(s)e^{i\l s}\prod_{a\notin I}\phi_{\hat x\ti\ai , \hat v\ti\ai }(s)\prod_{a\in I}\chi\ti\ai (s)ds}{\int \pi_{\b,2}(s)\prod_{a\neq b}\phi_{\hat x\ti\ai , \hat v\ti\ai }(s)}\,,\label{eq:delta1}\\
\D_2\ti&\coloneqq &\sum_{k=1}^{m-1}\sum_{\substack{I\subset[m]\setminus \{b\}\\|I|=k}}\frac{\int \pi_{\b,2}(s)\prod_{a\notin I}\phi_{\hat x\ti\ai , \hat v\ti\ai }(s)\prod_{a\in I}\chi\ti\ai (s)ds}{\int \pi_{\b,2}(s)\prod_{a\neq b}\phi_{\hat x\ti\ai , \hat v\ti\ai }(s)}\label{eq:delta2}\,.
\eea
It is helpful to introduce the following notation: for any function $F$ we set
\be\label{eq:defmeanv}
\meanv{F}_I\coloneqq \frac{\int \pi_{\b,2}(s)\prod_{a\notin I}\phi_{\hat x\ti\ai , \hat v\ti\ai }(s)F(s)ds}{\int \pi_{\b,2}(s)\prod_{a\notin I}\phi_{\hat x\ti\ai , \hat v\ti\ai }(s)}\,.
\ee
Thus
\be\label{eq:delta2II}
\D_2\ti=\sum_{k=1}^{m-1}\sum_{\substack{I\subset[m]\setminus \{b\}\\|I|=k}}\frac{\meanv{\prod_{a\in I}\chi\ti\ai }_I}{\meanv{\prod_{a\in I}\phi_{\hat x\ti\ai , \hat v\ti\ai }}_I}\,. 
\ee
By \eqref{eq:stimachi1} we have
\be\label{eq:boundChi}
\meanv{\prod_{a\in I}\chi\ti\ai }_I\lesssim \parav{\frac{1}{N^{\frac32-\e_t-\e2^{-t}}}}^{|I|}=\parav{\frac{1}{N^{\frac32-\e_{t+1}}}}^{|I|}\,. 
\ee
In addition, we write
\be\label{eq:addizione}
\meanv{\prod_{a\in I}\phi_{\hat x\ti\ai , \hat v\ti\ai }}_I=K_I\meanv{\phi_{\mu\tti_I,\s\tti_I}}_I\,,
\ee
where 
\be\label{eq:musigmatI}
\mu^{(t+1)}_{I}\coloneqq \frac{\sum_{a\in I}\frac{\hat x\ti\ia }{\hat v\ti\ia }}{\sum_{a\in I}\frac{1}{\hat v\ti\ia }} \,,\qquad \s^{(t+1)}_{I}\coloneqq  \left(\sum_{a\in I}\frac{1}{\hat v\ti\ia }\right)^{-1}\,
\ee
and $K_I$ is a normalisation constant. 
Let now $K'_I$ be a number such that
\be\label{eq:gegeben}
\min_{|s|\leq K'_I} \phi_{\mu^{(t+1)}_I,\s^{(t+1)}_I}\geq \frac{2}{K_I}\,.
\ee
This gives
\be\label{eq:K'I}
K'_I\leq \mu^{(t+1)}_{I}+\sqrt{2\s^{(t+1)}_{I}\log\parav{\frac{K_I}{2\sqrt{2\pi \s^{(t+1)}_{I}}}}}\,. 
\ee
By \eqref{eq:addizione} and \eqref{eq:gegeben}
\be\label{eq:lowKi}
\meanv{\prod_{a\in I}\phi_{\hat x\ti\ai , \hat v\ti\ai }}_I\geq 2\parav{1-\meanv{1_{\{|s|\geq K_I'\}}}_I}\,
\ee
and we can estimate by (\ref{eq:K'I}) and the definition (\ref{eq:defmeanv})
\be\label{eq:betalargenough}
\meanv{1_{\{|s|\geq K_I'\}}}_I\leq \exp\parav{-\b \mu^{(t+1)}_I-(\b-1+(2\s^{(t+1)}_{I^c})^{-1})\frac{|\mu^{(t+1)}_{I^c}|^2}{2\s^{(t+1)}_{I^c}}}\leq \frac{9}{10}
\ee
for $\b$ large enough. 

Using (\ref{eq:boundChi}), (\ref{eq:lowKi}) and (\ref{eq:betalargenough}) into \eqref{eq:delta2} we obtain
\bea
\absv{\D_2\ti}&\leq& C\sum_{k=1}^{m-1}\sum_{\substack{I\subset[m]\setminus \{b\}\\|I|=k}}\parav{\frac{1}{N^{\frac32-\e_{t+1}}}}^{k}\nn\\
&\leq& C\sum_{k=1}^{m-1}\binom{m-1}{k}\parav{\frac{1}{N^{\frac32-\e_{t+1}}}}^{k}\nn\\
&=&C\parav{\parav{1+\frac{1}{N^{\frac32-\e_{t+1}}}}^{m-1}-1}\nn\\
&\lesssim& \frac{1}{N^{\frac12-\e_{t+1}}}\,. \label{eq"D_2S1}
\eea
It follows that for $N$ large enough
\be\label{eq"D_2S2}
|1+\D_2\ti|\geq \frac12\,. 
\ee

Similarly, since $|e^{i\l s}|\leq1$, we have
\be\label{eq"D_1S1}
\absv{\D_1\ti}\lesssim \frac{1}{N^{\frac12-\e_{t+1}}}\,. 
\ee

By \eqref{eq;by}
\bea
\absv{\int e^{i\l s}\nu\ib ^{(t+1)}ds-[e^{i\l s}]_{\mu\ti\ib , \s\ti\ib }}&=&\frac{\absv{\D_1\ti-[e^{i\l s}]_{\mu\ti\ib , \s\ti\ib }\D_2\ti}}{|1+\D_2\ti|}\nn\\
&\leq&2 \absv{\D_1\ti}+2\absv{\D_2\ti}\nn\\
&\lesssim& \frac{1}{N^{\frac12-\e_{t+1}}}\,. 
\eea
This proves \eqref{eq:Gauss-Fourier}.
\end{proof}

\begin{Lem}\label{lemma:anal-t}
On the same event in which Corollary \ref{cor:analiticity} holds, the function
$\l\in\R\mapsto\upsilon\ia \ti(\l)$ is an analytic function. 
\end{Lem}
\begin{proof}
The proof is done by induction over $t\in\N$.
The analyticity of $\upsilon\ia \oone $ follows from Corollary \ref{cor:analiticity}.
Assume now that $\upsilon\ia \ti$ is an analytic function. Then also
$$
\l\in\R\mapsto \int e^{i\l s}\nu\ib \ti ds
$$
must be so, that is $\nu\ib \ti$ decays at least exponentially for large arguments. It follows by the first BP equation (\ref{eq:BP}) that also $\hat\nu\bi \ti$ has that decay and in particular it must be bounded by a constant smaller than one by \eqref{eq:phi+chi} and \eqref{eq:stimachi1}. Therefore by the second BP equation (\ref{eq:BP}) we see that $\nu\ia ^{(t+1)}$ inherits the decay of the $\ell_2$-prior $\pi_{2,\b}$, hence its Fourier transform is analytic. We conclude that also $\upsilon\ia ^{(t+1)}$ is an analytic function and the assertion follows by induction over $t\in\N$. 
\end{proof}

\begin{proof}[Proof of Proposition \ref{prop:gaussapprox-t}]
Combine Lemma \ref{lemma:approx=Gauss=t}, Lemma \ref{lemma:anal-t} and Lemma \ref{lemma:unmomento} in Appendix \ref{app:LLT}. Remark also that since in the proof of Lemma \ref{lemma:anal-t} we proved that $\nu\ia ^{(t+1)}$ has an Gaussian decay we can set $c=1$ as the constant appearing in the statement of Lemma \ref{lemma:unmomento}. Finally, continuous dependance on the entries of $A$ follows from Lemma \ref{lemma:prod} for $t=1$ and it is propagated by the BP equations (\ref{eq:BP}) for all $t\geq2$.
\end{proof}


\

\section{Two shadowing sequences}\label{sect:shadow}

From now on we denote by $\bar{m}\ia\ti$ and $\bar{s}\ia\ti$ the mean and variance updates of the BP-iteration \eqref{eq:BP} when they are initialised with a centred Gaussian distribution $\nu\ia\zzero=\phi_{0, v\zzero}.$ These were computed in the sections \ref{sect:exactGaussian}-\ref{sect:exactGaussian-bis}. We proved in Lemma \ref{lem:variance} that the variances $\bar{s}\ia\ti$ are constant as $a,i$ varies in $[m],[N]$ respectively. They are all fixed to some value
\be\label{eq:def-bars^t}
\bar{s}\ti=\frac{v\zzero(1-\d)}{\d^t(1-\d)+2\b v\zzero(1-\d^t)}\,. 
\ee 

Here we are interested in writing the actual means and variances obtained by the BP iteration starting by an arbitrary distribution in terms of these quantities, namely 
\begin{equation*}
    x\ia\ti=\bar{m}\ia\ti+\text{error}\quad \text{and}\quad v\ia\ti=\bar{s}\ia\ti+\text{error}.
\end{equation*}
To be precise, the main result of the section is the following.

\begin{Prop}\label{prop:shadow}
Consider initial conditions for the BP equations \eqref{eq:BP} $\nu^{(0)}_{i\to a}$, that is bounded and continuous densities with with finite fourth moment and $x\ia\zzero=0$, $v\ia\zzero=v^{(0)}>0$.
 Let $\bar s^{(t)}$ and $\bar m^{(t)}_{i\to a}$ as explained above. 
Fix $t\in\mathbb{N}$ and $\e>0$. 
Then, for all $(i,a)\in[N]\times[m]$ we have
\be\label{eq:shadow-V}
v\ia^{(t)} \;=\; \bar s^{(t)} \;+\; \ERR^{(t)}_2 \;+\; \err^{(t)}_2,
\qquad
\err^{(t)}_2 = O \big(N^{-1/2+2\varepsilon}\big),\quad
\ERR^{(t)}_2 = O \big(\beta^{-\alpha} N^{-1/2+2\varepsilon}\big)\,,
\ee
and 
\be\label{eq:shadowX}
x^{(t)}_{i\to a} \;=\; \bar m^{(t)}_{i\to a} \;+\; \ERR^{(t)}_{i\to a} \;+\; \err^{(t)}_1,
\qquad
\ERR^{(t)}_{i\to a} = O_\PP \big(N^{-1/2+2\varepsilon}\big),\quad
\err^{(t)}_1 = O_\PP \big(N^{-1/2+2\varepsilon}\big).
\ee
\end{Prop}

The central tool to prove this proposition is the introduction of the following two sequences
(compare with \eqref{eq:x_ia} and \eqref{eq:v_ia})
\be\label{eq:shadow}
    m\ia\ti := \left(\sum_{b\neq a}\frac{\hat{x}\bi^{(t-1)}}{\hat{v}\bi^{(t-1)}} \right) s\ia\ti \quad\text{and}\quad s\ia\ti:=\left(2\b+\sum_{b\neq a}\frac{1}{\hat{v}\bi ^{(t-1)}} \right)^{-1}\,,
\ee

where we recall that
\bea
\hat{x}\ai \ti&=&x\ia \ti+\frac{y_a-\sum_{j}A_{aj}x\ti\ja }{A_{ai}}\label{eq:xx-hatBIS}\,,\\
\hat{v}\ai \ti& =& \frac{\sum_{k \neq i} A_{ak}^2 v_{\ka}\ti}{A_{ai}^2}\,.\label{eq:v_hatBIS}
\eea

The sequence $m\ia\ti$ and $s\ia\ti$ would be respectively the mean and the variance of $\nu\ia\ti$ if all the $\hat \nu\ib^{(t-1)}$ were Gaussian distributions with mean $\hat{x}\bi^{(t-1)}$ and the variance $\hat{v}\bi ^{(t-1)}$.
Moreover, $m\ia\ti$ has the remarkable property that it shadows both the sequence $\bar{m}\ia\ti$ of the means of the BP iterates with the Gaussian initial condition and the sequence $x\ia\ti$ of the means of the BP iterates with the non-Gaussian initial condition $\nu^{(0)}_{i\to a}$. The same occurs in a much simpler way for the sequences of the variances: $s\ia\ti$, $\bar s\ia\ti$ and $v\ia\ti$. The shadowing of $x\ia\ti$ follows from the central limit theorem approach of the past two sections. For the shadowing of $\bar{m}\ia\ti$ one has to carefully analyse the BP iteration, as done in the sequel. 

To better illustrate this point, we analyse in detail the first two steps of the iteration. 

\subsection*{Iteration at $t=1$}
Recall that at $t=0$ we have
\begin{equation}\label{eq:quartico}
x\ia\zzero = 0, \qquad
v\ia\zzero = v\zzero, \qquad 
\hat{x}\ai\zzero = \frac{y_a}{A_{ai}}, \qquad 
\hat{v}\ai\zzero = \frac{1}{A_{ai}^2}\sum_{j\neq i}A_{aj}^2\,v\zzero
\;\approx\; \frac{1}{A_{ai}^2}\,\frac{v\zzero}{\delta}.
\end{equation}
Moreover, we have
\be\label{eq:s_1m_1}
\bar s\oone\ai=s\oone\ai\approx
\Bigl(2\beta+\frac{\delta}{v^{(0)}}\Bigr)^{-1}\,\quad\text{ and }\quad \bar m\oone\ai=m\oone\ai\,. 
\ee

\begin{Lem}\label{lem:t1-iteration}
It holds that
\bea
v\oone \ia &=&  \bar s\oone  + \err\oone _2,
\qquad
\err\oone _2 \;=\; O_{\PP}\Big(N^{-\frac{1}{2}}\Big)\,,\label{eq:v_uno}\\
x\oone \ia  &=& m\oone \ia  + \err\oone \ia ,
\qquad
\err\oone \ia  \;=\; O_{\PP} \Big(N^{-\frac{1}{2}}\Big)\,.\label{eq:x_uno}
\eea
The errors $\err\oone \ia$, $\err\oone _2$ are both a continuous function of the matrix entries of $A$.  
\end{Lem}
\begin{proof}
Consider
\be
s\ia\oone = \Bigl(2\beta+\sum_{b\neq a}\frac{1}{\hat{v}\bi\zzero}\Bigr)^{-1}\,. 
\ee
Using \eqref{eq:quartico} and $\sum_{b\neq a}A_{bi}^2\approx 1$, gives the approximation $$\sum_{b\neq a}\frac{1}{\hat{v}\bi\zzero} \approx \frac{\delta}{v\zzero}$$ 
whence 
\be\label{eq:s_1}
s\ia\oone \approx \Bigl(2\beta+\frac{\delta}{v\zzero}\Bigr)^{-1} = \bar s\oone\,.
\ee

Applying Proposition \ref{prop:Gaussian-approx} with $F(s)=s^2$ and $F(s)=s$ we get (see Lemma \ref{lemma:gaussianfacile}, ii) with $t=0$)
\[
\absv{v\ia\oone - s\ia\oone}= O_\PP\parav{\frac{1}{\sqrt N}} s\ia\oone=O_\PP\parav{\frac{1}{\sqrt N}} \bar s\oone\,.
\]
and it is a continuous function of the elements of the matrix $A$. This proves \eqref{eq:v_uno}.

Next, we compute 
\[
m\ia\oone
= s\oone\ia \sum_{b\neq a}\frac{\hat{x}\zzero\bi}{\hat{v}\zzero\bi}
= \frac{s\oone\ia}{v\zzero}\,\delta \sum_{b\neq a} y_bA_{bi},
\]
Proposition \ref{prop:Gaussian-approx} with $F(s)=s$ and the inequality $\mathbb{E}[\lvert X\lvert] \leq C\left(\lvert \mathbb{E}[X]\lvert +\sqrt{\var(X)}\right)$ yield
\[
x\ia\oone \;=\; m\ia\oone + \err\ia\oone,
\qquad
\err\ia\oone \;=\; O_{\PP}\parav{\frac{1}{\sqrt N}}
\left( \frac{\bar s\oone}{v\zzero}\delta \left\lvert\sum_{b\neq a} y_bA_{bi}\right\lvert +\sqrt{\bar s\oone}\right),
\]
and $\err\ia\oone$ is a continuous function of the elements of the matrix $A$. 
Using that $\bigl|\sum_{b\neq a}y_bA_{bi}\bigr| = O_\PP(1)$,
\[
\frac{\bar s\oone}{v\zzero}\,\delta
= \frac{\delta}{2\beta v\zzero+\delta}\,,
\]
and that $\sqrt{\bar s\oone}$ is bounded in $\beta$, we conclude $\err\oone \ia 
= O_{\PP} \Big(N^{-\frac{1}{2}}\Big).$

\end{proof}

\subsection*{Iteration for $t=2$}
The analysis of the case $t=2$ is slightly more complicated, as it has already most of the features appearing for general $t$. We start by recording the following simple fact.

\begin{Lem}\label{lem:recip-affine}
Let $\delta\in(0,1)$, $\beta>0$, and let $x>0$. Let
\[
f(x) \coloneqq \; \frac{x}{2\beta x+\delta}\,.
\]
Then for all $|\e|\leq x/2$ we have
\be\label{eq:recip-affine}
\bigl|f(x+\varepsilon)-f(x)\bigr|\leq \frac{|\e|\d}{(\b x+\d)^2}\,. 
\ee
\end{Lem}
\begin{proof}
Since $f$ is Lipschitz, it suffices to compute 
\be
\sup_{u\in (x,x+\e)} f'(u)=\sup_{u\in (x,x+\e)}\frac{\d}{(2\b x+\d)^2}\leq \frac{\d}{(\b x+\d)^2}\,. 
\ee
\end{proof}

Consider now
\begin{equation}\label{eq:sbar2}
    \bar{s}^{(2)} = \Bigl(2\beta+\frac{\delta}{s\oone}\Bigr)^{-1},
\end{equation}
The next lemma shows that $s^{(2)}$ differs from $\bar{s}^{(2)}$ by a term that vanishes as $N\to\infty$.

\begin{Lem}\label{lem:var_t2}
Uniformly over $a,i$, it holds that
\[
s\ia^{(2)} \;=\; \bar{s}^{(2)} + \ERR_2^{(2)}
\qquad\text{with}\qquad
\ERR_2^{(2)} = O_\PP\left(\frac{1}{\sqrt N}\right)\,
\]
and
\[
v\ia^{(2)} \;=\; \bar{s}^{(2)} + \err_2^{(2)},
\qquad
\err_2^{(2)} \;=\; O_\PP\left(\frac{1}{\sqrt N}\right).
\]
The errors $\ERR_2^{(2)}$, $\err_2^{(2)}$ are both a continuous functions of the matrix entries of $A$.  
\end{Lem}

\begin{proof}
We have by \eqref{eq:v_uno}
\[
\hat{v}\ai\oone = \frac{1}{A_{ai}^2}\sum_{j \neq i} A_{aj}^2 v\oone \;\approx\; \frac{1}{A_{ai}^{2}}\,\frac{1}{\delta}\,\bigl(s\oone+\err_2\oone\bigr),
\]
so that
\[
s\ia^{(2)} \;=\; \Bigl(2\beta+\sum_{b\neq a}\frac{1}{\hat{v}\bi\oone}\Bigr)^{-1}
\;\approx\; \Bigl(2\beta+\frac{\delta}{s\oone+\err_2\oone}\Bigr)^{-1}.
\]
In the last step we have used \eqref{eq:v_hatBIS} and (\ref{eq:v_uno}).
Let $f(u)\coloneqq(2\beta+\delta/u)^{-1}$. By Lemma~\ref{lem:recip-affine}, for $N$ sufficiently large we have
\[
\bigl|f(s\oone+\err_2\oone)-f(s\oone)\bigr|
\;\lesssim_{\delta}\; \frac{|\err_2\oone|}{1+\beta\,|\err_2\oone|}.
\]
From Lemma~\ref{lem:t1-iteration} we have $|\err_2\oone|=O_\PP \bigl(N^{-\frac{1}{2}}\bigr)$, hence
\[
\ERR_2^{(2)} \;=\; \bigl|s^{(2)}-\bar{s}^{(2)}\bigr|
\;=\; O_\PP \left(\frac{1}{\sqrt N}\right)
\]
and it is a continuous function of the matrix entries of $A$.  
Finally, Proposition~\ref{prop:gaussapprox-t} implies
\[
\absv{v^{(2)} \;-\; s^{(2)}}= O_\PP\parav{\frac{1}{\sqrt N}}\,
\]
and that $\err_2^{(2)}$ is a continuous function of the matrix entries of $A$.  
\end{proof}

\begin{Lem}\label{lem:DELTA2}
Let $\Delta\ia^{(2)}=\delta\,\frac{s\ia^{(2)}}{v\ia\oone }$ and $\bar{\Delta}^{(2)}=\delta\,\frac{\bar{s}^{(2)}}{\bar{s}\oone }$.
Uniformly over $a,i$,
\[
\Delta\ia^{(2)} \;=\; \bar{\Delta}^{(2)} \;+\; O_\PP \left(N^{-\frac{1}{2}}\right).
\]
Moreover, $\Delta\ia^{(2)}$ and $\bar{\Delta}^{(2)}$ are continuous functions of the matrix entries of $A$.  
\end{Lem}

\begin{proof}
Write $v\oone\ai=v\oone =\bar s\oone +\err\oone _2$ and $s\ai^{(2)}=s^{(2)}=\bar{s}^{(2)}+\ERR^{(2)}_2$ due to Lemma \ref{lem:t1-iteration} and \ref{lem:var_t2}. Then
\[
\frac{1}{\delta}\,\bigl|\Delta\ai^{(2)}-\bar{\Delta}^{(2)}\bigr|
=\left|\frac{s^{(2)}}{v\oone }-\frac{\bar{s}^{(2)}}{\bar{s}\oone }\right|
\le \frac{|\ERR^{(2)}_2|}{\,\bar s\oone +|\err\oone _2|\,}
\;+\; \frac{\bar{s}^{(2)}\,|\err\oone _2|}{\,\bar s\oone \bigl(\bar s\oone +|\err\oone _2|\bigr)}.
\]
By Lemma~\ref{lem:t1-iteration}, $\bar s\oone$ is fixed as $N$ grows while 
$|\err\oone _2|=O_\PP \big(N^{-\frac12}\big)$.
By Lemma~\ref{lem:var_t2}, $|\ERR^{(2)}_2|=O_\PP \big(N^{-\frac12}\big)$ and
$\bar{s}^{(2)}=(2\beta+\delta/s\oone )^{-1}$.
Therefore
\[
\frac{|\ERR^{(2)}_2|}{\,s\oone +|\err\oone _2|\,}
= O_\PP \big(N^{-\frac12}\big),
\]
and
\[
\frac{\bar{s}^{(2)}\,|\err\oone _2|}{\,s\oone \bigl(s\oone +|\err\oone _2|\bigr)}
= O_\PP \big(N^{-\frac12}\big).
\]
Multiplying by $\delta\in(0,1]$ yields the claim.
\end{proof}

\begin{Lem}\label{lem:m2}
Define 
\begin{equation*}
    \ERR\ia\ttwo= m\ia\ttwo - \bar{m}\ia\ttwo\,. 
\end{equation*}
Then $\ERR\ia\ttwo$ is a continuous function of the entries of the matrix $A$ and
\begin{equation*}
    \absv{\ERR\ia\ttwo} =O_\PP \left(\frac{1}{\sqrt N}\right).
\end{equation*}
Moreover,
\[
x\ia\ttwo=\bar{m}\ia\ttwo
+\err_1\ttwo,
\qquad
\err_1\ttwo=O_\PP\left(N^{-\frac{1}{2}}\right).
\]
The error $\err_1\ttwo$ is a continuous function of the entries of the matrix $A$.
\end{Lem}

\begin{proof}
Recall $\Delta\ttwo\ia$ and $\bar{\Delta}\ttwo$ defined in Lemma \ref{lem:DELTA2} and set for brevity $\Delta\ttwo:=\Delta\ttwo\ia$. Combining \eqref{eq:shadow}, \eqref{eq:xx-hatBIS}, \eqref{eq:v_hatBIS}, Lemma \ref{lem:t1-iteration}, \eqref{eq:concentration} and $x^{(0)}\ia=0$ for all $i\in [N], a\in[m]$, we write 
\[
m\ia\ttwo
= \Delta\ttwo\sum_{b\neq a}A_{bi}\,\hat{x}\bi\oone
= \Delta\ttwo\sum_{b\neq a}y_bA_{bi}
   -\Delta\ttwo\sum_{b\neq a}\sum_{j\neq i}A_{bi}A_{bj}\,m\jb\oone
   -\Delta\ttwo\sum_{b\neq a}\sum_{j\neq i}A_{bi}A_{bj}\,\err\jb\oone.
\]
We compare this with the MP-equation of \eqref{eq:MP-gauss}, that gives
\[
\bar{m}\ia\ttwo
= \bar{\Delta}\ttwo\Biggl(\sum_{b\neq a}y_bA_{bi}
 -\sum_{b\neq a}\sum_{j\neq i}A_{bi}A_{bj}\,\bar{m}\jb\oone\Biggr),
\]
since $\bar{m}\jb\oone=m\jb\oone$, see \eqref{eq:s_1m_1}.
Hence
\begin{align}
\ERR\ia\ttwo
= m\ia\ttwo-\bar{m}\ia\ttwo
&= \parav{\Delta\ttwo-\bar{\Delta}\ttwo}\parav{
 \sum_{b\neq a}y_bA_{bi}
-\sum_{b\neq a}\sum_{j\neq i}A_{bi}A_{bj}\,\bar{m}\jb\oone}
\label{eq:m2:T1}\\
&\qquad
+\Delta\ttwo\,
\parav{\sum_{b\neq a}\sum_{j\neq i}A_{bi}A_{bj}\,\err\jb\oone}.
\label{eq:m2:T2}
\end{align}
The continuity of $\ERR\ia\ttwo$ w.r.t. the matrix elements of $A$ follows from the continuity of $\Delta\ttwo$, $\bar{\Delta}\ttwo$ and $\err\jb\oone$.
To control \eqref{eq:m2:T1},
we recognise that the second term of the product is exactly $\frac{\bar{m}\ia\ttwo}{\bar{\Delta}\ttwo}$. 
By Lemma~\ref{lem:DELTA2} we have that
$\bigl|\Delta\ttwo-\bar{\Delta}\ttwo\bigr|
=O_\PP \left(N^{-\frac{1}{2}}\right)$ and by Proposition \ref{prop:AMP_dec} we know $\bar{m}\ia\ttwo =O_\PP(1)$. Together with the fact that $\bar{\Delta}\ttwo= O(1) $ we conclude that \eqref{eq:m2:T1} is $O_\PP \left(\frac{1}{\sqrt N}\right)$.\\

Next, we bound \eqref{eq:m2:T2}. 
From Lemma \ref{lem:t1-iteration} we have that
$\err\jb\oone$ is $O_\PP \left(\frac{1}{\sqrt N}\right)$. Moreover, $\err\jb\oone $ is independent of $A_{bi}$ and $A_{bj}$, since $j\neq i$. Therefore, $\Bigl|\sum_{b\neq a}\sum_{j\neq i}A_{bi}A_{bj}\,\err\jb\oone\Bigr| = O_\PP \left(\frac{1}{\sqrt N}\right)$.\\
Combining \eqref{eq:m2:T1}-\eqref{eq:m2:T2} and the fact that $\Delta\ttwo=O(1)$, we conclude $\ERR\ia\ttwo
= O_\PP \left(\frac{1}{\sqrt N}\right).$
Finally, Proposition \ref{prop:gaussapprox-t} gives
\be
\absv{\err_1\ttwo}=\absv{x\ia\ttwo - m\ia\ttwo}=O_{\PP} \bigl(N^{-\frac{1}{2}}\bigr)
\ee
and it is a continuous function of the entries of the matrix $A$. 
\end{proof}

\subsection*{Iteration for $t>2$}
We now deal with the case $t>2$. Recall that (see \eqref{eq:iteration_v1})
\begin{equation}\label{eq:sbar-rec}
    \bar{s}^{(t+1)} \;\coloneqq\; \Bigl(2\beta+\frac{\delta}{\bar{s}\ti}\Bigr)^{-1},\qquad t\ge 1.
\end{equation}

\begin{Lem}\label{lem:varCLT}
Assume $v^{(0)}=\beta^{-\kappa}$ with $\kappa\ge 0$. Then $v\ti\ia =v\ti$ is edge-independent for all $t$, and
\[
v\ti=\bar{s}\ti+
\err_2\ti,
\]
with
\[
\err_2\ti=O_\PP  \left(N^{-\frac12}\right)
\]
uniformly over $(a,i)$ and for each fixed $t$. The error $\err_2\ti$ is a continuous function of the matrix elements of $A$.
\end{Lem}

\begin{proof} We proceed by induction. The case $t=2$ was dealt in Lemma \ref{lem:var_t2}.
Assume that for some $t\ge2$ we have
\begin{equation}\label{eq:IHH}
v\ti=\bar{s}\ti
+\err_2\ti,\qquad
\err_2\ti=O_\P  \bigl(N^{-\frac12}\bigr),
\end{equation}
uniformly over $(a,i)$, and that $\bar{s}\ti$ is edge-independent.
From the induction hypothesis and (\ref{eq:v_hatBIS}) it follows that
\begin{equation}\label{eq:somma_conc}
s\ia^{(t+1)}=\parav{2\b+\sum_{b\neq a}\frac{1}{\hat v\ti_{b\to i}}
}^{-1}\;\approx\; \parav{2\b+\frac{\delta}{v\ti}}^{-1}=:s^{(t+1)}\,. 
\end{equation}
Next we compare $s^{(t+1)}$ with
\begin{equation}\label{eq:s-t+1}
\bar s^{(t+1)}=\frac{1}{\,2\beta+\frac{\delta}{\bar s\ti}\,}.
\end{equation}
Setting $f(u)\coloneqq(2\beta+\delta/u)^{-1}$ and
using Lemma~\ref{lem:recip-affine} and \eqref{eq:IHH},
\be\label{eq:split-AC}
\bigl|s^{(t+1)}-\bar s^{(t+1)}\bigr|
=\bigl|f(\bar s\ti+\err_2\ti)-f(\bar s\ti)\bigr|
\;=\; O_\PP \bigl(N^{-\frac12}\bigr)\,.
\ee
Combining the bound above and (\ref{eq:somma_conc}) we obtain that $s\ia^{(t+1)}$ shadows $\bar s^{(t+1)}$:
\be
\bigl|s\ia^{(t+1)}-\bar s^{(t+1)}\bigr|=O_\PP \bigl(N^{-\frac12}\bigr)\,. 
\ee

Finally, by Proposition~\ref{prop:gaussapprox-t} we have that $s\ia^{(t)}$ shadows also $v\ia^{(t)}$, that is
\[
\absv{\err_2\ti}=\absv{v\ia^{(t+1)}-s\ia^{(t+1)}}=O_\PP\parav{N^{-\frac12}}\qquad \text{for all $t$}\,,
\]
and it is a continuous function of the entries of the matrix $A$.
\end{proof}

\begin{Lem}\label{lem:DeltaCLT}
Let
\be\label{eq:defdeltat}
\Delta\ia\ti\coloneqq \delta\,\frac{s\ia^{(t+1)}}{v\ia\ti},
\qquad
\bar{\Delta}\ti\coloneqq \delta\,\frac{\bar{s}^{(t+1)}}{\bar{s}\ti}.
\ee
Then, uniformly over $(a,i)$,
\[
\Delta\ia\ti=\bar{\Delta}\ti+O_\PP \left(N^{-\frac12}\right)\,.
\]
Moreover, $\Delta\ia\ti$ and $\bar{\Delta}\ti$ are continuous functions of the matrix elements of $A$. 
\end{Lem}

\begin{proof}
Using that $s\ia^{(t+1)}\approx\frac{v\ia\ti}{2\beta v\ia\ti+\delta}$ and $\bar{s}^{(t+1)}=\frac{\bar{s}\ti}{2\beta \bar{s}\ti+\delta}$, we have that
\[
\Delta\ia\ti\approx\frac{\delta}{\,2\beta v\ia\ti+\delta\,},\qquad
\bar\Delta\ti=\frac{\delta}{\,2\beta \bar s\ti+\delta\,}.
\]
Set $g(u)\coloneqq\frac{\delta}{2\beta u+\delta}$ for $u>0$. Then
\[
|\Delta\ia\ti-\bar\Delta\ti|
=\bigl|g\bigl(\bar s\ti+e\ia\ti\bigr)-g\bigl(\bar s\ti\bigr)\bigr|,
\qquad e\ia\ti\coloneqq v\ia\ti-\bar s\ti.
\]
By Lemma~\ref{lem:varCLT}, $|e\ia\ti|=O_\PP \big(N^{-\frac12}\big)$.
Moreover, 
a direct computation gives, for all $x, e \ge0$,
\[
\bigl|g(x+e)-g(x)\bigr|
=\frac{2\beta\delta\,|e|}{\bigl(2\beta(x+e)+\delta\bigr)\bigl(2\beta x+\delta\bigr)}
\le \frac{2\beta\,|e|}{2\beta|e|+\delta}
\le C\,\min\{|e|,\beta^{-1}\}.
\]
For fixed $\beta>0$ this yields
\[
|\Delta\ia\ti-\bar\Delta\ti| \;=\; O_\PP \big(N^{-\tfrac12}\big),
\]
uniformly over $(a,i)$, which proves the claim.
\end{proof}

\begin{Lem}\label{lem:mt}
Define 
\be\label{eq:defE}
    \ERR\ia\ti:= m\ia\ti - \bar{m}\ia\ti. 
\ee
Then 
\be\label{eq:Esmall}
    \ERR\ia\ti =O_\PP \left(\frac{1}{\sqrt N}\right).
\ee
Moreover,
\begin{equation}\label{eq:x_ti_nove}
x\ia\ti=\bar{m}\ia\ti+\err_1\ti,
\qquad
\err_1\ti=O_\PP \left(N^{-\frac{1}{2}}\right)\,.
\end{equation}
\end{Lem}

\begin{proof} We proceed by induction. The case $t=2$ was dealt in Lemma \ref{lem:m2}. For $t\geq2$, we start from
\begin{equation}\label{eq:start}
m\tti\ia  \;=\; s\ia^{(t+1)}\sum_{b\neq a}\frac{\hat{x}\ti_{b\to i}}{\hat{v}\ti_{b\to i}}.
\end{equation}
By Lemma~\ref{lem:varCLT}, uniformly in $(b,i)$,
\[
\hat{v}\ti_{b\to i} \;\approx\; \frac{1}{A_{bi}^2}\,\frac{1}{\delta}\,v\ti
= \frac{1}{A_{bi}^2}\,\frac{1}{\delta}\,\bigl(\bar{s}\ti+\err_2\ti\bigr),
\]
so that (see \eqref{eq:defdeltat})
\begin{equation}\label{eq:start2}
m\tti\ia  \;=\; \Delta\ia\ti\sum_{b\neq a}A_{bi}^2\,\hat{x}\ti_{b\to i}\,. 
\end{equation}
We use \eqref{eq:x_ai} 
and the induction hypothesis
\[
\hat{x}\ti_{b\to i}
= \frac{y_b}{A_{bi}}
-\frac{1}{A_{bi}}\sum_{j\neq i}A_{bj}\,\bar{m}\ti_{j\to b}
-\frac{1}{A_{bi}}\sum_{j\neq i}A_{bj}\,\ERR\ti_{j\to b}
-\err_1\ti\frac{1}{A_{bi}}\sum_{j \neq i} A_{bj}
\]
Hence
\begin{align*}
m\tti\ia 
&= \Delta\ia\ti\sum_{b\neq a}y_bA_{bi}
-\Delta\ia\ti\sum_{b\neq a}\sum_{j\neq i}A_{bi}A_{bj}\,\bar{m}\ti_{j\to b}\\
&\quad -\,\Delta\ia\ti\sum_{b\neq a}\sum_{j\neq i}A_{bi}A_{bj}\,\ERR\ti_{j\to b}
-\,\Delta\ia\ti\err_1\ti\sum_{b\neq a}\sum_{j\neq i}A_{bi}A_{bj}.
\end{align*}
We compare with (see \eqref{eq:defdeltat})
\[
\bar{m}\tti\ia 
= \bar{\Delta}\ti\Biggl(\sum_{b\neq a}y_bA_{bi}
-\sum_{b\neq a}\sum_{j\neq i}A_{bi}A_{bj}\,\bar{m}\ti_{j\to b}\Biggr),
\qquad
\bar{\Delta}\ti\coloneqq \delta\,\frac{\bar{s}^{(t+1)}}{\bar{s}\ti}.
\]
We have
\begin{align}
\ERR\tti\ia 
=  m\tti\ia -\bar{m}\tti\ia 
&= (\Delta\ia\ti-\bar{\Delta}\ti)\parav{
 \sum_{b\neq a}y_bA_{bi}
-\sum_{b\neq a}\sum_{j\neq i}A_{bi}A_{bj}\,\bar{m}\ti_{j\to b} }
\label{eq:mT1-fix}\\
&\quad+\Delta\ia\ti\,
\sum_{b\neq a}\sum_{j\neq i}A_{bi}A_{bj}\,\ERR\ti_{j\to b}
\label{eq:mT2-fix}\\
&\quad+\Delta\ia\ti\,
\err_1\ti \sum_{b\neq a}\sum_{j \neq i}A_{bi} A_{bj} \,.
\label{eq:mT3-fix}
\end{align}
Therefore by the MP equations of Proposition \ref{prop:mean}
\be
\ERR\tti\ia =\frac{\Delta\ti\ia-\bar{\Delta}\ti}{\bar{\Delta}\ti}\bar{m}\tti\ia+\Delta\ia\ti\,
\sum_{b\neq a}\sum_{j\neq i}A_{bi}A_{bj}\,\ERR\ti_{j\to b}+\Delta\ia\ti\,
\err_1\ti \sum_{b\neq a}\sum_{j \neq i}A_{bi} A_{bj}\,.
\ee
The first and last summand on the r.h.s. are easily estimated. 
Lemma~\ref{lem:DeltaCLT} yields
\[
|\Delta\ia\ti-\bar{\Delta}\ti|
=O_\PP  \left(N^{-\frac{1}{2}}\right)\,,
\]
while Proposition~\ref{prop:AMP_dec} gives $\bar{m}\ti\ia=O_\PP(1)$.
Thus the first summand is $O_\PP  \left(N^{-\frac{1}{2}}\right)$. The last summand is also of this order of magnitude, because of the induction hypothesis, $\err_1\ti=O_\PP  \left(N^{-\frac{1}{2}}\right)$, $\Delta\ia\ti=O_\P(1)$ and the fact that $\sum_{b\neq a}\sum_{j \neq i}A_{bi} A_{bj}=O_\PP(1)$. 

Therefore we are left with bounding the second term. We study the recursion for $t\geq2$
\bea
\ERR\tti\ia &=&\text{Rem}\ti_{j\to a}+\Delta\ti\,
\sum_{b\neq a}\sum_{j\neq i}A_{bi}A_{bj}\,\ERR\ti_{j\to b}\,,\\
\text{Rem}\ti_{j\to a}&=&O_\PP  \left(N^{-\frac{1}{2}}\right)
\eea
and $\text{Rem}\ti_{j\to a}$ is a continuous function of the entries of the matrix $A$. 
It suffices to prove by induction over $t\geq1$ that $\ERR\tti\ia$ is a continuous function of the entries of the matrix $A$ and it is $O_\PP  \left(N^{-\frac{1}{2}}\right)$. The base case $t=1$ is dealt in Lemma \ref{lem:m2}. Assume now, that the assertion is true for some $t$. Then also $\ERR\tti\ia$ is a continuous function of the entries of the matrix $A$. Next we prove that $\ERR\tti\ia=O_\PP  \left(N^{-\frac{1}{2}}\right)$, which will conclude the proof. 

We note that $\ERR\tti\ia$ is a generic bounded continuous function of the entries of the matrix $A$. First of all, we observe that the Frobenius norm of $A$ is bounded by one with probability exponentially close to one as $N\to\infty$. On such a large probability event, the Stone-Weiestrass Theorem applies: for all $\e>0$ there is a polynomial $E\ia$ of some degree $D(\e)$ such that
\be
\|\ERR\tti\ia-E\ia\|_{L^\infty}\leq \e\,. 
\ee
Thus
\bea
\sum_{b\neq a}\sum_{j\neq i}A_{bi}A_{bj}\ERR\tti\ia&=&\sum_{b\neq a}\sum_{j\neq i}A_{bi}A_{bj}E\ia+\sum_{b\neq a}\sum_{j\neq i}A_{bi}A_{bj}(\ERR\tti\ia-E\ia)\nn\\
&\leq& \sum_{b\neq a}\sum_{j\neq i}A_{bi}A_{bj}E\ia+\e\sum_{b\neq a}\sum_{j\neq i}|A_{bi}A_{bj}|\nn\\
&\leq& \sum_{b\neq a}\sum_{j\neq i}A_{bi}A_{bj}E\ia+O_\P(\e m)\,.
\eea
As $\e$ can be made arbitrarily small, we only have to prove that the first addendum above is $O_\P(1)$. 

First we prove that the assertion holds for a monomial function. W.l.o.g. we can write
\be
E\jb=A_{bj}^{K}\parav{\sum_{k\neq j}A^{H}_{bk}}^p \,,
\ee
for given integers $K,p\geq0$ and $H\geq1$. Note that $E\jb=O_\P(1)$ (applying Lemma \ref{lemma:facilissimo}), thus we need to prove that 
\be\label{eq:needto}
\sum_{b\neq a}\sum_{j\neq i}A_{bi}A_{bj}E\jb=O_\P(1)\,. 
\ee
We set
\be
X_{bij}:=\sum_{k\neq i,j}A^{H}_{bk}\,.
\ee
Moreover,
\be
\parav{\sum_{k\neq j}A^{H}_{bk}}^p=\parav{A_{bi}^H+X_{bij}}^p=\sum_{\ell=0}^p\binom{p}{\ell}A_{bi}^{H\ell}X^{p-\ell}_{bij}\,.
\ee
Hence,
\be
\sum_{b\neq a}\sum_{j\neq i}A_{bi}A_{bj}E\jb=\sum_{\ell=0}^p\binom{p}{\ell}\sum_{b\neq a}A_{bi}^{H\ell+1}\sum_{j\neq i}A^{K+1}_{bj}X^{p-\ell}_{bij}\,.
\ee
Now we put
\be\label{eq:defYell}
Y^{(\ell)}_{bi}:=\sum_{j\neq i}A^{K+1}_{bj}X^{p-\ell}_{bij}\,.
\ee
We note that $X_{bij}$ is independent of $A_{bi}, A_{bj}$. Recalling the normalisation of the elements of the matrix $A$, applying Lemma \ref{prop:tail} with $x_i=m^{-H/2}$ for all $i=1\ldots N$, we have 
\be\label{eq:XO}
X_{bij}=
\begin{cases}
O_\P\parav{\frac{1}{m^{\frac{H-1}{2}}}}&\text{if $H$ is odd}\,,\\
O_\P\parav{\frac{1}{m^{\frac{H-2}{2}}}}&\text{if $H$ is even}\,.
\end{cases}
\ee
Therefore
\be\label{eq:XOell}
X^{p-\ell}_{bij}=
\begin{cases}
O_\P\parav{\frac{1}{m^{\frac{(p-\ell)(H-1)}{2}}}}&\text{if $H$ is odd}\,,\\
O_\P\parav{\frac{1}{m^{\frac{(p-\ell)(H-2)}{2}}}}&\text{if $H$ is even}\,.
\end{cases}
\ee
Moreover, if $H$ is even then clearly $X_{bij}>0$. Therefore the worst possible case is $H=2$, when the terms $X^{p-\ell}_{bij}$ are all positive $O_\P(1)$. Similarly, we will assume that $K=1$, so that also the terms $A^{K+1}_{bj}$ have the minimal decay $1/m$ with no cancellations due to the random signs in the sum in (\ref{eq:defYell}). Thus we have $Y^{(\ell)}_{bi}=O_\P(1)$ for all $\ell=1\ldots, p$. 

Next, we need to evaluate
\be\label{eq:scalY}
\sum_{\ell=0}^p\binom{p}{\ell}\sum_{b\neq a}A_{bi}^{H\ell+1}Y^{(\ell)}_{bi}\,. 
\ee
If $\ell>0$ then $H\ell+1\geq2$. Thus the summands $A_{bi}^{H\ell+1}Y^{(\ell)}_{bi}$ are at least $O_\P(m^{-1})$ and the sum over $b\neq a$ is at most $O_\P(1)$.

The term $\ell=0$ is dealt separately. By the Hoeffding inequality (w.r.t. the law of $A$) and the Jensen inequality we have
\be
P\parav{\absv{\sum_{b\neq a}A_{bi}Y^{(0)}_{bi}\geq t}\leq \squarev{e^{-c\frac{t^2m}{\sum_{b\neq a}(Y_{bi}^{(0)})^2}}}}
\leq e^{-c\frac{t^2m}{\sum_{b\neq a}\E\squarev{(Y_{bi}^{(0)})^2}}}=e^{-ct^2}\,,
\ee
that is 
$$
A_{bi}Y^{(0)}_{bi}=O_\P(1)\,. 
$$
Clearly the sum over $\ell$ of $O_\P(1)$ terms remains $O_\P(1)$. This proves the assertion under the assumption that $\ERR\tti\ia$ is a monomial function. Therefore the same follows also for any bounded polynomial function, combining different values of $H,K,p$. This concludes the proof. 
\end{proof}


\section{Analysis of the AMP equations}\label{sect:AnaAMP}

In this section we analyse the AMP iteration derived in Theorem \ref{TH:main} and prove Corollary \ref{cor:main}.

\begin{Lem}\label{lem:finite-window-stability}
    Let $(M_t)_{t\geq 0}$ be matrices in $\mathbb{R}^{N \times N}$ and $(b_t)_{t\geq 0}, \ (R\ti )_{t \geq 0}$ vectors in $\mathbb{R}^N.$ Consider, for a fixed $T \in \mathbb{N}$, the sequences
    \begin{align*}
        X\tti  &= M_t X\ti +b_t+R\ti , \\
        W\tti  &= M_t W\ti  + b_t,
    \end{align*}
    with $X\zzero =W\zzero .$ Assume that
    \begin{equation}\label{eq:prod-bound}
        C_T:=\sup _{0 \leq s<t\leq T}\left\lVert \prod_{\tau=s+1}^{t-1} M_\tau\right\rVert <\infty,
    \end{equation}
    and set 
    \begin{equation*}
        \varepsilon_T \coloneqq \max_{0 \leq t\leq T-1} \lVert R\ti \rVert\,. 
    \end{equation*}
	Then
    \begin{equation}\label{eq:finite-window-proximity-lemma}
        \max_{0 \leq t\leq T} \lVert X\ti -W\ti \rVert\leq C_T T \varepsilon_T\,.
    \end{equation}
\end{Lem}
\begin{proof}
Set $E\ti \coloneqq X\ti -W\ti $. Then $E\zzero =0$ and
\begin{equation*}
E\tti  = M_t E\ti  + R\ti .
\end{equation*}
We have that
\begin{equation*}
E\ti =\sum_{s=0}^{t-1}\Bigl(\prod_{\tau=s+1}^{ t-1} M_\tau\Bigr) R^{(s)}.
\end{equation*}
Taking norms and using \eqref{eq:prod-bound},
\begin{equation*}
\|E\ti \|  \le  \sum_{s=0}^{t-1} \Bigl\|\prod_{\tau=s+1}^{ t-1} M_\tau\Bigr\| \|R^{(s)}\|
 \le  C_T \sum_{s=0}^{t-1}\|R^{(s)}\|
 \le  C_T  t  \varepsilon_T.
\end{equation*}
Maximizing over $t\le T$ yields \eqref{eq:finite-window-proximity-lemma}.
\end{proof}

\begin{proof}[Proof of Corollary \ref{cor:main}]

By Lemma \ref{lem:finite-window-stability} we can ignore the $O_\PP \Big(\Delta\ti  \Gamma\ti _{1} N^{-1/2}\Big)$ term in the AMP updates (\ref{eq:AMP-cor}) of Theorem \ref{TH:main}, by paying a $O_\PP \Big(C_tt\Delta\ti  \Gamma\ti _{1} N^{-1/2}\Big)$ error in the solution at step $t$, which is irrelevant as $N\to\infty$. Thus we are left to study the recursion
\begin{equation}\label{eq:tv-iter}
X\tti 
 = 
X\ti 
 + \Delta\ti  \left[A^T y  -  A^T A X\ti \right],
\qquad X\zzero =0 \,.
\end{equation}

The spectrum of $\d A^TA$ is contained a.s. in the interval $[\l_-,\l_+]$, where $\l_\pm:=(1\pm\sqrt \d)^2$.
Next, we recall 
\be\label{eq:recall-deltat}
\D^{(t)}=\frac{\d}{\d+\frac{2\b v\zzero(1-\d)}{\d^t(1-\d)+2\b v\zzero(1-\d^t)}}\,.
\ee
By a direct computation we see that we can write
\be\label{eq:recall-delta2}
\D^{(t)}=\d-\d^{t+1}\g(\b,t)\,,\qquad \g(\b,t):=\frac{(1-\d)(2\b v^{(0)}-(1-\d))}{\d^{t+1}(1-\d)+2\b v^{(0)}(1-\d^{t+1})}\,.
\ee
Note that $\g(\b,t)\geq 0$ if $\b\geq (1-\d)/2v\zzero$, it is increasing as $t$ grows and $\g(\b,t)<1$ for all $t$.
We set 
\be\label{eq:opR}
R^{(t)}:=I_N-\D^{(t)}A^TA\,,\qquad R^{(\infty)}:=I_N-\d A^TA\,.
\ee
We have that
\bea
\sup_{t\in\N}\|R^{(t)}\|_\op&\leq& \sup_{t\in\N}\max_{\l\in[(1-\sqrt\d)^2, (1+\sqrt\d)^2]}|1-(1-\d^{t}\g(\b,t))\l|\nn\\
&=&\max_{\l\in[(1-\sqrt\d)^2, (1+\sqrt\d)^2]}|1-\l(1-\g(\b,0))|\nn\\
&=&1-(1-\sqrt\d)^2(1-\g(\b,0))=:\r_1\,,\label{eq:boundRt}
\eea
where $\r_1\in(0,1)$ for all $\d\in(0,1)$, and
\be
\|R^{(\infty)}\|_\op\leq 1-(1-\sqrt\d)^2=\sqrt\d(2-\sqrt\d)=:\r_2\in(0,1)\,. 
\ee
The equation \eqref{eq:tv-iter} writes as
\begin{equation}\label{eq:tv-iter2}
X\tti 
 = 
R\ti X\ti
 + \Delta\ti  A^T y\,,
\qquad X\zzero =0 \,.
\end{equation}
In parallel, we define $X^{(\infty)} $ as the solution of
\be
X^{(\infty)}  
 = 
R^{(\infty)} X^{(\infty)} 
 + \d  A^T y\,,
\qquad X\zzero =0 \,.
\ee
We have
\bea
X^{(\infty)} &=&\d\sum_{k=0}^{\infty}(R^{(\infty)})^k A^T y\nn\\
&=&\d A^T\sum_{k=0}^{\infty}(I_N-\d AA^T)^k y\nn\\
&=&\d A^T (I_N-(I_N-\d AA^T))^{-1}y\nn\\
&=&A^T(AA^T)^{-1}y=x^\star\,.\label{eq:Xinfty}
\eea
Set now $\Gamma\ti:=X\ti-X^{(\infty)}$. We have
\bea
\Gamma\tti&=&R\ti X\ti-R^{(\infty)} X^{(\infty)}+(\D\ti-\d)A^Ty\nn\\
&=&R\ti\Gamma\ti+(R\ti-R^{(\infty)}) X^{(\infty)}+(\D\ti-\d)A^Ty\nn\label{eq:grafnontorna}\,.
\eea
Now, we observe that
\be
R\ti-R^{(\infty)}=-(\D\ti-\d)A^TA\,,
\ee
thus
\bea
\Gamma\tti&=&R\ti X\ti-R^{(\infty)} X^{(\infty)}+(\D\ti-\d)A^Ty\nn\\
&=&R\ti\Gamma\ti+(\D\ti-\d)(A^Ty-A^TAX^{(\infty)})\nn\\
&=&R\ti\Gamma\ti\,
\eea
since by (\ref{eq:Xinfty}) we have
\be
A^TAX^{(\infty)}=A^TAx^\star=A^T AA^T(AA^T)^{-1}y=A^Ty\,.
\ee
Hence, by (\ref{eq:boundRt})
\be
\|\Gamma\tti\|_2\leq \|R\ti\|_\op\|\Gamma\ti\|_2\leq \r_1^t\|x^\star\|_2\,.
\ee
This proves the assertion.
\end{proof}


\appendix


\section{Equivalence of $L^2$ norms of random tensors}\label{App:equiv}

In this appendix we prove the following statement.

\begin{proposition}\label{prop:equiv}
Let $\xi_1,\ldots, \xi_m$ be i.i.d. centred r.vs with unitary variance and set
\be
\chi_r(\xi)\coloneqq m^{-\frac r2}\sum_{\substack{n_i\neq n_{i+1}\\i=1\ldots r-1}} \xi_{n_1}\ldots \xi_{n_r}\,. 
\ee
We have
\be
\|\chi_r(\xi)\|^2_{L^2}\leq\|\chi_r(g)\|^2_{L^2}+\frac{C}{m}\,,
\ee
where $C>0$ is a constant depending on $r$ and on the distribution on $\xi$ and $g_1,\ldots, g_m$ are independent standard Gaussian r.vs.
\end{proposition}
\begin{proof}
Let us set
\be
\mathcal B\coloneqq \{n_i\neq n_{i+1}\,:\,i=1\ldots r-1\}
\ee
and
\be
B_k\coloneqq \{n_1\ldots n_r\,:\,|i\,:\,n_i=n_{j}\,\,\text{for some }j\neq i,i+1|=k\}\,
\ee
(that is the subsets of $\mathcal B$ in which there are exactly $k$ pairs of indices taking the same values). The set $B_k$ easily identifies with a graph in which the vertex $i,j$ are connected iff $n_i=n_j$. We have
\be
\chi_r=\sum_{k=0}^{r}\chi_{r}^{(k)}\,,\qquad \chi_{r}^{(k)}\coloneqq m^{-\frac r2}\sum_{B_k}\xi_{n_1}\ldots \xi_{n_r}\,. 
\ee
Thus
\be
\|\chi_r\|_{L^2}\leq \sum_{k=0}^{r}\|\chi_{r}^{(k)}\|_{L^2}\,. 
\ee
The term $\chi_{r}^{(0)}$ corresponds to the decoupled chaos, for which we have
\be\label{eq:wickone}
\|\chi_{r}^{(0)}(\xi)\|_{L^2}=\|\chi_{r}^{(0)}(g)\|_{L^2}\,. 
\ee
Since all the indices here are distinct, this term corresponds to the ones in the Wick expansion of Theorem \ref{thm:Wick4} in which the vertices can be paired only in disjoint couples.

Let us assume now $k\geq1$. We define a cycle of length $p$ a chain of identities of $p$ indices. Clearly $p\geq3$. We set $B^a_k$ to be the subset of elements of $B_k$ containing no cycles and $B^c_k$ its complementary set w.r.t. $B_k$. We also define $p_i$ as the length of the path containing the index $n_i$ and if $i<j$ and $n_j$ belongs to the same cycle as $n_i$ we set $p_j=0$. For elements of $B^a_k$ we can extend this definition, setting $p_i\in\{0,1,2\}$ to be the number of indices equal to $n_i$ and if $i<j$ and $p_i=2$ then $p_j=0$. Clearly $k\leq |i\,:\,p_i=0|=:h$ and $|i\in B_k\,:\,n_i\neq n_j\,\forall j\in B_k\setminus\{i\}|\leq m^{r-h}$. Moreover, for each element in $B_k$ we can write
\be
\xi_{n_1}\ldots \xi_{n_r}= \xi^{p_1}_{n_1}\ldots \xi^{p_r}_{n_r}=\prod_{i\in[r]\,:\,p_i\neq0}\xi^{p_i}_{n_i}\,. 
\ee
It is important to notice that the indices in the terms appearing on the r.h.s of the display above are all distinct. 

Thus, squaring, we have
\bea
\|\chi_{r}^{(k)}(\xi)\|^2_{L^2}&=&m^{-r}\sum_{B_k\times B_k} E[\xi^{p_1}_{n_1}\ldots \xi^{p_r}_{n_r}\xi^{q_1}_{m_1}\ldots \xi^{q_r}_{m_r}]\nn\\
&=&m^{-r}\sum_{B^a_k\times B^a_k} E[\xi^{p_1}_{n_1}\ldots \xi^{p_r}_{n_r}\xi^{q_1}_{m_1}\ldots \xi^{q_r}_{m_r}]\label{eq:BaBa}\\
&+&2m^{-r}\sum_{B^a_k\times B^c_k} E[\xi^{p_1}_{n_1}\ldots \xi^{p_r}_{n_r}\xi^{q_1}_{m_1}\ldots \xi^{q_r}_{m_r}]\label{eq:BaBc}\\
&+&m^{-r}\sum_{B^c_k\times B^c_k} E[\xi^{p_1}_{n_1}\ldots \xi^{p_r}_{n_r}\xi^{q_1}_{m_1}\ldots \xi^{q_r}_{m_r}]\label{eq:BcBc}\,,
\eea
where the indices $q_i$'s are defined as the $p_i$'s, but referred to the $m_i$'s. Since the $n_i$'s and $m_i$'s are all distinct, the expectation value can only pair one of the $n_i$'s with one of the $m_i$'s. Moreover, if some of the $p_i$'s is odd it must necessarily be paired with some of the $q_i$'s in order to have a non-vanishing contribution. 

In each of the terms of the sum in (\ref{eq:BaBa}) there are $k$ $p_i$'s (resp. $q_i$'s) equal to two, $k$ $p_i$'s (resp. $q_i$'s) equal to zero and all the other equal to one. After the pairing given by the expectation, the terms without cycles in the product graph correspond to the terms of the Wick expansion of Theorem \ref{thm:Wick4} in which one pairs $k$ indices within each $B_k$ and the there are $r-2k$ cross pairings between the $n_i$'s and $m_i$'s. Together with (\ref{eq:wickone}) this recover all the Wick expansion giving $\|\chi_r(g)\|^2_{L^2}$. 

What is left are the terms containing cycles in the product graph, that is the remaining summands in \eqref{eq:BaBa} and all those of \eqref{eq:BaBc}, \eqref{eq:BcBc}. Assume for simplicity there is a single cycle of length $\ell\geq3$ and all the other indices have $p_i=1$. Indeed this is the worst possible case (i.e. giving the largest contribution). Since all the moments of $\xi$ are bounded, we only have to match the cardinality of the sums appearing with the factor $m^{-r}$ in front of all the expression. If $\ell$ is odd the cycle index must be paired with a $m$ index, if $\ell$ is even only the indices with $p_i=1$ (resp. $q_i=1$) must be cross-paired with a $m$ index.  Thus if $\ell$ is odd, the cardinality of the sum in \eqref{eq:BaBa}-\eqref{eq:BcBc} is bounded by $m^{r-\ell+1}$ and since we have $m^{-r}$ at the denominator this contribution is of order $1/m^{\ell-1}$. Similarly if $\ell$ is even, the cardinality of the sum in \eqref{eq:BaBa}-\eqref{eq:BcBc} is bounded by $m^{r-\ell+2}$, thus the contribution is of order $1/m^{\ell-2}$. 
\end{proof}


\section{Local Limit Theorems}\label{app:LLT}

Here we prove two versions of the local central limit theorem for independent random variables which are used in the main text.

We recall the definitions
\be\label{eq:def-Gauss}
\phi_{\mu,\s}(x)\coloneqq \frac{e^{-\frac{(x-\mu)^2}{2\s}}}{\sqrt{2\pi\s}}\,,\qquad
\phi\coloneqq \phi_{0,1}\,,
\ee
and the Hermite polynomials
\be\label{eq:hermite}
H_m(x)\coloneqq \phi(x)^{-1}\int e^{-i\l x-\frac{\l^2}{2}}(i\l)^m\frac{d\l}{\sqrt{2\pi}}\,,\quad m\in\N\,. 
\ee

We warn the reader that in this section the notation with the $\widehat \cdot$ has a different meaning with respect to the rest of the paper. Indeed our proofs are based on Fourier transform and we write the characteristic function of the probability density $\nu$ as
\be\label{eq:nu-fNFourier}
\widehat\nu(\l)\coloneqq \int dx e^{-i\l x} \nu(x)\,.
\ee

The following statement is essentially contained in \cite[XVI.2 and XVI.6]{feller}, but we report its proof below anyway. 

\begin{theorem}\label{lemma:local-limit-teorem}
Let $X_1,\ldots,X_N$ be independent r.vs with continuous densities $\nu_1,\ldots,\nu_N$. Assume that the fourth moment of each $X_n$, $n\in[N]$, is uniformly bounded for all $N\geq1$. Denote
\be\label{eq:cumulants}
P_{n,k}\coloneqq \frac{d^k}{dy^k}\log E[e^{yX_n}]\Big|_{y=0}\,,\qquad P_{k}\coloneqq \sum_{n\in[N]}P_{n,k}\,,\qquad k=1,\ldots,3\,.
\ee
Let $f_N$ be the density associated to $\sum_{n=1}^NX_n$.
Then
\be\label{eq:LLT}
\sup_{x\in\R}\left|f_N(x)-\phi_{P_1, P_2}\left(x\right)\left(1+\frac{1}{3!}\frac{P_3}{(P_2)^{\frac 32}}H_3\left(\frac{x-P_1}{\sqrt {P_2}}\right)\right)\right|\lesssim e^{-cP_2}+e^{-cN} + \frac{N}{P_2^{\frac{5}{2}}}\,.
\ee
Moreover, $f_N$ depends continuously on $P_1,P_2,P_3$.
\end{theorem}
\begin{proof}

We write
\be\label{eq:nu-fNFourier}
\widehat\nu_n(\l)\coloneqq \int dx e^{-i\l x} \nu_n(x)\,, \qquad f_N(x)=\frac{1}{\sqrt{2\pi}}\int d \l e^{-i\l x} \prod_{n=1}^N \widehat\nu_n(\l)\,.
\ee
From this expression we see that $f_N$ depends continuously on $P_1,P_2,P_3$, as each factor $\widehat\nu_n$ depends continuously on $P_{n,1},P_{n,2},P_{n,3}$.

In addition, our assumptions on the densities ensure that $\widehat\nu_n(\l)\to0$ as $|\l|\to\infty$ (by the Riemann-Lebesgue Lemma) and also
$$
\max_{n\in[N]}\int d\l |\widehat\nu_n(\l)|=:D<\infty\qquad \mbox{uniformly in $N$.} 
$$
By \cite[XV.1 Lemma 4]{feller} for any $n\in[N]$ it is $|\widehat\nu_n(\l)|<1$ for $\l\neq0$. Therefore for all $\r>0$ there is $a>0$ such that for any $n\in[N]$ $|\widehat\nu_n(\l)|<e^{-a}$ for all $|\l|>\r$ (that is, we can choose $\r$ and $a$ independently of $n\in[N]$). 
Therefore we can bound
\be\label{eq:LLTvalorilarge}
\int_{|\l|\geq \r} d \l e^{-i\l x} \prod_{n=1}^N \widehat\nu_n(\l)\leq e^{-a(N-1)} \int d\l |\widehat\nu_N(\l)|\leq 3De^{-a(\r)N}\,. 
\ee

So for an arbitrary $\r>0$ we split 
\bea
f_N(x)&=&\frac{1}{\sqrt{2\pi}}\int_{|\l|<\r} d \l e^{-i\l x} \prod_{n=1}^N \widehat\nu_n(\l)+\frac{1}{\sqrt{2\pi}}\int_{|\l|\geq \r} d \l e^{-i\l x} \prod_{n=1}^N \widehat\nu_n(\l)\label{eq:fNsplit}\\
&=& \frac{1}{\sqrt{2\pi}}\int_{|\l|\leq \r} d \l e^{-i\l x} \prod_{n=1}^N \widehat\nu_n(\l)+\OOO{e^{-a(\r)N}}\,.\label{eq:idchiave}
\eea

by (\ref{eq:LLTvalorilarge}). Thus, we need only to gain control locally around zero.
We do that by a Taylor expansion
\be\label{eq:taylor1}
\widehat\nu_n(\l)=1+\sum_{k=1}^3\frac{(i\l)^k}{k!}\mu_{n,k}+R'_{n}(\l)\,,
\ee
where
$$
\sup_{N\in\N}\max_{n\in[N]}R'_{n}(\l)\leq C|\l|^4\,.
$$
This implies that there exists $\r>0$ such that $\inf_{n\in[N]}|\nu_n(\l)|>0$ for all $|\l|<\r$ uniformly in $N$. This allows us to write $\nu_n(\l)=e^{\log \nu_n(\l)}$ for $|\l|<\r$. Let us Taylor expand the logarithm at exponent. For all $|\l|<\r$ we have by a direct computation
\bea
\log\widehat\nu_n(\l)&=&\sum_{p=1}^3(-1)^p\left(\sum_{k=1}^3\frac{(i\l)^k}{k!}P_{n,k}\right)^p+\tilde R''_{n}(\l)\nn\\
&=&i\l P_{n,1}-\frac{\l^2}{2}P_{n,2}+\frac{(i\l)^3}{3!}P_{n,3}+R''_{n}(\l)\,,
\eea
where again
$$
\sup_{N\in\N}\max_{n\in[N]}\tilde R''_{n}(\l), R''_{n}(\l)\leq C|\l|^4\,.
$$
Hence, for all $n\in[N]$ and $|\l|<\r$
\bea
\prod_{n\in[N]}\widehat \nu_n(\l)&=&\exp\left(\sum_{n\in[N]}\log\widehat\nu_n(\l)\right)\nn\\
&=&\exp\left(i\l P_{1}-\frac{\l^2}{2} P_{2}\right)\left(1+\frac{(i\l)^3}{3!}P_{3}+R_m'''(\l)\right)\label{eq:THEexpansion}
\eea
with
$$
\sup_{N\in\N}\max_{n\in[N]}R'''_{n}(\l)\leq C|\l|^4\,.
$$

Therefore
\bea
\int_{|\l|< \r} d \l e^{-i\l x} \prod_{n=1}^N \widehat\nu_n(\l)&=&\int_{|\l|< \r} d \l e^{-i\l x}\exp\left(i\l P_{1}-\frac{\l^2}{2}P_{2}\right)\left(1+\frac{(i\l)^3}{3!}P_{3}\right)\nn\\\label{eq:firstINt}\\
&+&\int_{|\l|< \r} d \l e^{i\l x}\exp\left(i\l P_{1}-\frac{\l^2}{2}P_{2}\right)R'''(\l)\,.\label{eq:secondINt}
\eea
First we bound \eqref{eq:secondINt} uniformly in $x$. We have
\bea
|\eqref{eq:secondINt}|&\leq& \int_{|\l|<\r} d \l e^{-\frac{\l^2}{2}P_{2}}|R'''(\l)| \leq N \int_{|\l|\leq \r} d \l e^{-\frac{\l^2}{2}P_{2}}|\l|^{4}\nn\\
&=&\frac{N}{P_2^\frac52}\int_{|\l|< \r\sqrt{P_2}} d \l e^{-\frac{\l^2}{2}}|\l|^{4}\leq \frac{CN}{P_2^{\frac52}}\label{eq:boundIIINT}
\eea
for a suitable absolute constant $C>0$.

Now, we look at \eqref{eq:firstINt}. First we do (again) the change of variable $\l\mapsto\sqrt{P_2}\l$ and get
\be
\eqref{eq:firstINt}=\frac{1}{\sqrt {P_2}}\int_{|\l|< \r\sqrt{P_2}} d \l e^{-i\l \frac{x}{\sqrt {P_2}}}\exp\left(i\l\frac{P_{1}}{\sqrt{P_2}}-\frac{\l^2}{2}\right)\left(1+\frac{(i\l)^3}{3!}P_{3}\right)\,.\nn
\ee
We can extend this integral to $\R$ with a small error. Indeed
\bea
&&\left|\int_{|\l|> \r\sqrt{P_2}} d \l e^{-i\l \frac{x}{\sqrt {P_2}}}\exp\left(i\l\frac{P_{1}}{\sqrt{P_2}}-\frac{\l^2}{2}\right)\left(1+\frac{(i\l)^3}{3!}P_{3}\right)\right|\nn\\
&\leq& e^{-\frac{\r^2P_{2}}{4}}\int d \l e^{-\frac{\l^2}{4}}\left(1+\frac{(i\l)^3}{3!}P_{3}\right)\leq Ce^{-\frac{N}{4} P_{2}}\label{eq:LLTrem1}\,.
\eea
Moreover,
\bea
&&\frac{1}{\sqrt {2\pi P_2}}\int d \l e^{-i\l \frac{x}{\sqrt {P_2}}}\exp\left(i\l\frac{P_{1}}{\sqrt{P_2}}-\frac{\l^2}{2}\right)\left(1+\frac{(i\l)^3}{3!}P_{3}\right)\,\nn\\
&=&\phi_{P_1, P_2}\left(x\right)\left(1+\frac{1}{3!}\frac{P_3}{(P_2)^{\frac 32}}H_3\left(\frac{x-P_1}{\sqrt {P_2}}\right)\right)\label{eq:Hermite-term}
\eea
by the definition (\ref{eq:hermite}). 
Combining \eqref{eq:firstINt}, \eqref{eq:secondINt}, \eqref{eq:boundIIINT}, \eqref{eq:LLTrem1}, \eqref{eq:Hermite-term}, we can evaluate the first summand on the r.h.s. of \eqref{eq:idchiave} and therefore obtain (\ref{eq:LLT}). 
\end{proof}

The following simple observation is crucial in the paper.  

\begin{corollary}\label{cor;LLT}
Let $L>0$ and define the open set $A_L\coloneqq \{x\in\R\,:\,|x-P_1|\leq L\sqrt{P_2}\}$.
Under the same assumptions and notations of Theorem \ref{lemma:local-limit-teorem}, we have that there is a continuous function $\chi$ such that
\be\label{eq:cor;LLT1}
\sup_{x\in \R}|\chi(x)|\leq C\parav{\frac{N}{P_2}+e^{-cP_2}}\,
\ee
and 
\be\label{eq:cor;LLT}
f_N(x)=\phi_{P_1, P_2}\left(x\right)\left(1+\frac{1}{3!}\frac{P_3}{(P_2)^{\frac 32}}H_3\left(\frac{x-P_1}{\sqrt {P_2}}\right)+\chi(x)\right)\qquad\forall x\in A_L\,.
\ee
Moreover, $\chi$ depends continuously on $P_1,P_2,P_3$.
\end{corollary}

\begin{proof}
It is easy to see that
\be
\inf_{x\in A_L} \phi_{P_1, P_2}\left(x\right)\left(1+\frac{1}{3!}\frac{P_3}{(P_2)^{\frac 32}}H_3\left(\frac{x-P_1}{\sqrt {P_2}}\right)\right)= \phi_{0, 1}\left(L\right)\left(1+\frac{1}{3!}\frac{P_3}{(P_2)^{\frac 32}}H_3\left(L\right)\right)=:c(L)\,.
\ee
Since $c(L)>0$ we can define
\be
\chi(x)\coloneqq c(L)^{-1}\parav{f_N(x)-\phi_{P_1, P_2}\left(x\right)\left(1+\frac{1}{3!}\frac{P_3}{(P_2)^{\frac 32}}H_3\left(\frac{x-P_1}{\sqrt {P_2}}\right)\right)}\,. 
\ee
The continuity of $\chi$ follows from the continuity of $f_N$. In addition, (\ref{eq:cor;LLT}) and (\ref{eq:cor;LLT1}) follow immediately from (\ref{eq:LLT}). 

The fact that $\chi$ depends continuously on $P_1,P_2,P_3$ follows from the analogue property of $f_N$.
\end{proof}

If the initial densities are close to be Gaussian, we obtain a finer local result, more in the spirit of the Richter Theorem \cite{bobkov}.

\begin{theorem}\label{prop:LLTPII}
Let $X_1,\ldots,X_N$ be independent r.vs with analytic densities $\nu_1,\ldots,\nu_N$ with mean values $x_1,\ldots, x_n$ and variances $v_1,\ldots, v_n$. Let $f_N$ be the density associated to $\sum_{n=1}^NX_n$. 
Let $\{\mu_n\}_{n\in[N]}\in\R$ and $\{\s_n\}_{n\in[N]}\in\R^+$ and $\e>0$ and assume that for all $n=1\ldots, N$ we have
\be\label{eq:HYP-FT}
\hat\nu_n(\l)=\hat\phi_{\mu_n,\s_n}(\l)+\e\upsilon_n(\l)
\ee
where $\upsilon_1\ldots,\upsilon_n$ are analytic functions bounded by 1.
Then for all $a\in(0,1)$ we have
\be\label{eq:LLT2}
\sup_{x\in\R}\left|f_N(x)-\phi_{P_1, P_2}(x)\right|\lesssim  \frac{\e}{N^a}\,,
\ee
where
\be
P_1=\sum_{n\in[N]}x_n\,,\qquad P_2\coloneqq \sum_{n\in[N]}v_n\,. 
\ee
\end{theorem}

\begin{proof}
By the same proof as in Theorem \ref{lemma:local-limit-teorem} we arrive to
\bea
f_N(x)&=&\frac{1}{\sqrt{2\pi}}\int_{|\l|<\r} d \l e^{-i\l x} \prod_{n=1}^N \widehat\nu_n(\l)+\frac{1}{\sqrt{2\pi}}\int_{|\l|\geq \r} d \l e^{-i\l x} \prod_{n=1}^N \widehat\nu_n(\l)\label{eq:fNsplit2}\\
&=& \frac{1}{\sqrt{2\pi}}\int_{|\l|\leq \r} d \l e^{-i\l x} \prod_{n=1}^N \widehat\nu_n(\l)+\OOO{e^{-a(\r)N}}\,.\label{eq:idchiave2}
\eea
Moreover in the analytic case we can exactly characterise $a(\r)=C\r^2$ for some constant $C>0$ (see e.g. \cite[Corollary 2]{bobkov}). This allows us to consider any $\r>0$ such that $\r^2N\to\infty$ as $N\to\infty$. 

Next note that for any $n$ the function $\frac{\upsilon_n(\l)}{\widehat\phi_{\mu_n,\s_n}(\l)}$ is analytic in any neighbourhood of the origin and vanishes in $\l=0$. We expand it as follows
\be
\frac{\upsilon_n(\l)}{\widehat\phi_{\mu_n,\s_n}(\l)}=\upsilon_n'(0)\l+(\upsilon_n''(0)+2i\mu_n\upsilon_n'(0))\frac{\l^2}{2}+R_n(\l)\,,
\ee
where $|R_n(\l)|\leq C|\l|^3$. Thus
\bea
\log\widehat\nu_n(\l)&=&\log\widehat\phi_{\mu_n,\s_n}(\l)+\log\parav{1+\e\frac{\upsilon_n(\l)}{\widehat\phi_{\mu_n,\s_n}(\l)}}\nn\\
&=&\log\widehat\phi_{\mu_n,\s_n}(\l)+\e\l\upsilon_n(0)+\frac{\e\l^2}{2}\parav{\upsilon_n''(0)+2i\mu_n\upsilon_n'(0)-(\upsilon_n')^2}+\e R'_n(\l)\,,
\eea
where $R_n'$ comes out from the Taylor expansion of the $\log$ about $\l=0$ and we have $|R'_n(\l)|\leq C|\l|^3$.
Hence
\be
\widehat \nu_n(\l)=\phi_{\mu_n,\s_n}(\l)e^{\e\l\upsilon_n(0)+\frac{\e\l^2}{2}\parav{\upsilon_n''(0)+2i\mu_n\upsilon_n'(0)-(\upsilon_n')^2}+\e R'_n(\l)}\,.
\ee 
A direct computation gives
\be
\phi_{\mu_n,\s_n}(\l)e^{\e\l\upsilon_n(0)+\frac{\e\l^2}{2}\parav{\upsilon_n''(0)+2i\mu_n\upsilon_n'(0)-(\upsilon_n')^2}}=\widehat\phi_{x_n,v_n}(\l)\,.
\ee
Therefore there is a function $R''_N(\l)$, analytic uniformly in $N$ with $|R''_N(\l)|\leq C|\l|^3$, such that
\bea
\prod_{n=1}^N \widehat\nu_n(\l)&= &\prod_{n=1}^N \widehat\phi_{x_n,v_n}(\l)e^{\e\sum_{n=1}^N R'_n(\l)}\nn\\
&=&\widehat\phi_{P_1,P_2}(\l)\parav{1+\e N R''_N(\l)}
\eea
(this part is also similar to the proof of Theorem \ref{lemma:local-limit-teorem}). We write
\be\label{eq:rk}
R''_N(\l)=\sum_{k\geq 3} r_k\l^k
\ee
where the $r_k$ decay at least exponentially fast in $k$. 

Now we have
\bea
\frac{1}{\sqrt{2\pi}}\int_{|\l|\leq \r} d \l e^{-i\l x} \prod_{n=1}^N \widehat\nu_n(\l)&=&\frac{1}{\sqrt{2\pi}}\int_{|\l|\leq \r} d \l e^{-i\l x} \widehat\phi_{P_1,P_2}(\l)\label{eq:termMain}\\
&+&\frac{\e N}{\sqrt{2\pi}}\int_{|\l|\leq \r} d \l e^{-i\l x} \widehat\phi_{P_1,P_2}(\l) R'_n(\l)\,. \nn\\\label{eq:termR}
\eea
We bound
\be\label{eq:lazera}
\sup_{x\in\R}\absv{\eqref{eq:termR}}\leq \frac{\e N}{\sqrt{2\pi}}\int_{|\l|\leq \r} d \l |R'_n(\l)|\lesssim \e N\r^4\,. 
\ee

Similarly we have 
\bea
\eqref{eq:termMain}&=&\frac{1}{\sqrt{2\pi}}\int d \l e^{-i\l x} \widehat\phi_{P_1,P_2}(\l)\nn\\
&+&\frac{1}{\sqrt{2\pi}}\int_{|\l|> \r} d \l e^{-i\l x} \widehat\phi_{P_1,P_2}(\l)\nn\\
&=&\phi_{P_1,P_2}(x)+\frac{1}{\sqrt{2\pi}}\int_{|\l|> \r} d \l e^{-i\l x} \widehat\phi_{P_1,P_2}(\l)\nn
\eea
with
\be
\absv{\frac{1}{\sqrt{2\pi}}\int_{|\l|> \r} d \l e^{-i\l x} \widehat\phi_{P_1,P_2}(\l)\nn\\}\leq e^{-c\r^2P_2}\,. 
\ee
Therefore we have for all $\r>0$
\be
f_N(x)=\phi_{P_1,P_2}(x)+O(\e N\r^4)+O(e^{-c\r^2\min(N, P_2)})\,.
\ee
Taking $\r=N^{-a}$ for $1/4<a<1/2$ we have the assertion. 
\end{proof}

We conclude with the following result for analytic densities. 

\begin{Lem}\label{lemma:unmomento}
Assume $\nu$ is a probability density with $\nu(s)\leq e^{-c_1|s|}$ for $s\in\R$, where $c_1>0$ is an absolute constant, and 
\be
\absv{\int ds(\nu(s)-\phi_{\mu,\s}(s))e^{i\l s}}\leq \e\,.
\ee
Then for all $k\in\N$
\be
\absv{\int ds(\nu(s)-\phi_{\mu,\s}(s))s^k}\leq \frac{k!}{c_1^k}\e\,.
\ee
\end{Lem}
\begin{proof}
Let us set $f\coloneqq \nu-\phi_{\mu,\s}$. Then we have that $|f(s)|\leq e^{-c_1|s|}$ for some constants $c_1>0$. Therefore
\be
\hat f(\l)\coloneqq \int ds f(s)e^{i\l s}
\ee 
is bounded by $\e$ and analytic on $\R$ and it extends to an analytic function in a strip in the complex plane $\Sigma_\d\coloneqq \{z
\in\mathbb C,\quad |\Im(z)|\leq \d\}$ where  $0<\d<c_1$ (see e.g. \cite[Theorem 3.1, Chapter 4]{stein}). 
We have by the Cauchy formula 
\be
\hat f^{(k)}(z)=\frac{k!}{2\pi i}\int_\g\frac{f(w)}{(w-z)^{k+1}}dw\,
\ee
for any closed contour $\g\in\Sigma_\d$ oriented counterclockwise. We choose $\g\coloneqq \{w\in\C\,:\,|w|=c_1\}$ and we set $z=0$ in the formula above. We bound
\be
\absv{\hat f^{(k)}(0)}=\frac{k!}{2\pi i}\int_\g\frac{|\hat f(w)|}{|w|^{k+1}}dw\leq \frac{k!}{c_1^k}\e\,. \,
\ee
It follows that
\be
\absv{\int ds(\nu(s)-\phi_{\mu,\s}(s))s^k}=|\hat f(0)^{(k)}|\leq \frac{k!}{c_1^k}\e\,.
\ee
\end{proof}


\section{Tail estimates}\label{app:tail}
Here we prove the following estimate.
\begin{Lem}\label{prop:tail}
Let $p\geq2$ and $X_1,\ldots,X_N$ be independent symmetric sub-Gaussian r.vs. with unitary variance. It is for any real $x_1,\ldots, x_N$
\be\label{eq:tail}
P\left(\left|\sum_{j\in[N]}x_jX^p_j\right|\geq \l\right)\lesssim
\begin{cases}
2e^{-\frac{\l^2}{8(2K)^p\|x\|_2^2}}&\l\leq 4(2K)^p \|x\|_2^2 N^{\frac{2-p}{2(p-1)}} \\
2e^{-\frac18\parav{\frac{\l N}{4(2K)^p \|x\|_2^2}}^{\frac 2p}} &\l>4(2K)^p\|x\|_2^2 N^{\frac{2-p}{2(p-1)}}\,,
\end{cases}
\ee
where $K$ is a constant depending only on the distribution of $X_1,\ldots,X_N$. 
\end{Lem}
\begin{proof}
Let us set for brevity $Y_j\coloneqq x_jX^p_j$. 
Note that by assumption there is a constant $K>0$ such that
$$
\forall i\in[n]\qquad E\squarev{\exp\parav{\frac{|Y_i|^{\frac2p}}{Kx_i^{\frac2p}}}}<2\,. 
$$
Let $K_j:=Kx_j^{\frac2p}$, $\a\coloneqq \frac{p}{2(p-1)}\in[1/2,1]$ and
\be
\dot Y_j\coloneqq Y_j1_{\{|Y_j|\leq N^\a/(2K_j)^{\frac{2}{p-2}}\}}\,. 
\ee
Clearly
\be\label{eq:chiaro}
\max_{i\in[N]}E[e^{\frac{|\dot Y_i|^{\frac2p}}{K_i}}]<2\,. 
\ee
Then
\bea
P\left(\left|\sum_{j\in[N]}Y_j\right|\geq \l\right)&\leq&P\left(\left|\sum_{j\in[N]}Y_j\right|\geq \l,|Y_j|\leq N^\a/(2K_j)^{\frac{2}{p-2}}\quad\forall j\in[N]\right)\nn\\
&+&P\left(|Y_j|> N^\a/(2K_j)^{\frac{2}{p-2}}\quad\forall j\in[N]\right)\nn\\
&\leq&P\left(\left|\sum_{j\in[N]}\dot Y_j\right|\geq \l\right)+Ce^{-cN^{2\a/p}}\,. 
\eea
Let us now consider $0<a<N^{-\a(1-2/p)}=N^{\a-1}$. We note that
\be\label{eq:sotto}
a|\dot Y_j|\leq \frac{|\dot Y_j|^{\frac2p}}{2K_j}
\ee
We compute for every $j\in[N]$
\bea
E[e^{a\dot Y_j} ]&\leq&1+\sum_{n\geq2}\frac{a^n}{n!}E[|\dot Y_j|^n]\nn\\
&=&1+a^2\sum_{n\geq0}\frac{a^n}{(n+2)!}E[|\dot Y_j|^2|\dot Y_j|^n]\nn\\
&\leq&1+(2K)^pa^2\sum_{n\geq0}\frac{a^n}{(n+2)!}E[e^{\frac{|\dot Y_j|^{\frac{2}{p}}}{2K}}|\dot Y_j|^n]\nn\\
&\leq&1+(2K_j)^pa^2\sum_{n\geq0}\frac{a^n}{n!}E[e^{\frac{|\dot Y_j|^{\frac{2}{p}}}{2K_j}}|\dot Y_j|^n]\nn\\
&\leq&1+(2K_j)^pa^2E[e^{\frac{|\dot Y_j|^{\frac{2}{p}}}{K_j}}]\nn\\
&\leq&1+2(2K_j)^pa^2\leq e^{2(2K_j)^pa^2}\,. 
\eea
In the first inequality above we used that $x^2\leq e^{x^{\frac2p}}$, in the third one we summed the exponential series and used (\ref{eq:sotto}), we used \eqref{eq:chiaro} in the penultimate bound. 
Therefore
\be
E[e^{a\sum_{j=1}^N\dot Y_j} ]\leq e^{2^{1+p}a^2\sum_{j=1}^NK_j^p}=e^{2a^2(2K)^p\|x\|^2_2}\,.
\ee

By the Markov inequality
\be
P\left(\left|\sum_{j\in[N]}\dot Y_j\right|\geq \l\right)\leq
\begin{cases}
2e^{-\frac{\l^2}{8(2K)^p\|x\|_2^2}}&\l\leq 4(2K)^p \frac{\|x\|_2^2}{N} N^{\a} \\
2e^{-\frac18\parav{\frac{\l N}{4(2K)^p \|x\|_2^2}}^{\frac 2p}} &\l>4(2K)^p\frac{\|x\|_2^2}{N} N^{\a}\,.
\end{cases}
\ee
\end{proof}

We give also the following lemma as an easy consequence. 

\begin{Lem}\label{lemma:facilissimo}
Let $X_1,\ldots,X_N$ be independent sub-Gaussian r.vs. Let $p\in \N$, $t>0$. It holds for all $t\gtrsim \sum_{i\in[N]}\frac{E[|X_i|^p]}{N}$
\be\label{eq:P1}
P\parav{\sum_{i\in[N]}\frac{|X_i|}{N}\geq t}\leq e^{-ct^2N}\,. 
\ee
and for $p\geq2$
\be\label{eq:Pp}
P\parav{\sum_{i\in[N]}\frac{|X_i|^p}{N}\geq t}\leq e^{-c(tN)^{\frac2p}}\,. 
\ee
\end{Lem}
\begin{proof}
We have
\be\label{eq:Pp2}
P\parav{\sum_{i\in[N]}\frac{|X_i|^p}{N}\geq t}=P\parav{\sum_{i\in[N]}\frac{|X_i|^p-E[|X_i|^p]}{N}\geq t-\frac1N\sum_{i\in[N]}E[|X_i|^p]}\,. 
\ee
Now we set $Y_i\coloneqq |X_i|^p-E[|X_i|^p]$ and $\t\coloneqq t-\frac1N\sum_{i\in[N]}E[|X_i|^p]$. Considering $t>\frac1N\sum_{i\in[N]}E[|X_i|^p]$ we have
\be
P\parav{\sum_{i\in[N]}\frac{|X_i|^p}{N}\geq t}=P\parav{\sum_{i\in[N]}\frac{Y_i}{N}\geq \t}\,,
\ee
where $Y_1\ldots Y_N$ are independent centred random variables. It is easy to verify that there is a constant $K>0$ such that for any $p\geq1$
$$
\max_{i\in[N]}E[e^{\frac{|Y_i|^{\frac2p}}{K}}]<2\,. 
$$
For $p=1$ The Hoeffding inequality gives (\ref{eq:P1}) for all $\t>0$. 

For $p\geq2$ the same proof of Lemma \ref{prop:tail} gives (\ref{eq:Pp}) provided that $\t\gtrsim N^{-\frac{p-2}{2(p-1)}}$, that is certainly satisfied if $t>2\frac1N\sum_{i\in[N]}E[|X_i|^p]$ for $N$ large enough.
\end{proof}


\section{Proof of Lemma \ref{lemma:Pprob}}\label{App:prooflemma}

Here we prove the central technical lemma used in the last section. 
Recalling \eqref{eq:cum-t-3}, we write
explicitly (bearing in mind \eqref{eq:xandv}-\eqref{eq:rhovarsigma})
\be
P\ai =-A_{ai}^{-3}\sum_{j\neq i}A_{aj}^3\left(\rho^{(0)}\ja -3x^{(0)}\ja v^{(0)}\ja -(x^{(0)}\ja )^3\right)\,.
\ee

We recall the statement of Lemma \ref{lemma:Pprob} and Lemma \ref{lemma:simpel} for the reader convenience. 

\begin{Lem}\label{lemma:simpelII}
There is a constant $c=c(\d, \nu^{(0)})$ such that the following holds: 
\bea
\mathbb{P}\left(\left|\sum_{j\neq i} A_{aj}x^{(0)}\ja  \right|\geq \l\right)&\leq& Ce^{-c\l^2}\,,\quad \l\geq0\label{eq:claimXkII}\\
\mathbb{P}\left(\left|A^3_{ai}P\ai \right|\geq \l\right)&\leq& \begin{cases}
Ce^{-c\l^{2} m^2}&\quad\mbox{$\l\leq \frac1{m^{\frac34}}$}\\
Ce^{-c\l^{\frac23} m}&\quad\mbox{$\l\geq \frac1{m^{\frac34}}$}
\end{cases}\label{eq:claimP3II}
\,.
\eea
\end{Lem}

\begin{Lem}\label{lemma:Pprob1-appd}
We have
\be\label{eq:Pprob1-D}
\sum_{b \neq a}\frac{\hat x\zzero \bi P\bi }{(\hat v\zzero \bi )^{3}}= O_\P\parav{\frac{1}{N}}\,,\qquad\mbox{and}\qquad
\sum_{b\neq a}\frac{(\hat x\zzero \bi )^2P\bi }{(\hat v\zzero \bi )^{3}}= O_\P\parav{\frac{1}{\sqrt N}}
\ee
and
\be\label{eq:Pprob2}
\sum_{b \neq a}\frac{P\bi }{(\hat v\zzero \bi )^{2}}\leq O_\P\parav{\frac{1}{\sqrt N}}\,,\quad \sum_{b \neq a}\frac{P\bi }{(\hat v\zzero \bi )^{3}}\leq O_\P\parav{\frac{1}{N^{\frac32}}}\,.
\ee
\end{Lem}
\begin{proof}[Proof of Lemma \ref{lemma:Pprob}/\ref{lemma:Pprob1-appd}]
First we prove (\ref{eq:Pprob1-D}). We shorten
\be\label{eq:defw}
w\ai \coloneqq \sum_{j\neq i}A_{aj}^3\left(\rho^{(0)}\ja -3x^{(0)}\ja v^{(0)}\ja +(x^{(0)}\ja )^3\right)\,.
\ee
A rewriting of (\ref{eq:claimP3II}) gives
\be\label{eq:questaladevipropriovedere}
P\left(|w\ai |\geq \l\right)\leq  \begin{cases}
Ce^{-c\l^{2} m^2}&\quad\mbox{$\l\leq \frac1{m^{\frac34}}$}\\
Ce^{-c\l^{\frac23} m}&\quad\mbox{$\l\geq \frac1{m^{\frac34}}$}
\end{cases}\,. 
\ee
A direct computation gives
\be
\frac{P\ai }{(\hat v^{(0)}\ai )^{\frac32}}=-\sign A_{ai}\left(\sum_{j \neq i} A_{aj}^2 v\ja^{(0)}\right)^{-\frac32}w\ai \,
\ee
and for $k\in\{1,2\}$
\be
\frac{(\hat x^{(0)}\ai )^k}{(\hat v^{(0)}\ai )^{\frac32}}=(\sign A_{ai})^k|A_{ai}|^{3-k}\left(y_a-\sum_{j \neq i} A_{aj} x\ja^{(0)}\right)^k\left(\sum_{j \neq i} A_{aj}^2 v\ja^{(0)}\right)^{-\frac32}\,. 
\ee
Thus
\bea
&&\absv{\sum_{b\neq a}\frac{(\hat x^{(0)}\bi )^kP\bi }{(\hat v^{(0)}\bi )^{3}}}=\absv{\sum_{b \neq a}(\sign A_{bi})^{k+1}|A_{bi}|^{3-k}\left(\sum_{j \neq i} A_{bj}^2 v\jb^{(0)}\right)^{-3}\left(y_b-\sum_{j \neq i} A_{bj} x\jb^{(0)}\right)^kw\bi }\nn\\
&\leq&N^{\frac p2}\max_{b \neq a}\left[ \left(\sum_{j \neq i} A_{bj}^2 v\jb^{(0)}\right)^{-3}\left|y_b-\sum_{j \neq i} A_{bj} x\jb^{(0)}\right|^k|w\bi |\right]  \sum_{b \neq a}\frac{|A_{bi}|^{3-k}}{N^{\frac p2}}\,
\eea
for some $p>0$ to be determined later. Therefore
\bea
&&P\left(\absv{\sum_{b \neq a}\frac{(\hat x^{(0)}\bi )^kP\bi }{(\hat v^{(0)}\bi )^{3}}}\geq \l\right)\nn\\
&\leq& P\left(N^{\frac p2}\max_{b \neq a}\left[ \left(\sum_{j \neq i} A_{bj}^2 v\jb^{(0)}\right)^{-3}\left|y_b-\sum_{j \neq i} A_{bj} x\jb^{(0)}\right|^k|w\bi |\right]  \sum_{b\neq a}\frac{|A_{bi}|^{3-k}}{N^{\frac p2}}\geq \l\right)\nn\\
&\leq& P\left(N^{\frac p2}\max_{b \neq a}\left[ \left(\sum_{j \neq i} A_{bj}^2 v\jb^{(0)}\right)^{-3}\left|y_b-\sum_{j \neq i} A_{aj} x\jb ^{(0)}\right|^k|w\bi |\right] \geq \l^\a\right)\label{eq:Pprobline1}\\
&+&P\left(\sum_{b \neq a}\frac{|A_{bi}|^{3-k}}{N^{\frac p2}}\geq \l^{1-\a}\right)\label{eq:Pprobline2}
\eea
for some $\a\in(0,1)$. 
The term in (\ref{eq:Pprobline2}) is easily bound by Lemma \ref{lemma:facilissimo}. We have
\be\label{eq:lambda1}
\eqref{eq:Pprobline2}\leq 
\begin{cases}
k=1&e^{-c\l^{1-\a} N^{1+\frac p2}}\,,\qquad \l\gtrsim N^{-\frac{p}{2(1-\a)}}\\
k=2&e^{-c\l^{2(1-\a)} N^{p}}\,,\qquad \l\gtrsim N^{-\frac{p-1}{2(1-\a)}}\,.
\end{cases}
\ee

Let us set now $V\ai \coloneqq \frac1m\sum_{j\neq i}v^{(0)}\ja $. Clearly $|V\ai |\leq \d^{-1}\max_{a\in[m], i\in[N]}v^{(0)}\ja $. Next, we use that if $P(B)\geq\frac12$ then $P(A)\leq 2P(A\cap B)+P(B^c)$.
Thus, we have that, for a given $a>0$ small enough, 
\bea
\eqref{eq:Pprobline1}&\leq& 2P\left(N^{\frac p2}\max_{b \neq a}\left[\left|y_b-\sum_{j \neq i} A_{aj} x\jb^{(0)}\right|^k|w\bi |\right] \geq (V\ai (1-a))^3\l^\a\right)\nn\\
&+&P\left(\left|\sum_{j\neq i} A^2_{aj}v^{(0)}\ja -V\ai \right|\geq aV\ai \right)\nn\\
&\leq& 2P\left(N^{\frac p2}\max_{b \neq a}\left[\left|y_b-\sum_{j \neq i} A_{bj} x\jb^{(0)}\right|^k|w\bi |\right] \geq (V\ai )^3(1-a^2)\l^\a \right)+e^{-(V\ai )^2a^2m}\,.\nn\\
\eea
Take now ${\l'}^\a\coloneqq (V\ai )^3(1-a)^3\l^\a$ and $q\geq0$. We have
\bea
&&P\left(N^{\frac p2}\max_{b\neq a}\left[\left|y_b-\sum_{j \neq i} A_{aj} x\jb^{(0)}\right|^k|w\bi |\right] \geq {\l'}^\a \right)\nn\\
&\leq&P\left(\max_{b\neq a}N^{-\frac q2}\left|y_b-\sum_{j \neq i} A_{bj} x\jb^{(0)}\right|^k \geq {\l'}^{\frac\a2}\right)\label{eq:lambda1/6-2}\\
&+&P\left(\max_{b\neq a}N^{\frac {p+q}2}|w\ai |\geq {\l'}^{\frac\a2}\right)\label{eq:lambda1/6-3}\,. 
\eea
Recall that we assumed $\max_a|y_a|=1$. We bound for $\l'\geq  N^{-{\frac q\a}}$
\be\label{eq:choosepq2}
\eqref{eq:lambda1/6-2}\leq\sum_{b\neq a} P\left(\left|y_b-\sum_{j \neq i} A_{bj} x\jb^{(0)}\right|^k \geq N^{\frac q2}{\l'}^{\frac\a2}\right) \leq Cm e^{-c({\l'}^\a N^q)^{\frac 1k}}\,,
\ee
where we used (\ref{eq:claimXkII}), and for ${\l'}\lesssim N^{\frac{2(p+q)-3}{2\a}}$
\be\label{eq:choosepq3}
\eqref{eq:lambda1/6-3}\leq \sum_{b\neq a}P\left(|w\bi |\geq N^{-\frac {p+q}2}{\l'}^{\frac\a2}\right) \leq Cm e^{-c{\l'}^{\frac\a3}N^{1-\frac{q+p}{3}}}
\ee
where we used (\ref{eq:questaladevipropriovedere}). 

Summarising, we have that: 

(A) for $k=1$, if
\be\label{eq:condLambda1}
\min\parav{N^{-\frac{p}{2(1-\a)}}, N^{-{\frac q\a}}}\lesssim\l\lesssim N^{\frac{2(p+q)-3}{2\a}}
\ee
then
\be\label{eq:stimaP1}
P\left(\absv{\sum_{b\neq a}\frac{\hat x^{(0)}\bi P\bi }{(\hat v^{(0)}\bi )^{3}}}\geq \l\right)\lesssim e^{-c\l^{1-\a} N^{1+\frac p2}}+e^{-c{\l}^\a N^q}+e^{-c{\l}^{\frac\a3}N^{1-\frac{q+p}{3}}}\,;
\ee 
(B) for $k=2$, if
\be\label{eq:condLambda2}
\min\parav{N^{-\frac{p-1}{2(1-\a)}}, N^{-{\frac q\a}}}\lesssim\l\lesssim N^{\frac{2(p+q)-3}{2\a}}
\ee
then
\be\label{eq:stimaP2}
P\left(\absv{\sum_{b\neq a}\frac{(\hat x^{(0)}\bi )^2P\bi }{(\hat v^{(0)}\bi )^{3}}}\geq \l\right)\lesssim e^{-c\l^{2(1-\a)} N^{p}}+e^{-c({\l'}^\a N^q)^{\frac 12}}+e^{-c{\l'}^{\frac\a3}N^{1-\frac{q+p}{3}}}\,
\ee 
(recall that $\l$ and $\l'$ are proportional).
For $k=1$ we choose $\a=1/2$ and $p=1,q=1/2$. For $k=2$ we take $\a=3/4$, $p=5/4$ and $q=3/8$. These choices give the two inequalities in (\ref{eq:Pprob1}).

Next we prove (\ref{eq:Pprob2}) following the same procedure. We write for $h=2,3$
\be
\frac{P\ai }{(\hat v^{(0)}\ia )^{h}}=\sign A_{ai}|A_{ai}|^{2h-3}\left(\sum_{k \neq i} A_{ak}^2 v_{\ka}^{(0)}\right)^{-h}w\ai \,,
\ee
hence
\bea
\sum_{b \neq a}\frac{P\bi }{(\hat v^{(0)}\bi )^{h}}&\leq&\max_{b \neq a}\left(\sum_{j \neq i} A_{bj}^2 v\bj^{(0)}\right)^{-h}\absv{w\bi }\sum_{b\neq a}|A_{bi}|^{2h-3}\nn\,. 
\eea
Splitting as before we get
\bea
P\left(\sum_{b\neq a}\frac{P\bi }{(\hat v^{(0)}\bi )^{h}}\geq\l\right)&\leq& P\left(\frac{1}{N^{\frac p2}}\sum_{b\neq a}|A_{bi}|^{2h-3}\geq\l^{1-\a}\right)\label{eq:splitting2h-1}\\
&+&2P\left(N^{\frac p2}\max_{b\neq a}|w\bi |\geq\l^\a\right)+Ce^{-cm}\label{eq:splitting2h-2}\,,
\eea
for some $\a\in(0,1)$ and $p>0$ to be determined later. 
By Lemma \ref{lemma:facilissimo}, we have
\be\label{eq:facile1}
\eqref{eq:splitting2h-1}\leq 
\begin{cases}
h=2&e^{-c\l^{2(1-\a)} N^{p}}\,,\qquad \l\gtrsim N^{-\frac{p-1}{2(1-\a)}}\\
h=3&e^{-c\l^{\frac23 (1-\a)} N^{\frac{p+3}{3}}}\,,\qquad \l\gtrsim N^{-\frac{p+1}{2(1-\a)}}
\end{cases}\,. 
\ee

Moreover by (\ref{eq:questaladevipropriovedere}) for $\l\lesssim N^{\frac{2p-3}{4\a}}$ we get
\be  
\eqref{eq:splitting2h-2}\leq m e^{-c\l^{2\a} N^{2-p}}\,.
\ee

Taking $p=3/2$, $\a=1/2$ for $h=2$ and $p=1/2$ and $\a=1/2$ for $h=3$ gives (\ref{eq:Pprob2}). 
\end{proof}



\begin{thebibliography}{00}

\bibitem{yau} A. Adhikari, C. Brennecke, P. von Soosten, H.-T. Yau. {\em Dynamical approach to the TAP equations for the Sherrington-Kirkpatrick model}. Journal of Statistical Physics, 183(35):1-27, 2021.

\bibitem{tuca} A. Auffinger, A. Jagannath {\em Thouless-Anderson-Palmer equations for generic p-spin glasses}, Ann. Probab. 47(4): 2230-2256 (2019).

\bibitem{bayati1} M. Bayati and A. Montanari, {\em The Dynamics of Message Passing on Dense Graphs, with Applications to Compressed Sensing}, IEEE Transactions on Information Theory, vol. 57, no. 2, pp. 764-785, 2011
%
\bibitem{bayati2} M. Bayati, M. Lelarge, A. Montanari, {\em Universality In Polytope Phase Transitions And Message Passing Algorithms}, The Annals of Applied Probability 2015, Vol. 25, No. 2, 753-822.
%
\bibitem{far} R. Berthier, A. Montanari, P-M Nguyen. {\em State evolution for approximate message passing with non-separable functions}. Information and Inference, 01 (2019).
%
\bibitem{bobkov} S. Bobkov, G. Chistyakov, F. G\"otze. {\em Richter's local limit theorem, its refinement, and related results.} Lithuanian Mathematical Journal 63.2 (2023): 138-160

\bibitem{bolt} E. Bolthausen {\em An Iterative Construction of Solutions of the TAP Equations for the Sherrington-Kirkpatrick Model} Comm. Math. Phys. Vol. 325, pag. 333-366, (2014).

\bibitem{celentano} M. Celentano, A. Montanari, Y. Wu, {\em The estimation error of general first order methods}, in Conference Computational Learning Theory (COLT) 2020.

\bibitem{chat} S. Chatterjee. {\em Spin glasses and Stein's method}. Probability Theory and Related Fields, 148(3-4):567-600, 2010.

\bibitem{chen2} W.-K. Chen, S. Tang. {\em On Convergence of the Cavity and Bolthausen's TAP Iterations to the Local Magnetization}. Communications in Mathematical Physics, 386:1209-1242, 2021.

\bibitem{chen} W-K. Chen, S. Tang, {\em On the TAP equations via the cavity approach in the generic mixed p-spin models} Comm. Math. Phys., Vol. 405, No. 87 (2024)

\bibitem{coja} A. Coja-Oghlan, W. Perkins {\em Belief propagation on replica symmetric random factor graph models} Annales de l'institut Henri Poincare D 5 (2), 211-249.

\bibitem{DMM1} D. L. Donoho, A. Maleki, A. Montanari {\em Message-passing algorithms for compressed sensing}, Proceedings of the National Academy of Sciences, (2009). 
%
\bibitem{DMM2} D. L. Donoho, A. Maleki, A. Montanari {\em Message passing algorithms for compressed sensing: I. motivation and construction}, 2010 IEEE information theory workshop on information theory (ITW 2010, Cairo).
%
\bibitem{DMM3} D. Donoho, I Johnstone, A Maleki, A Montanari, {\em Compressed sensing over $\ell_p$-balls: Minimax mean square error} 2011 IEEE International Symposium on Information Theory Proceedings. IEEE, 2011.
%
\bibitem{tantoarkg5} Donoho, David, and Andrea Montanari. {\em High dimensional robust m-estimation: Asymptotic variance via approximate message passing.} Probability Theory and Related Fields 166.3 (2016): 935-969.
%
\bibitem{feller} W. Feller. {\em An introduction to probability theory and its applications} Vol. 2, John Wiley \& Sons, (2008).
%
\bibitem{CS} S. Foucart , H. Rauhut {\em A Mathematical Introduction to Compressive Sensing}, Springer (2013).
%
\bibitem{janson}S. Janson. {\em Gaussian Hilbert Spaces} Cambridge University Press, (1997). 
%
\bibitem{tantoarkg4} Javanmard, Adel, and Andrea Montanari. {\em State evolution for general approximate message passing algorithms, with applications to spatial coupling.} Information and Inference: A Journal of the IMA 2.2 (2013): 115-144.
%
\bibitem{k} F. Koehler {\em Fast Convergence of Belief Propagation to Global Optima: Beyond Correlation Decay}, Advances in Neural Information Processing Systems, 2019.
%
\bibitem{malekiT} A. Maleki, {\em Approximate message passing algorithms for compressed sensing}. Diss. Stanford University, 2010.
%
\bibitem{mezard} M. Mezard, A. Montanari, {\em Information, physics, and computation}, Oxford University Press 2009.
%
\bibitem{mez} M. Mezard, {\em Mean-field message-passing equations in the Hopfield model and its generalizations}, Phys. Rev. E 95, 022117  (2017) 
%
\bibitem{tantoarkg6} A. Montanari, R. Venkataramanan. {\em Estimation of low-rank matrices via approximate message passing.} The Annals of Statistics 49.1 (2021): 321-345.
%
\bibitem{mossel} E. Mossel, J. Xu. {\em Local algorithms for block models with side information.} Proceedings of the 2016 ACM Conference on Innovations in Theoretical Computer Science. 2016.
%
\bibitem{perla} J. Pearl, {\em Probabilistic reasoning in intelligent systems: networks of plausible inference}, Morgan Kaufmann, San Francisco, 1988.
%
\bibitem{sly} A. Sly {\em Reconstruction of symmetric Potts models}, The Annals of Probability, 39, 1365-1406 (2011).
%
\bibitem{stein} E. Stein, R. Shakarchi. Complex analysis. Vol. 2. Princeton University Press, 2010.
%
\bibitem{talabook} M. Talagrand. {\em Mean field models for spin glasses. Volume I. Basic examples}, volume 54 of Ergebnisse der Mathematik und ihrer Grenzgebiete. 3. Folge. Springer-Verlag, Berlin, 2011.
%
\bibitem{tap} D. J. Thouless, P. W. Anderson, R. G. Palmer. {\em Solution of 'solvable model of a spin glass'}. Philosphical Magazine, 35(3):593-601, 1977.
%
\end{thebibliography}
\end{document}